\documentclass[titlepage]{amsart}
\usepackage{amsaddr}

\newtheorem{pro}{Proposition}[section]

\newtheorem{lem}[pro]{Lemma}
\newtheorem{cor}[pro]{Corollary}
\newtheorem{rk}[pro]{Remark}

\newtheorem{prop}[pro]{Proposition}
\newtheorem{thm}[pro]{Theorem}
\newtheorem{defn}[pro]{Definition}
\newtheorem{rem}[pro]{Remark}
\newtheorem{example}[pro]{Example}

\newcommand{\Ext}{\mathrm{Ext}}
\newcommand{\Hom}{\mathrm{Hom}}

\newcommand{\A}{\mathcal{A}}
\newcommand{\B}{\mathcal{B}}

\newcommand{\C}{\mathcal{C}}

\newcommand{\T}{\mathcal{T}}
\newcommand{\X}{\mathcal{X}}
\newcommand{\Y}{\mathcal{Y}}
\newcommand{\Z}{\mathcal{Z}}
\newcommand{\W}{\mathcal{W}}
\newcommand{\pd}{\mathrm{pd}}

\newcommand{\Proj}{\mathrm{Proj}}
\newcommand{\proj}{\mathrm{proj}}
\newcommand{\Inj}{\mathrm{Inj}}
\newcommand{\inj}{\mathrm{inj}}
\newcommand{\id}{\mathrm{id}}

\newcommand{\coresdim}{\mathrm{coresdim}}
\newcommand{\add}{\mathrm{add}}
\newcommand{\smd}{\mathrm{smd}}
\newcommand{\Add}{\mathrm{Add}}
\newcommand{\Prod}{\mathrm{Prod}}

\newcommand{\modu}{\mathrm{mod}}
\newcommand{\Mod}{\mathrm{Mod}}
\newcommand{\Ker}{\mathrm{Ker}}

\newcommand{\Coker}{\mathrm{CoKer}}

\global\long\def\Rep{\operatorname{Rep}}%
\global\long\def\lmcn{\mathrm{lmcn}}%
\global\long\def\lccn{\mathrm{lccn}}%
\global\long\def\rccn{\mathrm{rccn}}%
\global\long\def\ltccn{\mathrm{ltccn}}%
\global\long\def\rtccn{\mathrm{rtccn}}%
\global\long\def\ccn{\mathrm{ccn}}%

\global\long\def\tccn{\mathrm{tccn}}%
\global\long\def\Supp{\mathrm{Supp}}%

\usepackage{latexsym, amssymb, amscd} 
\usepackage{amsmath}
\usepackage[all]{xypic}
\usepackage{amsthm}
\usepackage{xargs}[2008/03/08] 

\newcommandx\suc[5][usedefault, addprefix=\global, 1=N, 2=M, 3=K, 4=, 5=]{#1\overset{#4}{\hookrightarrow}#2\overset{#5}{\twoheadrightarrow}#3}%

\newcommandx\p[2][usedefault,  addprefix=\global,  1=\mathcal{A},  2=\mathcal{B}]{\left(#1, #2\right)}
\newcommandx\pdr[2][usedefault,  addprefix=\global,  1=\mathcal{A},  2=M]{\mathrm{pd}{}_{#1}\left(#2\right)}
\newcommandx\idr[2][usedefault,  addprefix=\global,  1=\mathcal{B},  2=M]{\mathrm{id}{}_{#1}\left(#2\right)}
\newcommandx\Extx[4][usedefault,  addprefix=\global,  1=i,  2=\mathcal{C},  3=M,  4=X]{\mathrm{Ext}{}_{#2}^{#1}\left(#3, #4\right)}
\newcommandx\Injx[1][usedefault,  addprefix=\global,  1=R]{\operatorname{Inj}\left(#1\right)}
\newcommandx\Projx[1][usedefault,  addprefix=\global,  1=R]{\operatorname{Proj}\left(#1\right)}
\newcommandx\projx[1][usedefault,  addprefix=\global,  1=R]{\operatorname{proj}\left(#1\right)}
\newcommandx\smdx[1][usedefault,  addprefix=\global,  1=\mathcal{M}]{\operatorname{smd}\left(#1\right)}
\newcommandx\Kerx[1][usedefault,  addprefix=\global,  1=M]{\operatorname{Ker}\left(#1\right)}
\newcommandx\Homx[3][usedefault,  addprefix=\global,  1=\mathcal{C},  2=M,  3=N]{\mathrm{Hom}{}_{#1}(#2, #3)}
\newcommandx\End[2][usedefault,  addprefix=\global,  1=R,  2=M]{\mathrm{End}{}_{#1}(#2)}
\newcommandx\pdx[1][usedefault,  addprefix=\global,  1=M]{\operatorname{pd}\left(#1\right)}
\newcommandx\Cok[1][usedefault,  addprefix=\global,  1=M]{\operatorname{Coker}\left(#1\right)}
\newcommandx\Gen[1][usedefault,  addprefix=\global,  1=M]{\operatorname{Fac}_{1}\left(#1\right)}
\newcommandx\Genn[2][usedefault,  addprefix=\global,  1=M,  2=n]{\operatorname{Fac}_{#2}(#1)}
\newcommandx\Gennr[3][usedefault,  addprefix=\global,  1=\mathcal{T},  2=n,  3=\mathcal{X}]{\operatorname{Fac}_{#2}^{#3}(#1)}
\newcommandx\addx[1][usedefault,  addprefix=\global,  1=\mathcal{M}]{\operatorname{add}\left(#1\right)}
\newcommandx\Addx[1][usedefault,  addprefix=\global,  1=\mathcal{M}]{\operatorname{Add}\left(#1\right)}
\newcommandx\Modx[1][usedefault,  addprefix=\global,  1=R]{\operatorname{Mod}\left(#1\right)}
\newcommandx\modd[1][usedefault,  addprefix=\global,  1=R]{\operatorname{mod}\left(#1\right)}
\newcommandx\Pres[3][usedefault,  addprefix=\global,  1=T,  2=\mathcal{D},  3=n]{\operatorname{Pres}_{#2}^{#3}\left(#1\right)}
\newcommandx\Homk[4][usedefault,  addprefix=\global,  1=\mathcal{K}(R),  2=\sigma,  3=\omega,  4=1]{\mathrm{Hom}{}_{#1}(#2, #3[#4])}
\newcommandx\im[1][usedefault,  addprefix=\global,  1=f]{\mathrm{Im}{}(#1)}

\global\long\def\resdimr#1#2#3{\mathrm{resdim}{}_{#1}^{#3}\left(#2\right)}
\global\long\def\coresdimr#1#2#3{\mathrm{coresdim}{}_{#1}^{#3}\left(#2\right)}
\global\long\def\resdimx#1#2{\operatorname{resdim}_{#1}\left(#2\right)}
\global\long\def\coresdimx#1#2{\operatorname{coresdim}_{#1}\left(#2\right)}

\usepackage{tikz}
\usepackage{pgfplots}
\usetikzlibrary{arrows}
\usepackage{amssymb, amsmath}
\usetikzlibrary{positioning}
\usetikzlibrary{backgrounds}
\usepackage{enumitem}
\setenumerate{ label= (\alph*)}
\setenumerate[2]{ label= (\alph{enumi}\arabic*)} 

\usepackage[Lenny]{fncychap}
\usepackage{amscd}

\usepackage{etoolbox}

\makeatletter
\patchcmd{\@maketitle}
{\if@titlepage \newpage \else}
{\if@titlepage
 \vspace{\baselineskip}
 \else}
{}{}
\makeatother
\begin{document}

\title[Relative tilting theory II]{Relative tilting theory in abelian categories II:  $n$-$\mathcal{X}$-tilting theory}


\author{{A}lejandro Argud\'in-Monroy}
\address {Centro de Ciencias Matem\'aticas\\ Universidad Nacional Aut\'onoma de M\'exico, Morelia, Michoac\'an, MEXICO.}
\email{argudin@ciencias.unam.mx}

\author{Octavio Mendoza-Hern\'andez}
\address{Instituto de Matem\'aticas, Universidad Nacional Aut\'onoma de M\'exico,  M\'exico,  D.F. MEXICO.}
\email{omendoza@matem.unam.mx}

\subjclass[2010]{Primary: 18G20,  16E10; Secondary: 18E10,  18G25}
\keywords{ Relative Cotorsion pairs,  relative homological dimensions, relative tilting theory, Auslander-Buchweitz-Reiten approximation theory}
\thanks{
\indent {\em Funding:} This work was supported by the Project PAPIIT-Universidad Nacional Aut\'onoma de M\'exico IN100520. The first author was also supported by  a postdoctoral fellowship from Programa de Desarrollo de las Ciencias B\'asicas, Ministerio de educaci\'on y cultura,  Universidad de la Rep\'ublica,  Uruguay. He is currently supported with a postdoctoral fellowship from Programa de Becas Posdoctorales en la UNAM, Direcci\'on General de Asuntos del Personal Acad\'emico, Universidad Nacional Aut\'onoma de M\'exico.  }


\begin{abstract}
We introduce a relative tilting theory in abelian categories and show
that this work offers a unified framework of different previous notions
of tilting,  ranging from Auslander-Solberg relative tilting modules
on Artin algebras to infinitely generated tilting modules on arbitrary
rings. Furthermore,  we see that it presents a tool for developing
new tilting theories in categories that can be embedded nicely in
an abelian category. In particular,  we will show how the tilting theory
in exact categories built this way,  coincides with tilting objects
in extriangulated categories introduced recently. We will review 
Bazzoni\textquoteright s
tilting characterization,  the relative homological dimensions on the induced
tilting classes  and parametrise certain cotorsion-like pairs by using $n$-$\X$-tilting classes. As an application, we show how to construct relative tilting classes and cotorsion pairs in $\Rep(Q,\C)$ (the category of representations of a quiver $Q$ in an abelian category $\C$) from tilting classes in $\C,$ where $Q$ is finite-cone-shape.
\end{abstract}

\maketitle

\setcounter{tocdepth}{1}

\tableofcontents

\section{Introduction}

In the last 40 years,  tilting theory has been generalized in many
ways and contexts with different purposes.
Its roots can be traced back to the seminal work of P. Gabriel \cite{gabriel1972},  which showed a bijection between the indecomposable modules over a finite-dimensional algebra and the positive roots of a Lie group. After that,  J. Bernstein,  I. M. Gelfand and V. A. Ponomarev deepened this study with the aim of constructing all of the indecomposable modules over a finite-dimensional algebra \cite{bernstein1973}. Some time later,  M. Auslander,  M. I.  Platzeck and I. Reiten generalised these results constructing for the first time what we now know as a tilting object in the context of finitely generated modules over Artin algebras  \cite{auslander1979coxeter}.
 It was  S. Brenner and M. Butler who axiomatized and gave name to these objects in  \cite{brenner1980generalizations}. Subsequently,  a more general definition was
offered by D. Happel and C. M. Ringel in \cite{happel1982tilted}
with the goal of achieving a better understanding of tilting objects.
Few years later,  this definition would be extended from tilting objects
of projective dimension $\leq1$ to tilting objects of finite projective
dimension by Y. Miyashita in \cite{miyashita},  but still under the context of finitely generated modules. Later on,  the tilting
theory context would be extended from finitely generated modules over
Artin algebras to infinitely generated modules over arbitrary rings, 
this is the case of the work of L. Angeleri H\"ugel and F.
U. Coelho in \cite{Tiltinginfinitamentegenerado}. 

As can be appreciated, in the literature there are  a diverse family of different tilting definitions
with different properties and objectives. This family of tilting theories
can be bluntly divided in two subfamilies:  ``big'' tilting theories
and ``small'' tilting ones. 

The small tilting theories can be described as the ones defined using
only finite coproducts. Namely,  all the classical tilting theories,  which were
developed for finitely generated modules,  are generalized by the small tilting theories. Among them,  we can
mention the Brenner-Butler,  the Happel-Ringel,  and the Miyashita theories
referred above,  but also we can find more recent research works as
the tilting functors  by R. Mart\'inez and M.
Ortiz in \cite{martinez2014tilting}.

The big tilting theories are those ones that require arbitrary coproducts
on its constructions of tilting classes. These kind of theories started coming up when,  inter
alia,  the works of Brenner-Butler,  Happel-Ringel,  I. Assem \cite{assem1984torsion} 
and S. O. Smal{\o}   \cite{smalo1984torsion} were extended to
the setting of infinitely generated modules over arbitrary rings in
the works of R. R. Colby and K. R. Fuller \cite{colby1990tilting}, 
R. Colpi,  G. D'Este  and A. Tonolo \cite{Quasitiltingcounterequivalences}, 
R. Colpi,  A. Tonolo  and Jan Trlifaj \cite{partialcotilting},  R.
Colpi and J. Trlifaj \cite{Colpi-Trlifaj},  A. Tonolo,  J. Trlifaj, 
and L. Angeleri H\"ugel \cite{Tiltingpreenvelopes},  and L. Angeleri
H\"ugel and F. U. Coelho \cite{Tiltinginfinitamentegenerado}.
Recent works on big tilting theories are focused on abelian categories
with coproducts as can be seen in the works of L. Positselski
and J. {\v{S}}t'ov{\'\i}{\v{c}}ek \cite{positselskicorrespondence}, 
P. Nicol\'as,  M. Saor\'in  and A. Zvonareva \cite{nicolas2019silting}.

This manuscript is the last of two forthcoming papers and it is devoted to develop new tools for understanding the
tilting phenomenon. Namely,  we will be interested in studying the
relation of \emph{cotorsion-like pairs} in an abelian category,  with
a new tilting notion associated to a subcategory $\mathcal{X}\subseteq\mathcal{C}$, 
called $n$-$\mathcal{X}$-tilting. This kind of relations were studied
for the first time by M. Auslander and I. Reiten in \cite{auslandereiten, auslander1992homologically}.
One of their results is \emph{the Auslander-Reiten Correspondence}
\cite[Thm. 4.4]{auslander1992homologically},  which shows a correspondence
between tilting modules over an Artin algebra and covariantly finite
subcategories. It is worth mentioning that this theorem has been taken
 to different contexts by different authors. Some of them
are M. Auslander and {\O}. Solberg \cite[Thms, 3.2 and 3.24]{auslander1993relative2}, 
S. K. Mohamed \cite[Prop. 4.2]{mohamed2009relative}, 
L. Angeleri and O. Mendoza \cite[Thm. 3.2]{Hopel-Mendoza}, 
and B. Zhu and X. Zhuang in \cite[Thm. 2]{zhu2019tilting}. 

The paper is organized as follows. The cotorsion-like pairs we previously referred to  were presented
in \cite{parte1}. They are linked with a possible  generalization of the Auslander-Reiten
theory,  developed in \cite{auslandereiten},  and the Auslander-Buchweitz
approximation theory,  developed in \cite{Auslander-Buchweitz}.  In Section 2,  we will recall the main definitions 
and results of \cite{parte1}. In particular,  we will recall  notions related to the cotorsion pairs,  relative homological dimensions,  relative resolution dimensions,  closure properties and the class $\operatorname{Fac}_n ^{\X}(\T)$.

In Section 3,  we state and develop  our $n$-$\X$-tilting theory. The goal is to present a tilting theory relative to a class of objects $\X$ together with a set of tools that provides us information on the induced homological dimensions and approximation theory. In order that our results can be used in a wide variety of contexts,  we sought to provide a definition that on one hand encompasses different prior notions and on the second hand  can be specialized to big or small tilting classes according to our needs. In order to do that,  we define $n$-$\X$-tilting classes $\T$ in an abelian category $\C,$ see Definition \ref{def: X-tilting}, and say that an object $T\in\C$ is big  (small) $n$-$\X$-tilting if $\Add(T)$  ($\add(T)$) is an $n$-$\X$-tilting class. Let us describe briefly the most relevant results.  
In Theorem  \ref{thm: el par n-X-tilting}  we give some essential properties of the $n$-$\X$-tilting classes, among them it is shown that the pair $({}^{\bot}(\T^{\bot }), \T^{\bot })$ is $\X$-complete. On the other hand, Theorem \ref{prop: primera generalizacion} is the generalization of the ``Bazzoni's tilting characterization theorem" which was originally provided for tilting modules over a ring 
\cite[Thm. 3.11]{Bazzonintilting}. We also study the relationship between different relative homological dimensions of classes related with $n$-$\X$-tilting classes, as can be seen for example in Propositions \ref{prop: oct2}, \ref{prop: oct3} and \ref{prop: M ortogonal es preenvolvente esp en X}. We also have related the big and the small tilting classes. Indeed, in  Theorem \ref{thm: n-X-tilting sii n-X-tilting peque=0000F1o} it is shown that,  for a class of compact objects $\X$,  an object $T$ is big $n$-$\X$-tilting if and only if it is small $n$-$\X$-tilting.  One of our goals is to study the properties satisfied by the pair $({}^{\bot}(\T^{\bot }), \T^{\bot })$ for a $n$-$\X$-tilting class $\T.$ In order to do that, we introduce the notion of $n$-$\X$-tilting triple and characterize them in Theorem \ref{thm: dimensiones en un par de cotorsion tilting}. There are several consequences of the preceding theorem: (1) we give a bijective correspondence between equivalence classes of $\X$-complete hereditary cotorsion pairs (satisfying certain properties) and $n$-$\X$-tilting classes belonging to $\X$ (see Corollary \ref{corr-tcchp}); and (2) we get two versions of the Auslander-Reiten Correspondence in Corollary \ref{cor:  coro1 teo nuevo} (for big tilting) and Corollary \ref{cor:  coro2 teo nuevo} (for small tilting).

In Section 4,  we will show that the notion of $n$-$\X$-tilting generalizes a big variety of previous notions of tilting which appeared in different contexts. We will also see how our results help us to find equivalences between different tilting notions. The first example of this section are the $\infty$-tilting objects and pairs  which were defined by Leonid Positselski and Jan {\v{S}}t'ov{\'\i}{\v{c}}ek  in \cite{positselski2019tilting}.   The second one is related with the Miyashita $n$-tilting modules,  which can be seen as $n$-$\modd[R]$-tilting modules. In the third example of this section,  we will develop a theory of Miyashita $n$-tilting modules of type $FP_n, $ for left $n$-coherent rings. The fourth example of this section is devoted to study the tilting phenomena in the context of small exact categories. Namely,  for an small exact category $(\A, \mathcal{E}), $ with enough $\mathcal{E}$-projectives and $\mathcal{E}$-injectives,  we introduce the small 
$n$-tilting  and the  Auslander-Solberg $n$-tilting classes in $(\A, \mathcal{E}).$ We show that both of them are equivalent to the Zhu-Zhuang tilting theory for exact categories developed in \cite{zhu2019tilting}. Moreover,  we will explore a nice embedding of $\A$ into the functor category $\Mod(\mathcal{P}^{op}), $ given by Yoneda's functor,  where $\mathcal{P}$ is the set of all the $\mathcal{E}$-projective objects in $\A.$ We also show that the $n$-$\X$-tilting theory developed in the abelian category 
$\Mod(\mathcal{P}^{op})$ is strongly related with the small $n$-tilting classes in 
$(\A, \mathcal{E}).$ The fifth example is devoted to the 
S.K. Mohamed's relative tilting theory \cite{mohamed2009relative} and the
Auslander-Solberg tilting objects \cite{auslander1993relative2}.   It
is worth mentioning that Auslander-Solberg relative tilting theory has been studied by several authors in the context of Gorenstein homological algebra.
In particular, M. Pooyan and Y. Siamak recently published a paper on infinitely generated Gorenstein tilting modules  \cite{Pooyan-Siamak21}. We believe that our work will be a complementary tool
for this research line. In the sixth example,   we will study the tilting classes of functors developed by R. Mart\'inez and M. Ortiz  \cite{martinez2011tilting, martinez2013tilting, martinez2014tilting},  and characterize them in terms of $n$-$\X$-tilting theory. Finally,  in the last example,  we will study the relationship between silting,  quasitilting and $n$-$\X$-tilting modules, see Theorems \ref{thm: quasitilt sii 1-Gen-tiltin} and \ref{thm: silting vs 1-Gen-tilting}. One of the consequences of doing this is that we found enough conditions for a quasitilting finendo module to be silting (see Remark \ref{rem: silting}). 
\

In Section 5, we consider an abelian category $\C$ and a finite-cone-shape quiver $Q.$ It is presented two main results. The first one is Theorem \ref{thm:tilting repf} which tells us how to build a tilting class in the abelian subcategory $\Rep^{f}(Q,\C)\subseteq\Rep(Q,\C)$ and also in $\Rep(Q,\C)$ from
a tilting class in $\mathcal{C}.$ The second one is Theorem \ref{Rep-tilt-pair} that tells us how to construct hereditary complete cotorsion pairs in the category of representations from tilting classes in the abelian category $\C.$   Finally, some concrete examples are given where these theorems can be applied.

\section{Preliminaries }

In this section,  we introduce all the necessary notions and results to the development of the paper. For more details,  we recommend the reader to see in \cite{parte1}.

\subsection{Notation }

Throughout the paper,  we denote by $\mathcal{C}$ an abelian category. The symbol $\mathcal{M}\subseteq\mathcal{C}$  means that $\mathcal{M}$
is a class of objects of $\mathcal{C}$. In a similar way,  the symbol
$\p\subseteq\mathcal{C}^{2}$ will mean that $\mathcal{A}$ and $\mathcal{B}$
are classes of objects of $\mathcal{C}$. On the other hand,  $C\in\mathcal{C}$
will mean that $C$ is an object of $\mathcal{C}$. We will use the Grothendieck's notation \cite{Ab} to distinguish abelian
categories with further structure as ABk and their duals ABk*, for $k=3,4,5.$

For $n\geq0$,  we will consider the $n$-th Yoneda extensions  bifunctor
$\Extx[n][][-][-]: \mathcal{C}^{op}\times\mathcal{C}\rightarrow\mbox{Ab},$
 the long exact sequence induced by a short exact sequence
 \cite[Chap. VI,  Thm. 5.1]{mitchell} and the Shifting
Lemma  \cite[Lem. 2.2]{parte1}. If $\mathcal{C}$ is
AB4,  for any family $\{A_{i}\}_{i\in I}$ of objects in $\mathcal{C}, $ we will make use (without mention it) of
the natural isomorphism  \cite[Thm. 3.12]{argudin2019yoneda}
\[
\Psi_{n}: \Extx[n][][\bigoplus_{i\in I} A_{i}][B]\rightarrow\prod_{i\in I}\Extx[n][][A_{i}][B]\;\forall B\in\mathcal{C}.
\]

Let $\mathcal{X}\subseteq\mathcal{C}.$ For any integer $i\geq 0, $ we consider the right $i$-th orthogonal complement $\X^{\perp_i}: =\{C\in\C\;|\;\Ext^i_\C(-, C)|_\X=0\}$ and the total right orthogonal complement $\X^{\perp}: =\cap_{i\geq 1}\X^{\perp_i}$ of $\X.$  Dually,  we have the $i$-th and the total left orthogonal complements ${}^{\perp_i}\X$ and ${}^{\perp}\X$ of $\X, $ respectively.
In case we have some $\mathcal{Y}\subseteq\mathcal{C}$
such that $\mathcal{Y}\subseteq\mathcal{X}^{\bot}$ ($\mathcal{Y}\subseteq{}^{\bot}\mathcal{X}$), 
we say that $\mathcal{Y}$ is \textbf{$\mathcal{X}$-injective} (\textbf{$\mathcal{X}$-projective}). 
\

For a given $\mathcal{M}\subseteq\mathcal{C},$ we have that: $\smdx[\mathcal{M}]$ is the class of all the direct summands of objects in 
$\mathcal{M};$  $\mathcal{M}^{\oplus}$
($\mathcal{M}^{\oplus_{<\infty}}$) is the class of (finite) coproducts
of objects in $\mathcal{M};$ $\addx[\mathcal{M}]: =\smdx[\mathcal{M}^{\oplus_{<\infty}}]$
and $\Addx: =\smdx[\mathcal{M}^{\oplus}]$. Furthermore,  in case $\mathcal{M}$
consists of a single object $M, $ we set $M^{\oplus}: =\mathcal{M}^{\oplus}$, 
$M^{\oplus_{<\infty}}: =\mathcal{M}^{\oplus_{<\infty}}$,  $\smdx[M]: =\smdx[\mathcal{M}]$, 
$\Addx[M]: =\Addx[\mathcal{M}]$,  $\addx[M]: =\addx[\mathcal{M}]$,  $M^{\bot}: =\mathcal{M^{\bot}}, $
and $^{\bot}M: ={}^{\bot}\mathcal{M}$.

One important feature of this work is that we do not assume the existence
of enough projectives or enough injectives in the abelian category $\C.$ Instead we will be working
with the following notions appearing in \cite{Auslander-Buchweitz}. For $(\mathcal{X}, \omega)\subseteq\mathcal{C}^{2}, $
it is said that $\omega$ is a \textbf{relative cogenerator in $\mathcal{X}$}
if $\omega\subseteq\mathcal{X}$ and any $X\in\mathcal{X}$ admits
an exact sequence $\suc[X][W][X']\mbox{, }$ with $W\in\omega$ and
$X'\in\mathcal{X}$. The notion of \textbf{relative generator} is
defined dually.

\subsection{Cotorsion pairs,  approximations  and related notions}

Following \cite[Def. 3.1]{parte1},  we recall that for 
$\p\subseteq\mathcal{C}^{2}$ and $\mathcal{X}\subseteq\mathcal{C}, $ it is said
that $\p$ is a \textbf{left (right) cotorsion pair in $\mathcal{X}$}
 if $\mathcal{A}\cap\mathcal{X}={}^{\bot_{1}}\mathcal{B}\cap\mathcal{X}$
($\mathcal{B}\cap\mathcal{X}=\mathcal{A}^{\bot_{1}}\cap\mathcal{X}$).
Moreover,  $\p$ is a \textbf{cotorsion pair in $\mathcal{X}$} if
it is a left and right cotorsion pair in $\mathcal{X}.$  In case $\X=\C, $ we say that $(\A, \B)$ is a left (right) cotorsion pair if it is a left (right) cotorsion pair in $\C.$
\

Cotorsion pairs are known for their relation with approximations. Namely,  for a given
$\mathcal{Z}\subseteq\mathcal{C}$,   a morphism $f: Z\rightarrow M$
is called  \textbf{$\mathcal{Z}$-precover} if $Z\in\mathcal{Z}$
and $\Homx[][Z'][f]: \Hom_\C(Z', Z)\to \Hom_\C(Z', M)$ is an epimorphism $\forall Z'\in\mathcal{Z}$.
In case $f$ fits in an exact sequence $\suc[M'][Z][M][\, ][\, ]$, 
where $M'\in\mathcal{Z}^{\bot_{1}}$,   $f$ is called \textbf{special
$\mathcal{Z}$-precover}. Dually,  we have the notion of {\bf $\Z$-preenvelope} and {\bf special $\Z$-preenvelope}.
\

Let $(\mathcal{X}, \Z)\subseteq\mathcal{C}^2.$ Following,  \cite[Def. 3.12]{parte1},   it is said that 
$\mathcal{Z}$ is \textbf{special precovering in} $\mathcal{X}$
if any $X\in\mathcal{X}$ admits an exact sequence 
$\suc[B][A][X]$ in $\C$ with $\ensuremath{A}\ensuremath{\in\mathcal{Z}\cap\mathcal{X}}$  and $\ensuremath{B}\ensuremath{\in\mathcal{Z}^{\bot_{1}}\cap\mathcal{X}} .$
The notion of \textbf{special preenveloping in
$\mathcal{X}$} is defined dually.

Recall that  a cotorsion pair $\p$ is left complete if $\mathcal{A}$
is special precovering in $\mathcal{C}.$ As a generalization of that,  and following  \cite[Def. 3.13]{parte1},  it is said
that a (not necessarily cotorsion) pair $\p\subseteq\mathcal{C}^{2}$
is \textbf{left $\mathcal{X}$-complete} if any $X\in\mathcal{X}$
admits an exact sequence $\suc[B][A][X]$,  with $A\in\mathcal{A}\cap\mathcal{X}$
and $B\in\mathcal{B}\cap\mathcal{X}$. The notion of {\bf right $\mathcal{X}$-complete pair}
is defined dually. Moreover,  a pair is {\bf $\mathcal{X}$-complete} if it is right and
left $\mathcal{X}$-complete. The pair $\p$ is \textbf{$\mathcal{X}$-hereditary}
if $ \Extx[k][][\mathcal{A}\cap\mathcal{X}][\mathcal{B}\cap\mathcal{X}]=0$ $ \forall k>0$  \cite[Def. 3.7]{parte1}.

\subsection[Relative dimensions]{Relative homological dimensions and relative resolution dimensions}

In \cite{parte1},  we presented a possible generalization of a part of the Auslander-Buchweitz-Reiten approximation theory \cite{Auslander-Buchweitz,  auslandereiten} that were useful for the development of $n$-$\X$-tilting theory.
The goal of such work was to study the relations between the relative
homological dimensions and the existence of a particular class of relative
(co)resolutions. In what follows, we recall some of these notions and notations introduced in \cite{parte1}, for a more detailed treatment, we recommend the reader to see in \cite{parte1}.
\

Let   $\mathcal{B}, \mathcal{A}\subseteq\mathcal{C}$, 
and $C\in\mathcal{C}$. Following  \cite{Auslander-Buchweitz},  the \textbf{$\mathcal{A}$-projective dimension} $\pdr[][C]$
of $C$ is 
$\pdr[][C]: =\min\left\{ n\in\mathbb{N}\, |\: \Extx[k][][C][\mathcal{A}]=0\, \forall k>n\right\} \mbox{, }$
where the minimum of the empty set is the symbol $\infty.$ The \textbf{$\mathcal{A}$-projective dimension} of $\mathcal{B}$
is  $\pdr[][\mathcal{B}]: =\sup\left\{ \pdr[][B]\, |\: B\in\mathcal{B}\right\}.$ Dually, the 
\textbf{$\mathcal{A}$-injective dimension} $\id_\A(C)$ of $C$ and the\textbf{ $\mathcal{A}$-injective dimension} $\id_\A(\B)$ of 
$\mathcal{B}$ are defined dually.

We recall now,  from \cite[Def. 4.1]{parte1},  the notions of relative (co)resolution classes. Indeed, let $M\in\mathcal{C}$ and $\X,\Y\subseteq\mathcal{C}.$ A \textbf{$\mathcal{Y}_{\mathcal{X}}$-coresolution} of $M$ is an exact sequence in $\C$ of the form
$0\rightarrow M\stackrel{f_{0}}{\rightarrow}Y_{0}\stackrel{f_{1}}{\rightarrow}Y_{1}\stackrel{}{\rightarrow}...\stackrel{}{\rightarrow}Y_{n-1}\stackrel{f_{n}}{\rightarrow}Y_{n}\stackrel{}{\rightarrow}\cdots\mbox{, }$ with
$Y_{k}\in\mathcal{Y}\cup\left\{ 0\right\} $ $\forall k\geq0$ and
$\im[f_{i}]\in\mathcal{X}\cup\{0\}$ $\forall i\geq1.$ The class of all the objects in $\mathcal{C}$ having a $\mathcal{Y}_{\mathcal{X}}$-coresolution is denoted by $\mathcal{Y}_{\mathcal{X}, \infty}^{\vee}.$ A \textbf{finite (of length $n$) $\mathcal{Y}_{\mathcal{X}}$-coresolution}
of $M$ is an exact sequence in $\C$ of the form
$0\rightarrow M\stackrel{f_{0}}{\rightarrow}Y_{0}\stackrel{f_{1}}{\rightarrow}Y_{1}\stackrel{}{\rightarrow}...\stackrel{}{\rightarrow}Y_{n-1}\stackrel{f_{n}}{\rightarrow}Y_{n}\stackrel{}{\rightarrow}0\mbox{, }$ with
$Y_{n}\in\mathcal{X}\cap\mathcal{Y}$,  $Y_{k}\in\mathcal{Y}$ $\forall k\in[0, n-1]$, 
and $\im[f_{i}]\in\mathcal{X}$ $\forall i\in[1, n-1].$ The class of all the objects in $\mathcal{C}$ having a finite $\mathcal{Y}_{\mathcal{X}}$-coresolution is denoted by $\mathcal{Y}^{\vee}_{\mathcal{X}}.$ Moreover,  the  class of all the objects in $\mathcal{C}$ having a 
$\mathcal{Y}_{\mathcal{X}}$-coresolution of length $\leq n$  is denoted by $\mathcal{Y}^{\vee}_{\mathcal{X}, n}.$ Notice that 
$\cup_{n\in\mathbb{N}}\mathcal{Y}^{\vee}_{\mathcal{X}, n}=\mathcal{Y}^{\vee}_{\mathcal{X}}\subseteq \mathcal{Y}_{\mathcal{X}, \infty}^{\vee}.$ The \textbf{$\mathcal{Y}_{\mathcal{X}}$-coresolution dimension} 
of $M$ is  
$\coresdimr{\mathcal{Y}}M{\mathcal{X}}: =\min\{n\in\mathbb{N}\, |\,  M\in \mathcal{Y}^{\vee}_{\mathcal{X}, n}\}.$
For  $\mathcal{Z}\subseteq\C, $ we set 
$\coresdimr{\mathcal{Y}}{\mathcal{Z}}{\mathcal{X}}: =\sup\left\{ \coresdimr{\mathcal{Y}}Z{\mathcal{X}}\, |\: Z\in\mathcal{Z}\right\}.$ We consider the classes  $\p[\mathcal{X}][\mathcal{Y}]_{\infty}^{\vee}: =\mathcal{X}\cap\mathcal{Y}_{\mathcal{X}, \infty}^{\vee}, $
$\p[\mathcal{X}][\mathcal{Y}]^{\vee}: =\mathcal{X}\cap\mathcal{Y}_{\mathcal{X}}^{\vee}$
and $\p[\mathcal{X}][\mathcal{Y}]_{n}^{\vee}: =\mathcal{X}\cap\mathcal{Y}_{\mathcal{X}, n}^{\vee}.$ Dually,  it can be defined the $\mathcal{Y}_{\mathcal{X}}$-resolution (of length $n$) of $M,$ the $\mathcal{Y}_{\mathcal{X}}$-resolution dimension
$\resdimr{\mathcal{Y}}M{\mathcal{X}}$ of $M$ and the classes  $\mathcal{Y}_{\mathcal{X}}^{\wedge}$,  $\mathcal{Y}_{\mathcal{X}, \infty}^{\wedge}$ and
$\mathcal{Y}_{\mathcal{X}, n}^{\wedge}.$ We also have the classes
$\p[\mathcal{Y}][\mathcal{X}]_{\infty}^{\wedge}: =\mathcal{Y}_{\mathcal{X}, \infty}^{\wedge}\cap\mathcal{X}, $
$\p[\mathcal{Y}][\mathcal{X}]^{\wedge}: =\mathcal{Y}_{\mathcal{X}}^{\wedge}\cap\mathcal{X}$
and $\p[\mathcal{Y}][\mathcal{X}]_{n}^{\wedge}: =\mathcal{Y}_{\X, n}^{\wedge}\cap\mathcal{X}.$ If $\mathcal{X}=\mathcal{C}, $ we omit the ``$\mathcal{X}$'' symbol in the above notations. Note that,  $M$ is isomorphic to some object in 
 $\X\cap\Y$ if,  and only if,  $\coresdimr{\mathcal{Y}}M{\mathcal{X}}=0$ (respectively,  $\resdimr{\mathcal{Y}}M{\mathcal{X}}=0$).

\subsection{Closure properties}

Let   $\mathcal{Y}\subseteq\mathcal{X}\subseteq\mathcal{C}$ and $n\geq1.$ Following \cite[Def. 2.4]{parte1},  
we recall that $\mathcal{Y}$ is \textbf{closed by $n$-quotients in
$\mathcal{X}$} if for any exact sequence $0\rightarrow A\rightarrow Y_{n}\overset{\varphi_{n}}{\rightarrow}...\overset{}{\rightarrow}Y_{1}\overset{\varphi_{1}}{\rightarrow}B\rightarrow0$ in $\C, $
with $Y_{i}\in\mathcal{Y}$,  $\Kerx[\varphi_{i}]\in\mathcal{X}$ $\forall i\in[1, n]$
and $B\in\mathcal{X}$,  we have that $B\in\mathcal{Y}$. The notion of being
\textbf{closed by $n$-subobjects in $\X$} is defined dually. These closure
properties are useful to characterize classes $\mathcal{T}\subseteq\mathcal{C}$
such that $\pdr[\mathcal{X}][\mathcal{T}]\leq n$  and $\idr[\mathcal{X}][\mathcal{T}]\leq n, $ respectively,  
see \cite[Prop. 2.6]{parte1}. 

Other closure notions that we will be using in the development of the paper are the following ones \cite[Def. 3.3]{parte1}. Let
$\mathcal{M}, \mathcal{X}\subseteq\mathcal{C}$. We say that $\mathcal{M}$
is \textbf{closed under mono-cokernels in $\mathcal{M}\cap\mathcal{X}$}
if,  for any exact sequence $\suc[M][M'][M'']$ in $\C,$ with $M, M'\in\mathcal{M}\cap\mathcal{X},$ 
 we have that $M''\in\mathcal{M}.$ Dually,  it can be defined the notion of being 
\textbf{closed under epi-kernels in $\mathcal{M}\cap\mathcal{X}$}. In case $\mathcal{M}\subseteq\mathcal{X}$, 
we will simply say that $\mathcal{M}$ is closed under mono-cokernels and  epi-kernels,  respectively. Furthermore, 
$\mathcal{M}$ is \textbf{$\mathcal{X}$-resolving} if $\mathcal{M}$
contains an $\mathcal{X}$-projective relative generator in $\mathcal{X}$, 
it is closed under epi-kernels in $\mathcal{M}\cap\mathcal{X}$ and
under extensions; and the notion of being \textbf{$\mathcal{X}$-coresolving} is defined dually.
These notions are very useful to identify $\mathcal{X}$-hereditary pairs, see  \cite[Lems. 3.4 and 3.6]{parte1}. 

Following \cite[Def. 2.2]{ABsurvey}, a class $\mathcal{X}\subseteq\mathcal{C}$    
is \textbf{right thick} (\textbf{left thick}) if it is closed under extensions,  direct summands
and mono-cokernels (epi-kernels); and $\mathcal{X}$ is \textbf{thick} if it is left and right thick.

\subsection{The class of relative $n$-quotients}
Let  $\T, \X\subseteq\C.$ Following \cite[Sect. 5]{parte1},  we recall the notion of the relative 
$n$-$(\X, \T)$-quotients in $\C, $ and the different variants related with small and big classes.
\

For any integer $n\geq1, $ $\Gennr$ denotes  the class of the objects $C\in\mathcal{C}$
admitting an exact sequence $0\rightarrow K\rightarrow T_{n}\stackrel{f_{n}}{\rightarrow}...\stackrel{f_{2}}{\rightarrow}T_{1}\stackrel{f_{1}}{\rightarrow}C\rightarrow0$ in $\C, $ with $\Kerx[f_{i}]\in\mathcal{X}$ and $T_{i}\in\mathcal{T}\cap\mathcal{X}$
$\forall i\in[1, n].$ We also define $\operatorname{Gen}_{n}^{\mathcal{X}}(\mathcal{T}): =\Gennr[\mathcal{T}^{\oplus}][n][\mathcal{X}]$
and $\operatorname{gen}_{n}^{\mathcal{X}}(\mathcal{T}): =\Gennr[\mathcal{T}^{\oplus_{<\infty}}][n][\mathcal{X}]$.
For an object $T\in\mathcal{C}$,  we define $\operatorname{Gen}_{n}^{\mathcal{X}}(T): =\operatorname{Gen}_{n}^{\mathcal{X}}(\operatorname{Add}(T))$
and $\operatorname{gen}_{n}^{\mathcal{X}}(T): =\operatorname{gen}_{n}^{\mathcal{X}}(\operatorname{add}(T))$.
In case of $\mathcal{X}=\mathcal{C}$,  we set $\Genn[\mathcal{T}]: =\Gennr[\mathcal{T}][][\mathcal{C}]$, 
$\operatorname{Gen}_{n}(\mathcal{T}): =\operatorname{Gen}_{n}^{\mathcal{C}}(\mathcal{T})$
and $\operatorname{gen}_{n}(\mathcal{T}): =\operatorname{gen}_{n}^{\mathcal{C}}(\mathcal{T}).$ Some closure properties that such classes have can be found in \cite[Prop. 5.2]{parte1}.

\section{$n$-$\X$-tilting classes}

In this section,  we introduce the notion of $n$-$\X$-tilting class in an abelian category $\C$ and develop a relative tilting theory on 
$\X\subseteq\C.$  
Without further ado,  let us define our main object of study.

\begin{defn}\label{def: X-tilting}\label{def:  tilting peque=0000F1o-1} Let   $\mathcal{X}\subseteq\mathcal{C}$ and $n\in\mathbb{N}$.
A class $\mathcal{T}\subseteq\mathcal{C}$ is  \textbf{$n$-$\mathcal{X}$-tilting}
if the following conditions hold true.
\begin{description}
\item [(T0)] $\mathcal{T}=\smdx[\mathcal{T}].$
\item [(T1)] $\pdr[\mathcal{X}][\mathcal{T}]\leq n.$
\item [{(T2)}] $\mathcal{T}\cap\mathcal{X}\subseteq\mathcal{T}^{\bot}.$
\item [{(T3)}] There is a class $\omega\subseteq\mathcal{T}_{\mathcal{X}}^{\vee}$
which is a relative generator in $\mathcal{X}.$
\item [{(T4)}] There is a class $\alpha\subseteq\mathcal{X}^{\bot}\cap\mathcal{T}^{\bot}$
which is a relative cogenerator in $\mathcal{X}.$
\item [{(T5)}] Every $Z\in\mathcal{T}^{\bot}\cap\mathcal{X}$ admits a
$\mathcal{T}$-precover $T'\rightarrow Z$,  with $T'\in\mathcal{X}$.
\end{description}
An $n$-$\mathcal{X}$-tilting class $\mathcal{T}\subseteq\mathcal{C}$
is \textbf{big} (\textbf{small}) if $\mathcal{T}=\mathcal{T}^{\oplus}$
($\mathcal{T}=\mathcal{T}^{\oplus_{<\infty}}$). An object $T\in\mathcal{C}$
is\textbf{ big} \textbf{(small}) \textbf{$n$-$\mathcal{X}$-tilting}
if $\Addx[T]$ ($\addx[T]$) is $n$-$\mathcal{X}$-tilting. 
\end{defn}

Notice that the condition (T4) requires the existence of an $\X$-injective relative cogenerator in $\X$.
 It is a well-known fact that this property is satisfied, for example, by the class of 
 finitely generated modules over an Artin $k$-algebra. A non trivial situation where there also exist such relative cogenerator is in the category $\operatorname{Rep}(Q,\mathcal{C})$ of representations in an abelian category $\C$ of an arbitrary quiver $Q.$ Indeed, in \cite[Cor. 5.18]{AM23}, we show that, if $Q$ has a finite number of paths starting or ending at each vertex of $Q$ and $\C$ has enough injectives, then the class 
 $\X:=\operatorname{Rep}^{f}(Q,\mathcal{C})$ (of all the representations having finite support) admits an $\X$-injective relative cogenerator in $\X$. On the other hand, the condition (T5) is very helpful to prove nice properties of the pair $({}^\perp(\T^\perp), \T^\perp).$ For example, by using (T5), we can show that such a pair is $\X$-complete and that $\T\cap\X$ is a relative generator in $\T^\perp\cap\X.$

In Section 4,  we will show that the above definition generalizes a big variety of previous notions of tilting. In Section 5, more concrete examples are given in the context of representations of quivers in abelian categories.
For now,  the most nearby example is the tilting object in abelian
categories developed by Leonid Positselski and Jan  {\v{S}}t'ov{\'\i}{\v{c}}ek
\cite{positselskicorrespondence},  that we call PS $n$-tilting.

\begin{defn}\label{def: PS-tilting}\cite[Sect. 2,  Thm. 3.4(3)]{positselskicorrespondence} Let $\mathcal{C}$
be  AB3 and AB3{*}  with an injective cogenerator. An object
$T\in\mathcal{C}$ is \textbf{PS $n$-tilting} if the following conditions hold true.
\begin{description}
\item [{(PST1)}] $\pdx[T]\leq n.$
\item [{(PST2)}] $\Addx[T]\subseteq T^{\bot}.$
\item [{(PST3)}] There is a generating class $\mathcal{G}$ in $\mathcal{C}$
such that $\mathcal{G}\subseteq\left(\Addx[T]\right)^{\vee}$.
\end{description}
\end{defn}

\begin{rk} \label{rk: PS-tilting} Notice that $T\in\mathcal{C}$ is PS $n$-tilting
 if,  and only if,  $T$ is big $n$-$\mathcal{C}$-tilting.
Indeed,  it can be seen,  by taking $\mathcal{X}=\mathcal{C}$ and $\T=\Add(T)$ in Definition \ref{def: X-tilting},  that the conditions (T4) and (T5) are satisfied trivially,  and that
the conditions (PST1),   (PST2),   and (PST3) coincide with (T1),  (T2),  and (T3),  respectively.
\end{rk}

\subsection{Elementary properties of relative tilting classes}

\begin{lem}\label{lem: chico}\label{lem: inf1}\label{lem: inf1-1}\label{lem: chico-1}
For  $\mathcal{X}\subseteq\mathcal{C}, $
 the following statements hold true.
\begin{itemize}
\item[$\mathrm{(a)}$] If $\mathcal{T}\subseteq\mathcal{C}$ satisfies $\mathrm{(T2)}, $ then $\mathcal{T}\cap\mathcal{X}\subseteq\mathcal{T}^{\bot}\cap{}{}^{\bot}\left(\mathcal{T}^{\bot}\right).$ 
\item[$\mathrm{(b)}$]  If $\mathcal{T}\subseteq\mathcal{C}$ satisfies $\mathrm{(T1)}$ and $\mathrm{(T4), }$ then
$\mathcal{X}\subseteq(\mathcal{T}^{\bot}\cap\mathcal{X}){}_{\mathcal{X}, n}^{\vee}$.
\end{itemize}
\end{lem}
\begin{proof} (a)  By (T2),  $\mathcal{T}\cap\mathcal{X}\subseteq \mathcal{T}^{\bot}.$
Moreover,  $\mathcal{M}\subseteq{}{}^{\bot}\left(\mathcal{M}^{\bot}\right), $  for any class $\mathcal{M}\subseteq\mathcal{C}.$ Therefore 
$\mathcal{T}\cap\mathcal{X}\subseteq\mathcal{T}^{\bot}\cap{}{}^{\bot}\left(\mathcal{T}^{\bot}\right)\mbox{.}$ 
\

(b) It follows from \cite[Prop. 4.5(a)]{parte1}.
\end{proof}

\begin{lem}\label{lem: inf2}\label{lem: inf2-1} For  $\mathcal{X}, \mathcal{T}\subseteq\mathcal{C}, $ the following statements hold true.
\begin{itemize}
\item[$\mathrm{(a)}$]  $\mathcal{T}_{\mathcal{X}}^{\vee}\cap\mathcal{X}\subseteq\mathcal{T}{}^{\vee}\subseteq{}{}^{\bot}\left(\mathcal{T}^{\bot}\right)\subseteq{}{}^{\bot}\left(\mathcal{T}^{\bot}\cap\mathcal{X}\right).$
\item[$\mathrm{(b)}$]  if $\mathcal{X}=\smdx[\mathcal{X}]$ and $\mathcal{T}$ satisfies $\mathrm{(T0)}$
and $\mathrm{(T2)}, $ then $(\mathcal{T}\cap\mathcal{X})_{\mathcal{X}}^{\vee}\cap\mathcal{T}^{\bot}=\mathcal{T}\cap\mathcal{X}.$ 
\end{itemize}
\end{lem}
\begin{proof} (a) The inclusion $\mathcal{T}{}^{\vee}\subseteq{}{}^{\bot}\left(\mathcal{T}^{\bot}\right)$ follows from \cite[Lem. 4.3]{parte1}.
\

(b) Let $A\in(\mathcal{T}\cap\mathcal{X})_{\mathcal{X}}^{\vee}\cap \mathcal{T}^{\bot}$.
Hence,  there is an exact sequence 
\[
\eta: \: \suc[A][T_{0}][A']\mbox{ with }T_{0}\in\mathcal{T}\cap\mathcal{X}\mbox{ and }A'\in\left(\mathcal{T}\cap\mathcal{X}\right)_{\mathcal{X}}^{\vee}, 
\]
where $A'\in{}^{\bot}\left(\mathcal{T}^{\bot}\right)$ by (a).
Note that $\eta$ splits since $A\in\mathcal{T}^{\bot}, $ and thus,  $A\in\mathcal{T}\cap\mathcal{X}.$
\

Since $\T\cap\X\subseteq (\mathcal{T}\cap\mathcal{X})_{\mathcal{X}}^{\vee}, $ we get from  (T2) that 
$\T\cap\X\subseteq (\mathcal{T}\cap\mathcal{X})_{\mathcal{X}}^{\vee}\cap\mathcal{T}^{\bot}.$
\end{proof}

\begin{lem} \label{lem: inf3}\label{lem: inf3-1} Let    $\mathcal{X}=\smdx[\mathcal{X}]\subseteq\mathcal{C}$ and 
 $\mathcal{T}\subseteq\mathcal{C}$ satisfying  $\mathrm{(T0)}, $ $\mathrm{(T1)}, $  $\mathrm{(T2)}$  and
$\mathrm{(T4)}.$ Then  
\begin{itemize}
\item[$\mathrm{(a)}$] $\coresdimr{\mathcal{T}\cap\mathcal{X}}{(\mathcal{T}\cap\mathcal{X})_{\mathcal{X}}^{\vee}\cap\mathcal{X}}{\mathcal{X}}\leq\pdr[\mathcal{X}][\mathcal{T}]$;
\item[$\mathrm{(b)}$] $(\mathcal{T}\cap\mathcal{X})_{\mathcal{X}}^{\vee}=(\mathcal{T}\cap\mathcal{X})_{\mathcal{X}, k}^{\vee}\;\forall k>\pdr[\mathcal{X}][\mathcal{T}].$
\end{itemize}
\end{lem}
\begin{proof}
We consider $\W: =(\mathcal{T}\cap\mathcal{X})_{\mathcal{X}}^{\vee}\cap\mathcal{X}$
and $m: =\max\{1, \pdr[\mathcal{X}][\mathcal{T}]\}.$
\

(a) Let $X\in\W$. Then,  there is an exact sequence 
\[
0\rightarrow X\stackrel{f_{0}}{\rightarrow}Y_{0}\stackrel{f_{1}}{\rightarrow}Y_{1}\stackrel{}{\rightarrow}...\stackrel{}{\rightarrow}Y_{m-1}\stackrel{f_{m}}{\rightarrow}Y_{m}\stackrel{}{\rightarrow}0\mbox{, }
\]
with $Y_{m}\in\W$,  $Y_0, Y_{i}\in\mathcal{T}\cap\mathcal{X}$ and
$\im[f_{i}]\in\mathcal{X}$ $\forall i\in[1, m-1]$. Moreover,  by (T1),  (T4), 
and \cite[Prop. 2.6]{parte1}, 
it follows that $Y_{m}\in\mathcal{T}^{\bot}\cap\mathcal{X}$. Then, 
by Lemma \ref{lem: inf2}(b),  $Y_{m}\in\W\cap\mathcal{T}^{\bot}=\mathcal{T}\cap\mathcal{X}$
 and thus $\coresdimr{\mathcal{T}\cap\mathcal{X}}{\W}{\mathcal{X}}\leq m$.\\
Assume now that $\pdr[\mathcal{X}][\mathcal{T}]=0.$ Then
$\mathcal{T}\subseteq{}^{\bot}\mathcal{X}$  and,  for any
$W\in\W, $ there is an exact sequence $\eta_{W}: \: \suc[W][T_{W}][C_{W}][\, ][\, ]$
with $T_{W}, C_{W}\in\mathcal{T}\cap\mathcal{X}$. Now,  since $\mathcal{T}\subseteq{}^{\bot}\mathcal{X}$
and $\W\subseteq\mathcal{X}$,  $\Extx[1][][\mathcal{T}\cap\mathcal{X}][\W]=0$.
Hence $\eta_{W}$ splits $\forall W\in\W$. In particular,  $\W\subseteq\mathcal{T}\cap\mathcal{X}$
and thus $\coresdimr{\mathcal{T}\cap\mathcal{X}}{\W}{\mathcal{X}}=0;$ proving (a).
\
 
(b) Let $M\in(\mathcal{T}\cap\mathcal{X})_{\mathcal{X}}^{\vee}$. Then,  there 
is an exact sequence 
\[
\suc[M][T_{0}][M']\mbox{ with }T_{0}\in\mathcal{T}\cap\mathcal{X}\mbox{ and }M'\in\W.
\]
 It follows from (a) that $\coresdimr{\mathcal{T}\cap\mathcal{X}}{M'}{\mathcal{X}}\leq\pdr[\mathcal{X}][\mathcal{T}]=: n$.
Hence $M'\in(\mathcal{T}\cap\mathcal{X})_{\mathcal{X}, n}^{\vee}$
and thus $M\in(\mathcal{T}\cap\mathcal{X})_{\mathcal{X}, n+1}^{\vee}$.
Therefore,  for any $k>n, $  
$(\mathcal{T}\cap\mathcal{X})_{\mathcal{X}}^{\vee}=(\mathcal{T}\cap\mathcal{X})_{\mathcal{X}, n+1}^{\vee}\subseteq(\mathcal{T}\cap\mathcal{X})_{\mathcal{X}, k}^{\vee}\subseteq(\mathcal{T}\cap\mathcal{X})_{\mathcal{X}}^{\vee};$ proving (b).
\end{proof}

\begin{cor}\label{cor: coronuevo pag 55} Let 
$\mathcal{X}=\smdx[\mathcal{X}]\subseteq\mathcal{C}$ be closed under extensions, 
$\mathcal{T}\subseteq\mathcal{C}$ be $n$-$\mathcal{X}$-tilting,  and
$\omega=\smdx[\omega]$ be an $\mathcal{X}$-projective relative generator
in $\mathcal{X}$ such that $\omega\subseteq\mathcal{T}_{\mathcal{X}}^{\vee}$.
Then,  $\omega=\mathcal{X}\cap{}^{\bot}\mathcal{X}$ and $\coresdimr{\mathcal{T}\cap\mathcal{X}}{\omega}{\mathcal{X}}\leq\pdr[\mathcal{X}][\mathcal{T}]$.
Furthermore,  $\omega=\mathcal{T}\cap\mathcal{X}$ if $\pdr[\mathcal{X}][\mathcal{T}]=0$. 
\end{cor}
\begin{proof}
By the dual of \cite[Prop. 2.7]{ABsurvey},  we have that $\omega=\mathcal{X}\cap{}^{\bot}\mathcal{X}.$
On the other hand,  $\omega\subseteq(\mathcal{T}\cap\mathcal{X})_{\mathcal{X}}^{\vee}\cap\mathcal{X}$ since $\omega\subseteq\mathcal{X}\cap\mathcal{T}_{\mathcal{X}}^{\vee}$ and $\mathcal{X}$ is closed under extensions. Therefore,  by Lemma \ref{lem: inf3} (a),  it follows that $\coresdimr{\mathcal{T}\cap\mathcal{X}}{\omega}{\mathcal{X}}\leq\pdr[\mathcal{X}][\mathcal{T}]$.
\

Let us assume that $\pdr[\mathcal{X}][\mathcal{T}]=0$. Then $\mathcal{T}\subseteq{}^{\bot}\mathcal{X}$ and
 $\coresdimr{\mathcal{T}\cap\mathcal{X}}{\omega}{\mathcal{X}}=0.$ Hence
 $\omega\subseteq\mathcal{T}\cap\mathcal{X}\subseteq\mathcal{X}\cap{}^{\bot}\mathcal{X}=\omega$.
\end{proof}

The following result is a generalization of \cite[Lem. 2.3]{Tiltinginfinitamentegenerado}.

\begin{lem}\label{lem: inf4}\label{lem: props T2 con T3' y C2 con C3'}\label{lem: inf4-1}\label{lem: props T2 con T3' y C2 con C3'-1}
Let  $\mathcal{X}\subseteq\C$ be 
closed under extensions. If $\mathcal{T}\subseteq\mathcal{C}$ satisfies
$\mathrm{(T3)}, $ then $\mathcal{T}^{\bot}\cap\mathcal{X}\subseteq\Gennr[\mathcal{T}][1][\mathcal{X}]$.\end{lem}

\noindent \begin{minipage}[t]{0.70\columnwidth}%
\begin{proof}Let $A\in\mathcal{T}^{\bot}\cap\mathcal{X}$. By (T3),  there is an
exact sequence $\eta _1: \;\suc[K][W][A][][a]\mbox{, }$ with $W\in\omega$ and $K\in\mathcal{X}$. Moreover,  by (T3) and 
Lemma \ref{lem: inf2} (a), 
there is an exact sequence
$\eta _2: \;\suc[W][B][C][b]$
with $B\in\mathcal{T}$ and $C\in\mathcal{T}{}^{\vee}\cap\mathcal{X}\subseteq{}{}^{\bot}\left(\mathcal{T}^{\bot}\right)\cap\mathcal{X}$.
Notice that $B\in\mathcal{X}.$  Now,  considering the push-out of $b$
and $a$,  we get the exact sequences
$\eta _3: \: \suc[K][B][B'][][x] $ and
$\eta _4: \: \suc[A][B'][C][t] $, with $B'\in\Gennr[\mathcal{T}][1]$. 
Furthermore,  $\eta_4$ splits since $A\in\T^\perp$ and $C\in{}^\perp(\T^\perp).$ Thus,  there is some
$y: B'\rightarrow A$ such that $yt=1_A.$ Consider the exact sequence $\eta_5: \: \suc[K'][B][A][][yx]\mbox{.}$
Since $B\in\T\cap\X$,  it remains to show that $K'\in\mathcal{X}$.
For that purpose observe that,  by using $\eta_5$ and $\eta_2$,  we can
build the exact sequence 
$\suc[K][K'][C]\mbox{, }$
where $K, C\in\mathcal{X}$. Therefore $K'\in\mathcal{X}$ since $\X$ is closed under extensions.
\end{proof}%
\end{minipage}\hfill{}%
\fbox{\begin{minipage}[t]{0.25\columnwidth}%
\[
\begin{tikzpicture}[-, >=to, shorten >=1pt, auto, node distance=1cm, main node/.style=, x=1cm, y=1cm]

   \node[main node] (C) at (0, 0)      {$A$};
   \node[main node] (Z) [right of=C]  {$B'$};
   \node[main node] (X1) [right of=Z]  {$C$};

   \node[main node] (X) [above of=C]  {$W$};
   \node[main node] (W) [right of=X]  {$B$};
   \node[main node] (X2) [right of=W]  {$C$};

   \node[main node] (Y1) [above of=X]  {$K$};
   \node[main node] (Y2) [above of=W]  {$K$};





   \node[main node] (K') at (0, -1)      {$K$};
   \node[main node] (X') [right of=K']  {$W$};
   \node[main node] (C1') [right of=X']  {$A$};

   \node[main node] (YK') [below of=K']  {$K'$};
   \node[main node] (U') [right of=YK']  {$B$};
   \node[main node] (C2') [right of=U']  {$A$};

   \node[main node] (X1') [below of=YK']  {$C$};
   \node[main node] (X2') [below of=U']  {$C$};





\draw[right hook->,  thin]   (C)  to node  {$t$}  (Z);
\draw[->>,  thin]   (Z)  to node  {$$}  (X1);

\draw[right hook->,  thick]   (Y1)  to node  {$$}  (X);
\draw[right hook->,  thin]   (Y2)  to node  {$$}  (W);
\draw[->>,  thick]   (X)  to node  {$a$}  (C);
\draw[-,  double]   (X2)  to node  {$$}  (X1);
\draw[->>,  thin]   (W)  to node  {$x$}  (Z);

\draw[right hook->,  thick]   (X)  to  node  {$b$}  (W);
\draw[->>,  thick]   (W)  to  node   {$$}  (X2);

\draw[-,  double]   (Y1)  to  node  {$$}  (Y2);

\draw[right hook->,  thick]   (K')  to node  {$$}  (X');
\draw[->>,  thick]   (X')  to node  {$a$}  (C1');

\draw[right hook->,  thick]   (K')  to node  {$$}  (YK');
\draw[right hook->,  thin]   (X')  to node  {$b$}  (U');
\draw[->>,  thick]   (YK')  to node  {$$}  (X1');
\draw[->>,  thin]   (U')  to node  {$$}  (X2');

\draw[right hook->,  thin]   (YK')  to  node  {$$}  (U');
\draw[->>,  thin]   (U')  to  node   {$yx$}  (C2');

\draw[-,  double]   (X1')  to  node  {$$}  (X2');
\draw[-,  double]   (C1')  to  node   {$$}  (C2');
   
\end{tikzpicture}
\]%
\end{minipage}}

\vspace{2mm}

An important property of an infinitely generated tilting module of
finite projective dimension $T\in\Modx$ is that $\Addx[T]$ is a relative
generator in $T^{\bot}$. In our relative context,  such property can be translated
as the following one:  $\mathcal{T}\cap\mathcal{X}$ is a relative generator in $\mathcal{T}^{\bot}\cap\mathcal{X}$.
In that sense,  the following lemma is a generalization of \cite[Lem. 2.4]{Tiltinginfinitamentegenerado}.

\begin{lem}\label{lem: inf5}\label{lem: props C2 y T2}\label{lem: inf5-1}\label{lem: props C2 y T2-1}
For a class   $\mathcal{X}=\smdx[\mathcal{X}]\subseteq\mathcal{C}$
closed under extensions and $\mathcal{T}\subseteq\mathcal{C}$ satisfying
$\mathrm{(T2),  (T5)}$ and such that $\mathcal{T}^{\bot}\cap\mathcal{X}\subseteq\Gennr[\mathcal{T}][1], $
the following statements hold true.
\begin{itemize}
\item[$\mathrm{(a)}$] $\mathcal{T}\cap\mathcal{X}$ is a $\mathcal{T}^{\bot}\cap\mathcal{X}$-projective relative generator in $\mathcal{T}^{\bot}\cap\mathcal{X}$.
\item[$\mathrm{(b)}$]  Every morphism $A\rightarrow X$,  with $A\in{}{}^{\bot}\left(\mathcal{T}^{\bot}\cap\mathcal{X}\right)$ 
(or $A\in{}{}^{\bot}\left(\mathcal{T}^{\bot}\right)$)
and $X\in\mathcal{T}^{\bot}\cap\mathcal{X}$,  factors through $\mathcal{T}\cap\mathcal{X}$.
Moreover,  if $\mathcal{T}=\smdx[\mathcal{T}]$,  then 
\[
\mathcal{T}\cap\mathcal{X}=\mathcal{T}^{\bot}\cap\mathcal{X}\cap{}{}^{\bot}\left(\mathcal{T}^{\bot}\cap\mathcal{X}\right)=\mathcal{T}^{\bot}\cap\mathcal{X}\cap{}{}^{\bot}\left(\mathcal{T}^{\bot}\right)\mbox{.}
\]
\item[$\mathrm{(c)}$]  If $\T=\smd(\T),$ then  $\resdimx{\mathcal{T}}X\leq\resdimr{{}\mathcal{T}\cap\mathcal{X}}X{\mathcal{T}^{\bot}\cap\mathcal{X}}\leq\pdr[\mathcal{T}^{\bot}\cap\mathcal{X}][X]+1,$ for any $X\in\T^\perp\cap\X.$
\end{itemize}
\end{lem}
\begin{proof} We only need to prove (a) since (b) and (c) follow from (a).
\

Let $X\in\mathcal{T}^{\bot}\cap\mathcal{X}$. By (T5),  there is a
$\mathcal{T}$-precover $g: T'\rightarrow X$ with $T'\in\mathcal{X}$.
\begin{minipage}[t]{0.70\columnwidth}%
Notice that  $g$ is an epimorphism since $X\in$ $\mathcal{T}^{\bot}\cap\mathcal{X}\subseteq\Gennr[\mathcal{T}][1][\mathcal{X}].$ Let us prove that $K:=\Kerx[g]\in\mathcal{T}^{\bot}\cap\mathcal{X}$. 
Consider the exact sequence $\suc[K][T'][X][][g]\mbox{.}$ By (T2)
and the fact that $g$ is an $\mathcal{T}$-precover,  it follows that
$K\in\mathcal{T}^{\bot}$. It remains to show that $K\in\mathcal{X}$. Since $X\in\Gennr[\mathcal{T}][1]$, 
there is an exact sequence 
$\suc[K'][B][X][][f]\mbox{, }$
with $B\in\mathcal{T}\cap\mathcal{X}$ and $K'\in\mathcal{X}$. Let
$Z$ be the pull-back of $f$                            \makebox[\linewidth][s]{and $g$. We have the  exact
sequences
$\eta: \: \suc[K'][Z][T']$}%
\end{minipage}\hfill{}%
\fbox{\begin{minipage}[t]{0.25\columnwidth}%
\[
\begin{tikzpicture}[-, >=to, shorten >=1pt, auto, node distance=1cm, main node/.style=, x=1.5cm, y=1.5cm]

   \node[main node] (C) at (0, 0)      {$X$};
   \node[main node] (X0) [left of=C]  {$T'$};
   \node[main node] (X1) [left of=X0]  {$K$};

   \node[main node] (X) [above of=C]  {$B$};
   \node[main node] (E) [left of=X]  {$Z$};
   \node[main node] (X2) [left of=E]  {$K$};

   \node[main node] (Y1) [above of=X]  {$K'$};
   \node[main node] (Y2) [above of=E]  {$K'$};





\draw[->>,  thick]   (X0)  to node  {$g$}  (C);
\draw[right hook->,  thick]   (X1)  to node  {$$}  (X0);

\draw[right hook->,  thick]   (Y1)  to node  {$$}  (X);
\draw[right hook->,  thin]   (Y2)  to node  {$$}  (E);
\draw[->>,  thick]   (X)  to node  {$f$}  (C);
\draw[-,  double]   (X1)  to node  {$$}  (X2);
\draw[->>,  thin]   (E)  to node  {$$}  (X0);

\draw[->>,  thin]   (E)  to  node  {$$}  (X);
\draw[right hook->,  thin]   (X2)  to  node   {$$}  (E);

\draw[-,  double]   (Y2)  to  node  {$$}  (Y1);
   
\end{tikzpicture}
\]%
\end{minipage}}
\

\noindent and $\eta': \: \suc[K][Z][B]\mbox{.}$
Since $K', T'\in\mathcal{X}$,  we have $Z\in\mathcal{X}.$ Furthermore,  $\eta'$
splits since $K\in \T^{\bot}$ and $B\in\mathcal{T}$. Therefore $K\in\mathcal{X}.$ 
\end{proof}

\begin{lem}\label{lem: inf6}\label{lem: inf6-1} 
Let  $\mathcal{X}=\smdx[\mathcal{X}]\subseteq\mathcal{C}$ be
closed under extensions,  and let $\mathcal{T}=\smdx[\mathcal{T}]\subseteq\mathcal{C}$ be a class 
satisfying $\mathrm{(T1),  (T2),  (T4),  (T5)}$ and such that $\mathcal{T}^{\bot}\cap\mathcal{X}\subseteq\Gennr[\mathcal{T}][1][\mathcal{X}]$.
Then,  $\mathcal{X}\subseteq(\mathcal{T}^{\bot}\cap\mathcal{X})_{\mathcal{X}}^{\vee}$
and $(\mathcal{T}\cap\mathcal{X})^{\vee}\subseteq{}{}^{\bot}(\mathcal{T}^{\bot}\cap\mathcal{X})$.
Moreover,  for each $X\in\mathcal{X}, $ the following statements hold true: 
\begin{itemize}
\item[$\mathrm{(a)}$] $m: =\coresdimr{\mathcal{T}^{\bot}\cap\mathcal{X}}X{\mathcal{X}}\leq\pdr[\mathcal{X}][\mathcal{T}]<\infty$;
\item[$\mathrm{(b)}$] there are exact sequences 
$\suc[X][M_{X}][C_{X}]$ and $\suc[K_{X}][B_{X}][X]$
such that $M_{X}, \: K_{X}\in\mathcal{T}^{\bot}\cap\mathcal{X}$; $C_{X}, \: B_{X}\in(\X,\T\cap\X)^{\vee};$
$\coresdimr{\mathcal{T}\cap\mathcal{X}}{C_{X}}{\mathcal{X}}=m-1$
and $\coresdimr{\mathcal{T}\cap\mathcal{X}}{B_{X}}{\mathcal{X}}\leq m;$
\item[$\mathrm{(c)}$] $B_{X}\rightarrow X$ is a $(\mathcal{T}\cap\mathcal{X})^{\vee}$-precover;
\item[$\mathrm{(d)}$] $X\rightarrow M_{X}$ is a $\mathcal{T}^{\bot}\cap\mathcal{X}$-preenvelope.
\end{itemize}
\end{lem}
\begin{proof}
By Lemma \ref{lem: inf5},  it follows that $\mathcal{T}\cap\mathcal{X}$
is a $\mathcal{T}^{\bot}\cap\mathcal{X}$-projective relative generator
in $\mathcal{T}^{\bot}\cap\mathcal{X}.$ In particular, by \cite[Lem. 4.3]{parte1}, we get that $(\mathcal{T}\cap\mathcal{X})^{\vee}\subseteq{}{}^{\bot}(\mathcal{T}^{\bot}\cap\mathcal{X}).$  Moreover, by Lemma \ref{lem: inf1},  $\mathcal{X}\subseteq(\mathcal{T}^{\bot}\cap\mathcal{X})_{\mathcal{X}, n}^{\vee}$ for $n:=\pd_\X(\T).$  Hence,  by \cite[Thm. 4.4]{parte1}, the result follows.
\end{proof}

In what follows,  we will see that the condition $\mathcal{T}^{\bot}\cap\mathcal{X}\subseteq\Gennr[\mathcal{T}][1][\mathcal{X}]$, 
obtained in Lemma \ref{lem: inf4},  is equivalent to (T3) if it is assumed that $\mathcal{T}$
satisfies (T1),  (T2),  (T4),  and (T5). The next proposition is a generalization
of \cite[Thm. 3.4 (2, 3)]{positselskicorrespondence}.

\begin{prop}\label{prop: equiv a t3} 
Let  $\mathcal{X}=\smdx[\mathcal{X}]\subseteq\mathcal{C}$ be closed
under extensions,  and let $\mathcal{T}=\smdx[\mathcal{T}]\subseteq\mathcal{C}$ be a class
satisfying $\mathrm{(T1),  (T2),  (T4)}, $ and $\mathrm{(T5)}.$ Then,  $\mathcal{T}$ satisfies 
$\mathrm{(T3)}$ if
and only if $\mathcal{T}^{\bot}\cap\mathcal{X}\subseteq\Gennr[\mathcal{T}][1][\mathcal{X}]$.
Furthermore,  in such case,  we can choose a relative generator $\omega$
in $\mathcal{X}$ such that $\omega\subseteq(\mathcal{T}\cap\mathcal{X})_{\mathcal{X}}^{\vee}$.\end{prop}
\begin{proof}
By Lemma \ref{lem: inf4},   it is enough to prove that $\mathcal{T}^{\bot}\cap\mathcal{X}\subseteq\Gennr[\mathcal{T}][1][\mathcal{X}]$
implies (T3). 

\noindent
\begin{minipage}[t]{0.7\columnwidth}%
By Lemma \ref{lem: inf6},  every $X\in\mathcal{X}$ admits
an exact sequence $\suc[X][M_{X}][C_{X}][f]\mbox{, }$
with $M_{X}\in\mathcal{T}^{\bot}\cap\mathcal{X}, C_{X}\in(\mathcal{X}, \mathcal{T}\cap\mathcal{X})_{\mathcal{X}}^{\vee}.$
From the inclusion $\mathcal{T}^{\bot}\cap\mathcal{X}\subseteq\Gennr[\mathcal{T}][1][\mathcal{X}]$,  we have that
 $M_{X}$ admits a short exact sequence
$\suc[M'_{X}][T_{0}][M_{X}][][g]$
with $T_{0}\in\mathcal{T}\cap\mathcal{X}$ and $M'_{X}\in\mathcal{X}$.
Considering the pull-back of $f$ and $g$,  we get an exact sequence
$\suc[M'_{X}][P_{X}][X]\mbox{, }$
where $P_{X}\in(\mathcal{X}, \mathcal{T}\cap\mathcal{X})^{\vee}$.
Hence,  $\omega:=\left\{ P_{X}\right\} _{X\in\mathcal{X}}$ is a relative generator
in $\mathcal{X}$ satisfying (T3).%
\end{minipage}\hfill{}%
\fbox{\begin{minipage}[t]{0.25\columnwidth}%
\[
\begin{tikzpicture}[-, >=to, shorten >=1pt, auto, node distance=1cm, main node/.style=, x=1.5cm, y=1.5cm]

   \node[main node] (C) at (0, 0)      {$X$};
   \node[main node] (Z) [right of=C]  {$M_X$};
   \node[main node] (X1) [right of=Z]  {$C_X$};

   \node[main node] (X) [above of=C]  {$P_X$};
   \node[main node] (W) [right of=X]  {$T_0$};
   \node[main node] (X2) [right of=W]  {$C_X$};

   \node[main node] (Y1) [above of=X]  {$M'_X$};
   \node[main node] (Y2) [above of=W]  {$M'_X$};





\draw[right hook->,  thick]   (C)  to node  {$f$}  (Z);
\draw[->>,  thick]   (Z)  to node  {$$}  (X1);

\draw[right hook->,  thin]   (Y1)  to node  {$$}  (X);
\draw[right hook->,  thick]   (Y2)  to node  {$$}  (W);
\draw[->>,  thin]   (X)  to node  {$ $}  (C);
\draw[-,  double]   (X2)  to node  {$$}  (X1);
\draw[->>,  thick]   (W)  to node  {$g$}  (Z);

\draw[right hook->,  thin]   (X)  to  node  {$ $}  (W);
\draw[->>,  thin]   (W)  to  node   {$$}  (X2);

\draw[-,  double]   (Y1)  to  node  {$$}  (Y2);
   
\end{tikzpicture}
\]%
\end{minipage}}\\
\end{proof}

 Let $R$ be a ring. It can
be proved that $T^{\bot}$ is  preenveloping in $\Modx, $ for any 
$T\in\Modx$ \cite[Thm. 3.2.1]{Approximations}. This is a property
that greatly enriches tilting theory. Below, in item (c),  we will prove a similar
property in our relative context.

\begin{thm}\label{thm: el par n-X-tilting}\label{thm: el par n-X-tilting-1} 
For a class   $\mathcal{X}=\smdx[\mathcal{X}]\subseteq\mathcal{C}$ closed under extensions and an $n$-$\X$-tilting class $\mathcal{T}\subseteq\mathcal{C},$  the following statements hold
true.
\begin{itemize}
\item[$\mathrm{(a)}$]  ${}{}^{\bot}(\mathcal{T}^{\bot}\cap\mathcal{X})\cap\mathcal{X}={}^{\bot}(\mathcal{T}^{\bot})\cap\mathcal{X}=\mathcal{T}_{\mathcal{X}}^{\vee}\cap\mathcal{X}=\left(\mathcal{T}\cap\mathcal{X}\right)_{\mathcal{X}}^{\vee}\cap\mathcal{X}.$ 
\item[$\mathrm{(b)}$] $\mathcal{T}^{\bot}\cap\mathcal{X}=\Gennr[\mathcal{T}][k][\mathcal{X}]\cap\mathcal{X}$
$\forall k\geq\max\{1, \pdr[\mathcal{X}][\mathcal{T}]\}$.
\item[$\mathrm{(c)}$] The pair $({}{}^{\bot}(\mathcal{T}^{\bot}), \mathcal{T}^{\bot})$
is  $\mathcal{X}$-complete and hereditary.
\end{itemize}
\end{thm}
\begin{proof} (a) By Lemma \ref{lem: inf2} (a),  we get the inclusions 
\begin{center}
$\left(\mathcal{T}\cap\mathcal{X}\right)_{\mathcal{X}}^{\vee}\cap\mathcal{X}\subseteq\mathcal{T}_{\mathcal{X}}^{\vee}\cap\mathcal{X}\subseteq{}{}^{\bot}(\mathcal{T}^{\bot})\cap\mathcal{X}\subseteq{}{}^{\bot}(\mathcal{T}^{\bot}\cap\mathcal{X})\cap\mathcal{X}.$
\end{center}
 Consider $X\in{}{}^{\bot}(\mathcal{T}^{\bot}\cap\mathcal{X})\cap\mathcal{X}$.
From Lemma \ref{lem: inf6} and Proposition \ref{prop: equiv a t3},  we get an exact sequence
$\suc[X][M_{X}][C_{X}]$
with $M_{X}\in\mathcal{T}^{\bot}\cap\mathcal{X}$ and $C_{X}\in \mathcal{X}\cap(\mathcal{T}\cap\mathcal{X})_{\mathcal{X}}^{\vee}$.
Moreover,  $C_{X}\in{}{}^{\bot}\left(\mathcal{T}^{\bot}\cap\mathcal{X}\right)$
by Lemma \ref{lem: inf2} (a). Notice that $M_{X}\in{}{}^{\bot}\left(\mathcal{T}^{\bot}\cap\mathcal{X}\right)$
since $C_X,X\in{}{}^{\bot}\left(\mathcal{T}^{\bot}\cap\mathcal{X}\right)$.
By Proposition \ref{prop: equiv a t3} and Lemma \ref{lem: inf5} (b),  $\mathcal{T}\cap\mathcal{X}=\mathcal{T}^{\bot}\cap\X\cap{}{}^{\bot}\left(\mathcal{T}^{\bot}\cap\mathcal{X}\right)$
and thus  $M_{X}\in\mathcal{T}\cap\mathcal{X}$. Therefore,  $X\in\left(\mathcal{T}\cap\mathcal{X}\right)_{\mathcal{X}}^{\vee}\cap\mathcal{X}$.
\

(b) By Lemma \ref{lem: inf4} and Lemma \ref{lem: inf5} (a),  $\mathcal{T}\cap\mathcal{X}$
is a relative generator in $\mathcal{T}^{\bot}\cap\mathcal{X}$. Hence,  it follows that 
$\mathcal{T}^{\bot}\cap\mathcal{X}\subseteq\Gennr[\mathcal{T}][k][\mathcal{X}]\cap\mathcal{X}$
$\forall k\geq1$. Let $m: =\max\{1, \pdr[\mathcal{X}][\mathcal{T}]\}$.
We will show that $\Gennr[\mathcal{T}][k][\mathcal{X}]\cap\mathcal{X}\subseteq\mathcal{T}^{\bot}\cap\mathcal{X}$
$\forall k\geq m$. Consider $C\in\Gennr[\mathcal{T}][k][\mathcal{X}]\cap\mathcal{X}$
with $k\geq m$. By definition,  there is an exact sequence 
$\suc[K][T_{k}\stackrel{f_{k}}{\rightarrow}...\stackrel{f_{2}}{\rightarrow}T_{1}][C][][f_1]$ 
where $\Kerx[f_{i}]\in\mathcal{X}$ and $T_{i}\in\mathcal{T}\cap\mathcal{X}$
$\forall i\in[1, k]$. Then,  by  \cite[Prop. 2.6]{parte1}, 
it follows that $C\in\mathcal{T}^{\bot}\cap\mathcal{X}.$
\

(c) It is clear that the pair $({}{}^{\bot}(\mathcal{T}^{\bot}), \mathcal{T}^{\bot})$ is hereditary. Let us prove that it is $\mathcal{X}$-complete.
By Lemma \ref{lem: inf2} (a),  $(\mathcal{T}\cap\mathcal{X})_{\mathcal{X}}^{\vee}\cap\mathcal{X}\subseteq{}^{\bot}(\mathcal{T}^{\bot})\cap\mathcal{X}$.
On the other hand,  by Lemma \ref{lem: inf4} and  Lemma \ref{lem: inf6},  for each $X\in\X, $ 
there are  exact sequences $\suc[X][M_{X}][C_{X}]\mbox{ and }\suc[K_{X}][B_{X}][X]\mbox{, }$
where $M_{X}, K_{X}\in\mathcal{T}^{\bot}\cap\mathcal{X}$ and $C_{X}, B_{X}\in(\mathcal{T}\cap\mathcal{X})_{\mathcal{X}}^{\vee}\cap\mathcal{X};$ proving (c).
\end{proof}

The following result is a generalization of \cite[Thm. 4.3]{Wei}, 
\cite[Thm. 3.11]{Bazzonintilting}  and \cite[Thm. 4.4]{Tiltinginfinitamentegenerado}.

\begin{thm}\label{prop: primera generalizacion}\label{prop: primera generalizacion-1}
For   $n\geq1, $ $\mathcal{X}=\smdx[\mathcal{X}]\subseteq\mathcal{C}$
 closed under extensions and $\mathcal{T}=\smdx[\mathcal{T}]\subseteq\mathcal{C}$
satisfying $\mathrm{(T4)}$ and $\mathrm{(T5)}, $  the following statements are equivalent.
\begin{itemize}
\item[$\mathrm{(a)}$] $\mathcal{T}$ is $n$-$\mathcal{X}$-tilting.
\item[$\mathrm{(b)}$] $\mathcal{T}^{\bot}\cap\mathcal{X}=\Gennr\cap\mathcal{X}.$
\item[$\mathrm{(c)}$] $\mathcal{T}^{\bot}\cap\mathcal{X}=\Gennr[][k]\cap\mathcal{X}$ $\forall k\geq n$.
\item[$\mathrm{(d)}$] $\mathcal{T}^{\bot}\cap\mathcal{X}$ is closed by $n$-quotients in
$\mathcal{X}$ and $\mathcal{T}\cap\mathcal{X}\subseteq\mathcal{T}^{\bot}\cap\mathcal{X}\subseteq\Gennr[][1]$.
\end{itemize}
\end{thm}
\begin{proof} 
 (a) $\Rightarrow$ (b)  It follows from Theorem \ref{thm: el par n-X-tilting} (b).
\

(b) $\Rightarrow$ (c) It is enough to prove that $\Gennr[\mathcal{T}][n+1][\mathcal{X}]\cap\mathcal{X}\supseteq\Gennr[\mathcal{T}][n][\mathcal{X}]\cap\mathcal{X}$.
Let $N\in\Gennr[\mathcal{T}][n][\mathcal{X}]\cap\mathcal{X}=\mathcal{T}^{\bot}\cap\mathcal{X}.$ Then,  by (T5),    there is a 
$\mathcal{T}$-precover $f: A\rightarrow N$ with $A\in\mathcal{X}.$ Moreover, since $\Gennr[\mathcal{T}][n]\subseteq\Gennr[\mathcal{T}][1]$,  we have the exact sequence 
$\eta: \; \suc[K][A][N][][f]$, where $A\in\T\cap\X\subseteq \Gennr[\mathcal{T}][n][\mathcal{X}]\cap\mathcal{X}=\mathcal{T}^{\bot}\cap\mathcal{X}$. Using that $A, N\in\T^{\perp}$ and that $f$ is a $\T$-precover, we get that $K\in\mathcal{T}^{\bot}.$ Let us prove that $K\in \X.$ Notice that there is an exact sequence 
$\eta': \; \suc[K'][M_0][N][][f'] $, where $M_0\in \T\cap\X$ and $K'\in\X.$ Then,  from the pull-back construction of $f$ and $f', $ we  get an exact sequence $\eta'': \; \suc[K][P][M_0]$, where $P\in\X$ since $\X$ is closed under extensions. Notice that $\eta''$ splits since $M_0\in\T$ and $K\in\T^\perp$ and thus $K\in\X.$ Therefore $K\in\T^\perp\cap\X=\Gennr\cap\mathcal{X}$ and from the exact sequence $\eta, $ it follows that $N\in \Gennr[\mathcal{T}][n+1][\mathcal{X}]\cap\mathcal{X}.$
\

(c) $\Rightarrow$ (d)  By (c),  we know that $\mathcal{T}\cap\mathcal{X}\subseteq\Gennr[\mathcal{T}][n]\cap\mathcal{X}=\mathcal{T}^{\bot}\cap\mathcal{X}\subseteq\Gennr[][1]\mbox{.}$
Since 
$\mathcal{T}^{\bot}\cap\mathcal{X}=\Gennr[\mathcal{T}][n]\cap\mathcal{X}=\Gennr[\mathcal{T}][n+1]\cap\mathcal{X}\mbox{, }$ from 
\cite[Prop. 5.2]{parte1},  we get that  
 $\mathcal{T}^{\bot}\cap\mathcal{X}$ is closed by $n$-quotients
in $\mathcal{X}.$  
\

(d) $\Rightarrow$ (a) Since $\mathcal{T}^{\bot}\cap\mathcal{X}$ is
closed by $n$-quotients in $\mathcal{X}$ and $\mathrm{(T4)}$ holds true,  it follows from \cite[Prop. 2.6]{parte1},   that $\pdr[\mathcal{X}][\mathcal{T}]\leq n$ and thus (T1) holds true.  Furthermore,  by (d),  we have that 
 (T2) holds true and  $\mathcal{T}^{\bot}\cap\mathcal{X}\subseteq\Gennr[][1]\cap\mathcal{X}$.
Therefore,  by Proposition \ref{prop: equiv a t3},  we conclude that $\mathcal{T}$
is $n$-$\mathcal{X}$-tilting.
\end{proof}

As an easy consequence of Theorem \ref{prop: primera generalizacion},  we can give an equivalent
condition of (T5) in case $\mathcal{T}\subseteq\mathcal{X}.$

\begin{cor} \label{prop: primera generalizacion-1-1} 
Let   $n\geq1$,  $\mathcal{X}=\smdx[\mathcal{X}]\subseteq\mathcal{C}$
closed under extensions,  and let $\mathcal{T}=\smdx[\mathcal{T}]\subseteq\mathcal{X}$
satisfying $\mathrm{(T1),   (T2),   (T3)}$ and $\mathrm{(T4)}.$ Then,  the following statements
are equivalent.
\begin{itemize}
\item[$\mathrm{(a)}$] $\mathcal{T}$ is $n$-$\X$-tilting.
\item[$\mathrm{(b)}$] $\mathcal{T}^{\bot}\cap\mathcal{X}=\Gennr[\mathcal{T}]\cap\mathcal{X}=\Gennr[\mathcal{T}][n+1]\cap\mathcal{X}.$
\end{itemize}
\end{cor}

\subsection{$n$-$\mathcal{X}$-tilting classes and relative dimensions}

\begin{prop}\label{prop: (a)}\label{prop: (a)-1} 
Let   $\mathcal{X}=\smdx[\mathcal{X}]\subseteq\mathcal{C}$ be a class 
closed under extensions and $\mathcal{T}\subseteq\mathcal{C}$ be an
$n$-$\mathcal{X}$-tilting class. Then,  the pair $\p: =({}{}^{\bot}(\T^{\bot}), \T^{\bot})$
and the class $\nu: =\mathcal{A}\cap\mathcal{B}\cap\mathcal{X}$ satisfy that
$\nu$ is a relative $\mathcal{B}\cap\mathcal{X}$-projective generator
in $\mathcal{B}\cap\mathcal{X}$ and a relative $\mathcal{A}\cap\mathcal{X}$-injective
cogenerator in $\mathcal{A}\cap\mathcal{X}$. Furthermore,  the following statements hold true.
\begin{itemize}
\item[$\mathrm{(a)}$] $\nu$ $=\mathcal{A}\cap\mathcal{X}\cap\left(\mathcal{A}\cap\mathcal{X}\right)^{\bot}$ $=\mathcal{B}\cap\mathcal{X}\cap{}^{\bot}(\mathcal{B}\cap\mathcal{X})=(\nu, \mathcal{A}\cap\mathcal{X})^{\wedge}$ $=(\mathcal{B}\cap\mathcal{X}, \nu)^{\vee}=$
$=\mathcal{A}\cap\mathcal{X}\cap\nu^{\wedge}$ $=\mathcal{B}\cap\mathcal{X}\cap\nu^{\vee}.$
\item[$\mathrm{(b)}$] $\mathcal{X}\subseteq(\mathcal{B}\cap\mathcal{X})_{\mathcal{X}}^{\vee}\subseteq(\mathcal{B}\cap\mathcal{X})^{\vee}.$
\item[$\mathrm{(c)}$] $(\mathcal{T}\cap\mathcal{X})_{\mathcal{X}}^{\vee}\subseteq(\mathcal{T}\cap\mathcal{X})^{\vee}\subseteq{}^{\bot}(\mathcal{B}\cap\mathcal{X}).$
\item[$\mathrm{(d)}$] $\mathcal{T}\cap\mathcal{X}=(\mathcal{T}\cap\mathcal{X})_{\mathcal{X}}^{\vee}\cap\mathcal{B}\cap\mathcal{X}=(\mathcal{T}\cap\mathcal{X})^{\vee}\cap\mathcal{B}\cap\mathcal{X}=\nu=$
$(\mathcal{T}\cap\mathcal{X})_{\mathcal{X}}^{\vee}\cap\mathcal{B}.$
\item[$\mathrm{(e)}$] $\mathcal{A}\cap(\mathcal{X}, \nu)^{\vee}=\mathcal{A}\cap\mathcal{X}.$
\item[$\mathrm{(f)}$] 
$\mathcal{B}\cap(\nu, \mathcal{X})^{\wedge}=\left\{ M\in\mathcal{B}\cap\mathcal{X}\, |\: \pdr[\mathcal{B}\cap\mathcal{X}][M]<\infty\right\} $.
\end{itemize}
\end{prop}
\begin{proof}
Notice that $\A$ and $\B$ are closed under extensions and direct summands. By Theorem \ref{thm: el par n-X-tilting} (c), 
 it follows that the pair  $\p$ is $\mathcal{X}$-hereditary and $\mathcal{X}$-complete. Therefore,  by \cite[Thm. 4.24 (a, b)]{parte1}  we get (a). Moreover, 
by \cite[Prop. 4.23 (a,  b)]{parte1},  it follows  that $\nu$
is a relative $\mathcal{B}\cap\mathcal{X}$-projective generator in
$\mathcal{B}\cap\mathcal{X}$ and a relative $\mathcal{A}\cap\mathcal{X}$-injective
cogenerator in $\mathcal{A}\cap\mathcal{X}.$
\

The items  (b) and (c) follow from Lemma \ref{lem: inf6},  and the item (d) follows from putting together  Lemma \ref{lem: inf2} (b),  the item (c)
 and Lemma \ref{lem: inf5} (b). 
 \
 
Let us prove (e). Indeed, by \cite[Lem. 4.17 (a)]{parte1}, 
we know that 
\begin{center}
$\mathcal{A}\cap(\mathcal{X}, \nu)^{\vee}=\left\{ M\in\mathcal{A}\cap\mathcal{X}\, |\: \idr[\mathcal{A}\cap\mathcal{X}][M]<\infty\right\} \mbox{.}$
\end{center}
Now,  from (T4) and \cite[Prop. 4.5 (a)]{parte1}, 
$\coresdimr{\mathcal{B}\cap\mathcal{X}}{\mathcal{X}}{\mathcal{X}}\leq\pdr[\mathcal{X}][T]<\infty$;
and thus, by \cite[Thm. 4.24 (a)]{parte1}, 
$\pdr[\mathcal{A}\cap\mathcal{X}][\mathcal{A}\cap\mathcal{X}]=\coresdimr{\mathcal{B}\cap\mathcal{X}}{\mathcal{X}}{\mathcal{X}}<\infty$.
Therefore $\mathcal{A}\cap(\mathcal{X}, \nu)^{\vee}=\mathcal{A}\cap\mathcal{X}.$ 
Finally,  the item (f) follows from the dual result of \cite[Lem. 4.17 (a)]{parte1}.
\end{proof}

\begin{prop}\label{prop: (b)}\label{prop: (b)-1} 
Let   $\mathcal{X}=\smdx[\mathcal{X}]\subseteq\mathcal{C}$ be
closed under extensions,  and let $\mathcal{T}\subseteq\mathcal{C}$ be $n$-$\mathcal{X}$-tilting.
Then,  for the pair $\p: =({}{}^{\bot}(\T^{\bot}), \T^{\bot}), $  it follows that 
$^{\bot}(\mathcal{B}\cap\mathcal{X})\cap\mathcal{X}=\mathcal{A}\cap\mathcal{X}$
and $(\mathcal{A}\cap\mathcal{X})^{\bot}\cap\mathcal{X}=\mathcal{B}\cap\mathcal{X}.$ Moreover,  the following statements hold true.
\begin{itemize}
\item[$\mathrm{(a)}$] For any $X\in\mathcal{X}, $ we have that
  \begin{itemize}
  \item[$\mathrm{(a1)}$] $\resdimr{\mathcal{A}}X{\mathcal{X}}$  $= \resdimr{\mathcal{A}\cap\mathcal{X}}X{\mathcal{X}}$
 $= \resdimr{\mathcal{A}\cap\mathcal{X}}X{\, }$  $=\resdimr{\mathcal{A}}X{\, }$ $=$
 $=\pdr[\mathcal{B}\cap\mathcal{X}][X]$ $\leq\resdimr{\mathcal{A}\cap\mathcal{X}}{\mathcal{B}\cap\mathcal{X}}{\, }+1\mbox{;}$
  \item[$\mathrm{(a2)}$] $\resdimr{\mathcal{A}}X{\, }\leq\pdr[\mathcal{B}][X]\leq\resdimr{\mathcal{A}}{\mathcal{B}}{\, }+1;$
  \item[$\mathrm{(a3)}$] $\idr[\mathcal{A}\cap\mathcal{X}][X]=$ $\coresdimr{\mathcal{B}}X{\mathcal{X}}=$
$\coresdimr{\mathcal{B}\cap\mathcal{X}}X{\mathcal{X}}=$ $\coresdimr{\mathcal{B}\cap\mathcal{X}}X{\, }=$
$=\coresdimr{\mathcal{B}}X{\, }$ $\leq\coresdimr{\mathcal{B}\cap\mathcal{X}}{\mathcal{A}\cap\mathcal{X}}{\, }+1\mbox{;}$
  \item[$\mathrm{(a4)}$] $\coresdimx{\mathcal{B}}X\leq\idr[\mathcal{A}][X]\leq\coresdimr{\mathcal{B}}{\mathcal{A}}{\, }+1.$
  \end{itemize}
\item[$\mathrm{(b)}$] $\coresdimr{\mathcal{B}\cap\mathcal{X}}{\mathcal{X}}{\mathcal{X}}\leq\pdr[\mathcal{X}][\mathcal{T}]=\pdr[\mathcal{X}][\mathcal{A}]=\pdr[\mathcal{X}][^{\bot}(\mathcal{B}\cap\mathcal{X})]<\infty$.
\end{itemize}
\end{prop}
\begin{proof}
Notice that $\p$ is $\mathcal{X}$-hereditary and $\mathcal{X}$-complete
by Theorem \ref{thm: el par n-X-tilting} (c). Moreover,  by Proposition \ref{prop: (a)} (a),   $\left(\mathcal{A}\cap\mathcal{X}\right)^{\bot}\cap\mathcal{A}\cap\mathcal{X}\subseteq\mathcal{B}\cap\mathcal{X}$ and ${}^\perp(\B\cap\X)\cap\B\cap\X\subseteq\A\cap\X.$ 
Thereupon,  by \cite[Prop. 4.11(e)]{parte1} and its dual,  
$(\mathcal{A}\cap\mathcal{X})^{\bot}\cap\mathcal{X}\subseteq\mathcal{B}\subseteq(\mathcal{A}\cap\mathcal{X})^{\bot}$ and ${}^{\bot}(\mathcal{B}\cap\mathcal{X})\cap\mathcal{X}\subseteq\mathcal{A}\subseteq{}^{\bot}(\mathcal{B}\cap\mathcal{X})$.
Hence,  $^{\bot}(\mathcal{B}\cap\mathcal{X})\cap\mathcal{X}=\mathcal{A}\cap\mathcal{X}$
and $(\mathcal{A}\cap\mathcal{X})^{\bot}\cap\mathcal{X}=\mathcal{B}\cap\mathcal{X}$.
It remains to prove (a) and (b). Indeed,  the item (a) follows from \cite[Prop. 4.11]{parte1} and its dual. Finally,  the item (b)
follows from \cite[Prop. 4.5]{parte1} since $\T\subseteq\A.$
\end{proof}

By Theorem  \ref{thm: el par n-X-tilting} (c) and \cite[Prop. 4.23]{parte1}, we get the following result.

\begin{cor}\label{prop: oct1}\label{prop: oct1-1}
 Let   $\mathcal{X}=\smdx[\mathcal{X}]\subseteq\mathcal{C}$ be
closed under extensions and let $\mathcal{T}\subseteq\mathcal{C}$ be $n$-$\mathcal{X}$-tilting.
Then,  for the pair $\p: =({}{}^{\bot}(\mathcal{T}^{\bot}), \mathcal{T}^{\bot})$
and the class $\nu: =\mathcal{A}\cap\mathcal{B}\cap\mathcal{X}, $ the following statements hold true.
\begin{itemize}
\item[$\mathrm{(a)}$] $\pdr[\mathcal{B}\cap\mathcal{X}][M]=\pdr[\nu][M]=\pdr[\nu^{\wedge}][M]=\resdimr{\mathcal{A}}M{\mathcal{X}}=\resdimx{\mathcal{X}\cap\mathcal{A}}M$
$\forall M\in\left(\mathcal{A}, \mathcal{X}\right)^{\wedge}$.
\item[$\mathrm{(b)}$] $\pdr[\mathcal{B}\cap\mathcal{X}][M]=\resdimr{\nu}M{\mathcal{X}}$
$\forall M\in(\nu, \mathcal{X})^{\wedge}$.
\item[$\mathrm{(c)}$] $\idr[\mathcal{A}\cap\mathcal{X}][M]=\idr[\nu][M]=\idr[\nu^{\vee}][M]=\coresdimr{\mathcal{B\cap\mathcal{X}}}M{\, }=\coresdimr{\mathcal{B}}M{\mathcal{X}}$
$\forall M\in(\mathcal{X}, \mathcal{B})^{\vee}$.
\item[$\mathrm{(d)}$] $\idr[\mathcal{A}\cap\mathcal{X}][M]=\coresdimr{\nu}M{\mathcal{X}}$
$\forall M\in(\mathcal{X}, \nu)^{\vee}.$
\end{itemize}
\end{cor}

\begin{prop}\label{prop: oct2} 
Let   $\mathcal{X}=\smdx[\mathcal{X}]\subseteq\mathcal{C}$
be closed under extensions,  and let $\mathcal{T}\subseteq\mathcal{C}$ be
$n$-$\mathcal{X}$-tilting. Then,  for the pair $\p: =({}{}^{\bot}(\mathcal{T}^{\bot}), \mathcal{T}^{\bot})$
and the class $\nu: =\mathcal{A}\cap\mathcal{B}\cap\mathcal{X}, $ we have that
$\mathcal{A}\cap\mathcal{X}\subseteq\nu^{\vee}$. Moreover,  the following statements hold. 
\begin{itemize}
\item[$\mathrm{(a)}$] $\pdr[\mathcal{X}][\nu]$ $=\coresdimx{\mathcal{B}\cap\mathcal{X}}{\mathcal{A}\cap\mathcal{X}}$
$=\coresdimr{\mathcal{B}\cap\mathcal{X}}{\mathcal{A}\cap\mathcal{X}}{\mathcal{X}}$
 $=\coresdimr{\mathcal{B}}{\mathcal{X}}{\mathcal{X}}$
$=\coresdimx{\mathcal{B}\cap\mathcal{X}}{\mathcal{X}}$
$=\coresdimr{\mathcal{B}\cap\mathcal{X}}{\mathcal{X}}{\mathcal{X}}$
$=\pdr[\mathcal{A\cap\mathcal{X}}][\mathcal{A}\cap\mathcal{X}]$ 
$=\pdr[\mathcal{\mathcal{X}}][\mathcal{A}\cap\mathcal{X}]$
$=\coresdimx{\nu}{\mathcal{A}\cap\mathcal{X}}$
 $=\coresdimr{\mathcal{B}}{\mathcal{A}\cap\mathcal{X}}{\mathcal{X}}$
$\leq\pdr[\mathcal{X}][\mathcal{T}]<\infty\mbox{.}$ 
\item[$\mathrm{(b)}$] $\idr[\mathcal{X}][\nu]\leq$ $\resdimx{\mathcal{A}\cap\mathcal{X}}{\mathcal{B}\cap\mathcal{X}}$
$=\resdimr{\mathcal{A}\cap\mathcal{X}}{\mathcal{B}\cap\mathcal{X}}{\mathcal{X}}$
$=\resdimx{\mathcal{A}\cap\mathcal{X}}{\mathcal{B}\cap\mathcal{X}}$
$=\resdimr{\mathcal{A}\cap\mathcal{X}}{\mathcal{X}}{\mathcal{X}}$
$=\resdimr{\mathcal{A}}{\mathcal{X}}{\mathcal{X}}$ $=\resdimx{\nu}{\mathcal{B}\cap\mathcal{X}}$
$=\idr[\mathcal{B}\cap\mathcal{X}][\mathcal{B}\cap\mathcal{X}]$ $=\idr[\mathcal{X}][\mathcal{B}\cap\mathcal{X}]$
$=\resdimr{\mathcal{A}}{\mathcal{B}\cap\mathcal{X}}{\mathcal{X}}$.
\item[$\mathrm{(c)}$] $\idr[\mathcal{X}][\mathcal{B}\cap\mathcal{X}]<\infty$ if and only
if $\mathcal{B}\cap\mathcal{X}\subseteq\nu^{\wedge}$ and $\idr[\mathcal{X}][\nu]<\infty$.
Furthermore,  if $\idr[\mathcal{X}][\mathcal{B}\cap\mathcal{X}]<\infty, $ then
$\mathcal{B}\cap(\nu, \mathcal{X})^{\wedge}=\mathcal{B}\cap\mathcal{X}\mbox{,  }\mathcal{X}\subseteq(\mathcal{A}, \mathcal{X})^{\wedge}\subseteq\left(\mathcal{A}\cap\mathcal{X}\right)^{\wedge}\mbox{ and }\idr[\mathcal{X}][\mathcal{B}\cap\mathcal{X}]=\idr[\mathcal{X}][\nu]\mbox{.}$
\end{itemize}
\end{prop}
\begin{proof}
By Theorem \ref{thm: el par n-X-tilting} (c),  we know that the pair $\p$ is hereditary and $\mathcal{X}$-complete. In order to prove (a),  observe  that,   by Proposition \ref{prop: (b)} (b) and \cite[Thm. 4.24 (a1)]{parte1}, 
it follows that 
\[
\pdr[\mathcal{A}\cap\mathcal{X}][\mathcal{A}\cap\mathcal{X}]=\coresdimr{\mathcal{B}\cap\mathcal{X}}{\mathcal{X}}{\mathcal{X}}\leq\pdr[\mathcal{X}][\T]<\infty\mbox{.}
\]
Then,  by \cite[Thm 4.24 (a2)]{parte1},   $\mathcal{A}\cap\mathcal{X}\subseteq\nu^{\vee}$ and
$\pdr[\mathcal{X}][\mathcal{A}\cap\mathcal{X}]=\pdr[\mathcal{X}][\nu]$.
The rest of the equalities appearing in (a) follow from \cite[Thm. 4.24 (a1)]{parte1}.
\

Except for the equality $\mathcal{B}\cap(\nu, \mathcal{X})^{\wedge}=\mathcal{B}\cap\mathcal{X}$ in (c) (under the hypothesis that 
$\idr[\mathcal{X}][\mathcal{B}\cap\mathcal{X}]<\infty$),  the items (b) and (c) follow from \cite[Thm. 4.24(b)]{parte1}. Let us prove such equality. Indeed,  assume that $\idr[\mathcal{X}][\mathcal{B}\cap\mathcal{X}]<\infty.$ Then,   by (b),  we have that 
$\pd_{\B\cap\X}(\B\cap\X)=\id_{\B\cap\X}(\B\cap\X)<\infty.$ Hence,  from Proposition \ref{prop: (a)} (f),  the required equality follows.
\end{proof}

\begin{prop}\label{prop: oct3}\label{prop: oct3-1} 
Let   $\mathcal{X}=\smdx[\mathcal{X}]\subseteq\mathcal{C}$ be
closed under extensions and let $\mathcal{T}\subseteq\mathcal{C}$ be $n$-$\mathcal{X}$-tilting.
Then,  for the pair $\p: =({}{}^{\bot}(\mathcal{T}^{\bot}), \mathcal{T}^{\bot})$
and the class $\nu: =\mathcal{A}\cap\mathcal{B}\cap\mathcal{X}, $  the
following statements hold true.
\begin{itemize}
\item[$\mathrm{(a)}$] $\pdr[\mathcal{X}][\mathcal{X}]=\pdr[\mathcal{X}][(\mathcal{B}\cap\mathcal{X})_{\mathcal{X}}^{\vee}]=\pdr[\mathcal{X}][(\mathcal{B}\cap\mathcal{X})^{\vee}]=\pdr[\mathcal{X}][\mathcal{B}\cap\mathcal{X}]$.
\item[$\mathrm{(b)}$] $\nu=\T\cap\X.$ Moreover,  if $\mathcal{T}\subseteq\mathcal{X}, $
then 
$$\pdr[\mathcal{X}][\mathcal{T}]=\pdr[\mathcal{X}][\nu]=\pdr[\mathcal{X}][\nu{}_{\mathcal{X}}^{\vee}]=\pdr[\mathcal{X}][\nu{}^{\vee}].$$
\end{itemize}
\end{prop}
\begin{proof} 
We point out that,  by Theorem \ref{thm: el par n-X-tilting} (c),  the pair $\p$ is hereditary and $\mathcal{X}$-complete. Then,  (a) follows from Proposition \ref{prop: (a)} (b) and \cite[Lem. 4.3]{parte1}. Finally,  (b) can be obtained from Proposition \ref{prop: (a)} (c, d) and Proposition \ref{prop: (b)} (b).
\end{proof}

\begin{prop}\label{prop: M ortogonal es preenvolvente esp en X}\label{prop: M ortogonal es preenvolvente esp en X-1}
Let   $\mathcal{X}=\smdx[\mathcal{X}]\subseteq\mathcal{C}$
be closed under extensions and let $\mathcal{T}\subseteq\mathcal{C}$ be
$n$-$\mathcal{X}$-tilting. Then,  for the pair $\p: =({}{}^{\bot}(\T^{\bot}), \T^{\bot})$
and the class $\nu: =\mathcal{A}\cap\mathcal{B}\cap\mathcal{X}, $ the following statements hold true.
\begin{itemize}
\item[$\mathrm{(a)}$] $\nu^{\vee}\subseteq{}^{\bot}(\mathcal{B}\cap\mathcal{X})$ and,  for any $Z\in(\mathcal{B}\cap\mathcal{X})_{\mathcal{X}}^{\vee}$ and  $m: =\coresdimr{\mathcal{B}\cap\mathcal{X}}Z{\mathcal{X}}$,  there
are short exact sequences 
\begin{alignat*}{1}
\suc[Z][M_{Z}][C_{Z}][g_{Z}] & \qquad\mbox{with \ensuremath{C_{Z}\in(\mathcal{X}, \nu){}^{\vee}, \, M_{Z}\in\mathcal{B}\cap\mathcal{X}}}\mbox{, }\\
\suc[K_{Z}][N_{Z}][Z][][f_{Z}] & \qquad\mbox{with \ensuremath{N_{Z}\in\nu_{\mathcal{\mathcal{X}}}^{\vee}},  \ensuremath{K_{Z}\in\mathcal{B}\cap\mathcal{X}}}\mbox{, }
\end{alignat*}
such that $g_{Z}$ is a $\mathcal{B}\cap\mathcal{X}$-preenvelope
and $f_{Z}$ is an $\nu^{\vee}$-precover. Furthermore,  $\coresdimr{\nu}{C_{Z}}{\mathcal{X}}=m-1$, 
$\coresdimr{\nu}{N_{Z}}{\mathcal{X}}\leq m$,  and 
\begin{center}
$\nu^{\vee}\cap\mathcal{X}=\mathcal{A}\cap\mathcal{X}=\mathcal{A}\cap(\mathcal{X}, \nu)^{\vee}.$
\end{center}

\item[$\mathrm{(b)}$] $\nu^{\wedge}\subseteq{}(\mathcal{A}\cap\mathcal{X})^{\bot}$ and,  for any 
$Z\in(\mathcal{A}\cap\mathcal{X})_{\mathcal{X}}^{\wedge}$ and
 $m: =\resdimr{\mathcal{A}\cap\mathcal{X}}Z{\mathcal{X}}$,  there
are short exact sequences 
\begin{alignat*}{1}
\suc[Z][N_{Z}][C_{Z}][f_{Z}] & \qquad\mbox{with \ensuremath{N_{Z}\in\nu_{\mathcal{X}}^{\wedge}, \, C_{Z}\in\mathcal{A}\cap\mathcal{X}}}\mbox{, }\\
\suc[K_{Z}][M_{Z}][Z][][g_{Z}] & \qquad\mbox{with \ensuremath{K_{Z}\in(\nu, \mathcal{X}){}^{\wedge}},  \ensuremath{M_{Z}\in\mathcal{A}\cap\mathcal{X}}}\mbox{, }
\end{alignat*}
such that $g_{Z}$ is a $\mathcal{A}\cap\mathcal{X}$-precover and
$f_{Z}$ is a $\nu^{\wedge}$-preenvelope. Furthermore,  $\resdimr{\nu}{K_{Z}}{\X}=m-1$, 
$\resdimr{\nu}{N_{Z}}{\X}\leq m, $ and
\begin{center}
$\nu^{\wedge}\cap\mathcal{X}=\mathcal{B}\cap\mathcal{X}=\mathcal{B}\cap(\nu, \mathcal{X})^{\wedge}\;$ if $\idr[\mathcal{X}][\mathcal{B}\cap\mathcal{X}]<\infty.$
\end{center}

\item[$\mathrm{(c)}$]  For any $Z\in(\mathcal{B}\cap\mathcal{X})^{\vee}$ and $m: =\coresdimr{\mathcal{B}\cap\mathcal{X}}Z{\, }$, 
there are short exact sequences 
\begin{alignat*}{1}
\suc[Z][M_{Z}][C_{Z}][g_{Z}] & \qquad\mbox{with \ensuremath{C_{Z}\in(\mathcal{X}, \nu){}^{\vee}, \, M_{Z}\in\mathcal{B}\cap\mathcal{X}}}\mbox{, }\\
\suc[K_{Z}][N_{Z}][Z][][f_{Z}] & \qquad\mbox{with \ensuremath{N_{Z}\in\nu{}^{\vee}},  \ensuremath{K_{Z}\in\mathcal{B}\cap\mathcal{X}}}\mbox{, }
\end{alignat*}
such that $g_{Z}$ is a $\mathcal{B}\cap\mathcal{X}$-preenvelope
and $f_{Z}$ is a $\nu^{\vee}$-precover. Furthermore,  $\coresdimr{\nu}{C_{Z}}{\mathcal{X}}=m-1$
and $\coresdimr{\nu}{N_{Z}}{\mathcal{X}}\leq m$.

\item[$\mathrm{(d)}$]  $\forall Z\in(\mathcal{A}\cap\mathcal{X})^{\wedge}$,  with $m: =\resdimr{\mathcal{A}\cap\mathcal{X}}Z{\, }$, 
there are short exact sequences 
\begin{alignat*}{1}
\suc[Z][N_{Z}][C_{Z}][f_{Z}] & \qquad\mbox{with \ensuremath{N_{Z}\in\nu^{\wedge}, \, C_{Z}\in\mathcal{A}\cap\mathcal{X}}}\mbox{, }\\
\suc[K_{Z}][M_{Z}][Z][][g_{Z}] & \qquad\mbox{with \ensuremath{K_{Z}\in\nu{}^{\wedge}},  \ensuremath{M_{Z}\in\mathcal{A}\cap\mathcal{X}}}\mbox{, }
\end{alignat*}
such that $g_{Z}$ is an $\mathcal{A}\cap\mathcal{X}$-precover and
$f_{Z}$ is an $\nu^{\wedge}$-preenvelope. Furthermore,  $\resdimr{\nu}{K_{Z}}{\, }=m-1$
and $\resdimr{\nu}{N_{Z}}{\, }\leq m.$

\item[$\mathrm{(e)}$]  The pair $(\nu^{\vee}, \mathcal{B})$ is right $(\mathcal{B}\cap\mathcal{X})_{\mathcal{X}}^{\vee}$-complete, 
$\mathcal{X}$-complete  and $\mathcal{X}$-hereditary.

\item[$\mathrm{(f)}$]  $\mathcal{B}\cap\mathcal{X}$ is special preenveloping in $(\mathcal{B}\cap\mathcal{X})_{\mathcal{X}}^{\vee}$
and in $\mathcal{X}$. Moreover,  the pair $({}^{\bot}(\mathcal{B}\cap\mathcal{X}), \mathcal{B}\cap\mathcal{X})$
is right $\mathcal{X}$-complete. 

\item[$\mathrm{(g)}$]  Any object of $(\mathcal{B}\cap\mathcal{X})^{\vee}$ admits a special
$\mathcal{B}\cap\mathcal{X}$-preenvelope.
\end{itemize}
\end{prop}
\begin{proof} (a) Consider the pair $(\B\cap\X, \nu).$  Then,  by Proposition \ref{prop: (a)},   $\nu$ is a relative $\B\cap\X$-projective 
 generator in $\B\cap\X.$ In particular,  by  \cite[Lem. 4.3]{parte1},  we have that  $\nu^\vee\subseteq {}^\perp(\B\cap\X).$ Hence,  from \cite[Thm. 4.4]{parte1},  we  almost get  the item (a), remaining to show the equalities $\nu^{\vee}\cap\mathcal{X}=\mathcal{A}\cap\mathcal{X}=\mathcal{A}\cap(\mathcal{X}, \nu)^{\vee}.$ However, using that $\nu^\vee\subseteq {}^\perp(\B\cap\X), $   these equalities follow from Propositions \ref{prop: (b)} and  \ref{prop: (a)} (e).  
\

(b) It follows as in (a) by using Proposition \ref{prop: (a)},  the dual of \cite[Lem. 4.3]{parte1}, the dual of  \cite[Thm 4.4]{parte1},   Proposition \ref{prop: oct2} (c) and Proposition  \ref{prop: (b)}.
\

(c)  It can be proved by following similar arguments as we did in (a).
\

(d) It can be proved by following similar arguments as we did in (b).
\

(e) It follows from (a),  Proposition \ref{prop: (a)} (b) and Lemma  \ref{lem: inf6}.
\

(f) Let $X\in(\mathcal{B}\cap\mathcal{X})_{\mathcal{X}}^{\vee}$. By (a), 
there is an exact sequence 
$\suc[X][M_{X}][C_{X}][g_{X}]\mbox{, }$
with $C_{X}\in(\mathcal{X}, \nu)^{\vee}$,  $M_{X}\in\mathcal{B}\cap\mathcal{X}$, 
 and $g_{X}$ a $\mathcal{B}\cap\mathcal{X}$-preenvelope. Thereupon, 
the following statements are easy to prove. First,  $\mathcal{B}\cap\mathcal{X}$
is special preenveloping in $(\mathcal{B}\cap\mathcal{X})_{\mathcal{X}}^{\vee}$
since $M_{X}\in\mathcal{B}\cap\mathcal{X}=\mathcal{B}\cap\mathcal{X}\cap(\mathcal{B}\cap\mathcal{X})_{\mathcal{X}}^{\vee}$
and $C_{X}\in{}^{\bot_{1}}(\mathcal{B}\cap\mathcal{X})\cap(\mathcal{B}\cap\mathcal{X})_{\mathcal{X}}^{\vee}$
by Proposition \ref{prop: (a)}; and second,  $\mathcal{B}\cap\mathcal{X}$ is special
preenveloping in $\mathcal{X}$ since $\mathcal{X}\subseteq(\mathcal{B}\cap\mathcal{X})_{\mathcal{X}}^{\vee}$ 
(see Proposition \ref{prop: (a)} (b)),    $M_{X}\in\mathcal{B}\cap\mathcal{X}=\mathcal{X}\cap{}(\mathcal{B}\cap\mathcal{X})$
and $C_{X}\in\mathcal{X}\cap{}{}^{\bot_{1}}(\mathcal{B}\cap\mathcal{X})$.
\

(g) Let $X\in(\mathcal{B}\cap\mathcal{X})^{\vee}$. Consider the exact
sequence given by (c),  
$\suc[X][M_{X}][C_{X}][g_{X}]$
 with $M_{X}\in\mathcal{B}\cap\mathcal{X}$ and $C_{X}\in\nu^{\vee}$.
Then $g_{X}$ is a special $\mathcal{B}\cap\mathcal{X}$-preenvelope
since $M_{X}\in\mathcal{B}\cap\mathcal{X}$ and
$C_{X}\in(\T\cap\mathcal{X})^{\vee}\subseteq{}^{\bot}(\mathcal{B}\cap\mathcal{X})\subseteq{}^{\bot_{1}}(\mathcal{B}\cap\mathcal{X})$  by Proposition \ref{prop: (a)}(b).
\end{proof}

Next,  in a similar way as Lemma \ref{lem: props C2 y T2} and Proposition \ref{prop: M ortogonal es preenvolvente esp en X}, 
we will show the behaviour of the  pairs $\p$ such that $\mathcal{B}\cap\mathcal{X}=\mathcal{T}^{\bot}\cap\mathcal{X}$, 
where $\mathcal{T}$ is $n$-$\mathcal{X}$-tilting.

\begin{prop}\label{lem: partiltingescompleto}  For a class   $\mathcal{X}=\smdx[\mathcal{X}]\subseteq\mathcal{C}$  closed under extensions,  an $n$-$\mathcal{X}$-tilting  $\mathcal{T}\subseteq\mathcal{C}$  and a pair $\p$ in $\C$ such that $\mathcal{B}\cap\mathcal{X}=\mathcal{T}^{\bot}\cap\mathcal{X}, $ the following statements hold true.
\begin{itemize}
\item[$\mathrm{(a)}$] Let $\Extx[1][\mathcal{C}][\mathcal{A}][\mathcal{B}\cap\mathcal{X}]=0.$ Then any morphism $A\rightarrow X, $ with $A\in\mathcal{A}$ and $X\in\mathcal{B}\cap\mathcal{X}$, 
factors through $\mathcal{T}\cap\mathcal{X}$.
\item[$\mathrm{(b)}$]  $\mathcal{A}\cap\mathcal{B}\cap\mathcal{X}\subseteq{}^{\bot}(\mathcal{B}\cap\mathcal{X})\cap\mathcal{X}\cap\mathcal{B}=\mathcal{T}\cap\mathcal{X}$ if $\Extx[1][\mathcal{C}][\mathcal{A}][\mathcal{B}\cap\mathcal{X}]=0.$
\item[$\mathrm{(c)}$] Let $^{\bot_{1}}\mathcal{B}\cap\mathcal{X}\subseteq\mathcal{A}$ and
$\idr[\mathcal{A}][\mathcal{B}\cap\mathcal{X}]=0.$ Then,  the following
conditions are equivalent: 
  \begin{itemize}
\item[$\mathrm{(c1)}$] $\mathcal{T}\cap\mathcal{X}=\mathcal{A}\cap\mathcal{B}\cap\mathcal{X}$;
\item[$\mathrm{(c2)}$] $\p$ is $\mathcal{X}$-complete;
\item[$\mathrm{(c3)}$] $\p$ is left $\mathcal{X}$-complete.
  \end{itemize}
\end{itemize}
\end{prop}
\begin{proof} (a) It follows from Proposition \ref{prop: equiv a t3} and Lemma \ref{lem: props C2 y T2} (a).
\

(b) By Proposition \ref{prop: equiv a t3} and Lemma \ref{lem: props C2 y T2}(b),  $\mathcal{T}\cap\mathcal{X}={}{}^{\bot}\left(\mathcal{T}^{\bot}\cap\mathcal{X}\right)\cap\mathcal{T}^{\bot}\cap\mathcal{X}={}{}^{\bot}\left(\mathcal{B}\cap\mathcal{X}\right)\cap\mathcal{B}\cap\mathcal{X}\mbox{.}$
Thus,  by (a),  we conclude that $\mathcal{A}\cap\mathcal{B}\cap\mathcal{X}$ $\subseteq\mathcal{T}\cap\mathcal{X}.$
\

(c) We only prove $(c1)\Rightarrow(c2).$ Indeed,   by the item (b),  Proposition \ref{prop: equiv a t3} and Lemma \ref{lem: props C2 y T2}, 
$\mathcal{T}\cap\mathcal{X}=\mathcal{A\cap\mathcal{B}\cap\mathcal{X}}$
is a relative generator in $\mathcal{T}^{\bot}\cap\mathcal{X}=\mathcal{B}\cap\mathcal{X}$.
Thus,  by \cite[Thm. 4.4 (a)]{parte1}, 
$\forall X\in(\mathcal{B}\cap\mathcal{X})_{\mathcal{X}}^{\vee}$ there
are short exact sequences 
$\suc[X][M_{X}][C_{X}][g_{X}]$,  with $\, M_{X}\in\mathcal{B}\cap\mathcal{X}$,  $C_{X}\in\left(\mathcal{A\cap\mathcal{B}\cap\mathcal{X}}\right)^{\vee}\cap\mathcal{X}$,  and 
$\suc[K_{X}][B_{X}][X][][f_{X}]$,  with $B_{X}\in\left(\mathcal{A\cap\mathcal{B}\cap\mathcal{X}}\right)^{\vee}$,  $K_{X}\in\mathcal{B}\cap\mathcal{X}.$
Furthermore,  since $\mathcal{A\cap\mathcal{B}\cap\mathcal{X}}\subseteq\mathcal{A}\subseteq{}{}^{\bot}\left(\mathcal{B}\cap\mathcal{X}\right)$, 
\cite[Thm. 4.4 (c)]{parte1} implies
$(\mathcal{A}\cap\mathcal{B}\cap\mathcal{X})^{\vee}\subseteq{}^{\bot}(\mathcal{B}\cap\mathcal{X})\mbox{.}$ 
Also,  by \cite[Prop. 4.5 (a)]{parte1}, 
$\mathcal{X}\subseteq(T^{\bot}\cap\mathcal{X})_{\mathcal{X}}^{\vee}=(\mathcal{B}\cap\mathcal{X})_{\mathcal{X}}^{\vee}$.  Lastly,  by \cite[Lem. 4.3]{parte1}, 
$\pdr[\mathcal{B}][\left(\mathcal{A}\cap\mathcal{B}\cap\mathcal{X}\right)^{\vee}]=\pdr[\mathcal{B}][\mathcal{A}\cap\mathcal{B}\cap\mathcal{X}]=0\mbox{.}$
Therefore 
$\left(\mathcal{A}\cap\mathcal{B}\cap\mathcal{X}\right)^{\vee}\cap\X\subseteq{}{}^{\bot}\mathcal{B}\cap\X\subseteq{}{}^{\bot_{1}}\mathcal{B}\cap\X\subseteq\mathcal{A}$ and thus $(\A, \B)$ is $\X$-complete.
\end{proof}

\subsection{\label{sub: -cerrada-por} Alternative conditions for the axiom (T3)}

\begin{defn}\label{def: condiciones T3} For  $\mathcal{T}, \mathcal{X}\subseteq\mathcal{C}, $ we consider the following
conditions.
\begin{description}
\item [{(T3')}] There exists $\omega\subseteq\T^\vee_\X$ which is an $\mathcal{X}$-projective relative generator 
in $\mathcal{X}.$ 
\item [{(T3'')}] There exists $\sigma\subseteq\mathcal{T}_{\mathcal{X}}^{\vee}$
such that $\Addx[\sigma]$ is an $\mathcal{X}$-projective relative
generator in $\mathcal{X}$.
\item [{(t3'')}] There exists $\sigma\subseteq\mathcal{T}_{\mathcal{X}}^{\vee}$
such that $\addx[\sigma]$ is an $\mathcal{X}$-projective relative
generator in $\mathcal{X}$.
\end{description}
\end{defn}

The following lemma is a generalization of \cite[Lem. 2.3]{Tiltinginfinitamentegenerado}.

\begin{lem}\label{lem: lema previo a existencia de preenvolvente} Let   $\mathcal{X}\subseteq\mathcal{C}$ be closed
under extensions and $\mathcal{T}\subseteq\mathcal{C}$ be a class such that $\T\cap\X\subseteq \T^\perp$ and $\sigma\subseteq\mathcal{X}\cap\mathcal{T}_{\mathcal{X}}^{\vee}.$
Then,  for any $W\in\sigma$ and any finite $(\mathcal{T}\cap\mathcal{X})_\X$-coresolution 
$\suc[W][M_{0}\rightarrow...][M_{n}][f_{0}][f_{n}], $
we have that $f_{0}$ is a special $\mathcal{T}^{\bot}\cap\mathcal{X}$-preenvelope, 
a special $\mathcal{T}\cap\mathcal{X}$-preenvelope and a special $\mathcal{T}^{\bot }$-preenvelope.
\end{lem}
\begin{proof}
Let $W\in\sigma$ and $\suc[W][M_{0}\rightarrow...][M_{n}][f_{0}][f_{n}]$
be a finite $(\mathcal{T}\cap\mathcal{X})_\X$-coresolution. By Lemma \ref{lem: chico} (a), 
$\mathcal{T}\cap\mathcal{X}\subseteq\mathcal{T}^{\bot}\cap{}{}^{\bot}\left(\mathcal{T}^{\bot}\right)$.
Hence $M_{j}\in{}{}^{\bot}\left(\mathcal{T}^{\bot}\right)$ $\forall j\in[0, n]$.
Moreover $K_{j}: =\Kerx[f_{j}]\in{}{}^{\bot}\left(\mathcal{T}^{\bot}\right)$
$\forall j\in[1, n]$ since $^{\bot}\left(\mathcal{T}^{\bot}\right)$ is closed
under epi-kernels. In particular $\Extx[1][][K_{2}][X]=0$ $\forall X\in\mathcal{T}^{\bot}$ and thus 
 $f_{0}: W\rightarrow M_{0}$ is a special $\mathcal{T}^{\bot }$-preenvelope, 
which is a $\mathcal{T}^{\bot}\cap\mathcal{X}$-preenvelope and a
$\mathcal{T}\cap\mathcal{X}$-preenvelope since $M_{0}\in\mathcal{T}\cap\mathcal{X}\subseteq\mathcal{T}^{\bot}.$
\end{proof}

\begin{lem}\label{lem: exitencia de la preenvolvetnte}\label{lem: exitencia de la preenvolvetnte-1}
Let $\mathcal{C}$ be an AB4 (abelian) category,  $\mathcal{X}=\Addx[\mathcal{X}]\subseteq\mathcal{C}$ 
($\mathcal{X}=\addx[\mathcal{X}]\subseteq \C$)
be closed under extensions,  $\mathcal{T}\subseteq\mathcal{C}$
be a class satisfying $\mathrm{(T2)}, $ $\sigma\subseteq\mathcal{X}\cap\mathcal{T}_{\mathcal{X}}^{\vee}$
and $\omega: =\Addx[\sigma]$ ($\omega: =\addx[\sigma]$). Then,  the following
statements hold true. 
\begin{itemize}
\item[$\mathrm{(a)}$] $\omega\subseteq{}{}^{\bot}\left(\mathcal{T}^{\bot}\right)\cap\mathcal{X}$.
\item[$\mathrm{(b)}$] If $\mathcal{T}=\mathcal{T}^{\oplus}$ ($\mathcal{T}=\mathcal{T}^{\oplus_{<\infty}}$), 
then every $W\in\omega$ admits an exact sequence $\suc[W][M_{W}][C_{W}][f]\mbox{, }$
where $M_{W}\in\mathcal{T}\cap\mathcal{X}$,  $C_{W}\in{}^{\bot}\left(\mathcal{T}^{\bot}\right)\cap\mathcal{X}$
and $f$ is a special $\mathcal{T}^{\bot}\cap\mathcal{X}$-preenvelope, 
a special $\mathcal{T}\cap\mathcal{X}$-preenvelope and a special
$\mathcal{T}^{\bot}$-preenvelope.
\end{itemize}
\end{lem}
\begin{proof}
Let us prove the lemma by assuming that $\C$ is an AB4 category. The case when  $\mathcal{C}$ is just abelian can be done by applying similar arguments. 
\

(a) Let us show that $\sigma \subseteq {}^{\bot}\left(\mathcal{T}^{\bot}\right)\cap\mathcal{X}$. Indeed,  by Lemma \ref{lem: lema previo a existencia de preenvolvente},  every  $S\in\sigma$ admits an exact sequence 
$\suc[S][M_{S}][C_{S}]\mbox{, }$
with $M_{S}\in\mathcal{T}\cap\mathcal{X}\subseteq{}{}^{\bot}\left(\mathcal{T}^{\bot}\right)\cap\mathcal{X}\cap\mathcal{T}^{\bot}$,  
 $C_{S}\in\mathcal{X}\cap{}{}^{\bot}\left(\mathcal{T}^{\bot}\right)$,  and thus,  $S\in^{\bot} (\T ^{\bot})\cap\X$ since   $^{\bot} (\T ^{\bot})$ is closed under epi-kernels. Finally,  $\Addx[\sigma]\subseteq{}{}^{\bot}\left(\mathcal{T}^{\bot}\right)\cap\mathcal{X}$ since $\mathcal{C}$ is AB4 and $\sigma \subseteq {}^{\bot}\left(\mathcal{T}^{\bot}\right)\cap\mathcal{X}.$
\

(b)  Let $\mathcal{T}=\mathcal{T}^{\oplus}$ and $W\in\omega: =\Addx[\sigma]$. Then
there is $W'\in\omega$ and a set $\{S_{i}\}_{i\in I}\subseteq\sigma\subseteq\mathcal{X}$
such that $W\oplus W'\cong\bigoplus_{i\in I}S_{i}^{(\alpha_{i})}$.
By Lemma \ref{lem: lema previo a existencia de preenvolvente},  for every
$i\in I$,  there is an exact sequence 
$\suc[S_{i}][M_{S_{i}}][C_{S_{i}}]$
with $M_{S_{i}}\in\mathcal{T}\cap\mathcal{X}$ and $C_{S_{i}}\in{}{}^{\bot}\left(\mathcal{T}^{\bot}\right)\cap\mathcal{X}$.
Let $S: =\bigoplus_{i\in I}S_{i}^{(\alpha_{i})}$,  $M': =\bigoplus_{i\in I}M_{S_{i}}^{(\alpha_{i})}$
and $C: =\bigoplus_{i\in I}C_{S_{i}}^{(\alpha_{i})}$.\\%
\fbox{\begin{minipage}[t]{0.2\columnwidth}%
\[
\begin{tikzpicture}[-, >=to, shorten >=1pt, auto, node distance=1cm, main node/.style=, x=.5cm, y=.5cm]

 \node[main node] (1) at (0, 0){$W'$};
 \node[main node] (2) at (-2, 0){$S$};
 \node[main node] (3) at (-4, 0){$W$};
 \node[main node] (4) at (-4, -2){$W$};
 \node[main node] (5) at (-2, -2){$M'$};
 \node[main node] (6) at (0, -2){$Z$};
 \node[main node] (7) at (-2, -4){$C$};
 \node[main node] (8) at (0, -4){$C$};

\draw[right hook->,  thin]   (3)  to  node  {$$}    (2);
\draw[->>,  thin]   (2)  to  node  {$$}    (1);
\draw[right hook->,  thin]   (4)  to  node  {$$}    (5);
\draw[->>,  thin]   (5)  to  node  {$$}    (6);
\draw[right hook->,  thin]   (2)  to  node  {$$}    (5);
\draw[->>,  thin]   (5)  to  node  {$$}    (7);
\draw[right hook->,  thin]   (1)  to  node  {$$}    (6);
\draw[->>,  thin]   (6)  to  node  {$$}    (8);
\draw[-,  double]   (3)  to  node  {$$}    (4);
\draw[-,  double]   (7)  to  node  {$$}    (8);

\end{tikzpicture}
\]%
\end{minipage}}\hfill{}%
\begin{minipage}[t]{0.75\columnwidth}%
 Since $\mathcal{C}$
is AB4 and $\mathcal{T}=\mathcal{T}^{\oplus}$,  we have the short exact
sequence $\suc[S][M'][C][f][g]\mbox{, }$
where $S\in\omega$,  $M'\in\mathcal{T}\cap\mathcal{X}$ and $C\in{}{}^{\bot}\left(\mathcal{T}^{\bot}\right)\cap\mathcal{X}$;
and the splitting exact sequence 
$\suc[W][S][W'][][h]\mbox{.}$
Considering the push-out of $h$ with $f$,  we get a short exact sequence
$\eta: \;\suc[W][M'][Z][\alpha]\mbox{.}$
Finally,  observe from the exact sequence 
$\suc[W'][Z][C]$ 
that $Z\in{}^{\bot}\left(\mathcal{T}^{\bot}\right)\cap\mathcal{X}$.
Therefore,  using $\eta, $ we can conclude the desired result. 
\end{minipage}\\
\end{proof}

\begin{lem}\label{lem: props T2 con T3' y C2 con C3'-2}\label{lem: props T2 con T3' y C2 con C3'-2-1}
Let $\mathcal{C}$ be an AB4 (abelian) category,  $\mathcal{X}\subseteq\mathcal{C}$
be closed under extensions and such that $\mathcal{X}=\Addx[\mathcal{X}]$
($\mathcal{X}=\addx[\mathcal{X}]$). If $\mathcal{T}\subseteq\mathcal{C}$
satisfies $\mathrm{(T2),  (T3'')\,  ((t3'')), }$ and $\mathcal{T}=\mathcal{T}^{\oplus}$
($\mathcal{T}=\mathcal{T}^{\oplus_{<\infty}}$),  then $\mathcal{T}^{\bot }\cap\mathcal{X}\subseteq\Gennr[\mathcal{T}][1]$.\end{lem}
\begin{proof}
It can be proved in a similar way as Lemma \ref{lem: props T2 con T3' y C2 con C3'}.
\end{proof}

We close this section with a generalization of \cite[Cor. 3.6]{positselskicorrespondence}.

\begin{prop}\label{prop: en caso de tener un generador proyectivo}\label{cor: caract}\label{rem: T3 simple}\label{prop: en caso de tener un generador proyectivo-1}\label{cor: caract-1}\label{rem: T3 simple-1}
Let   $\mathcal{X}=\smdx[\mathcal{X}]\subseteq\mathcal{C}$
be closed under extensions and admitting an
$\mathcal{X}$-projective relative generator in $\mathcal{X}$,  and let $\mathcal{T}\subseteq\mathcal{C}$ be
satisfying $\mathrm{(T0),  (T1),  (T2),}$ $\mathrm{  (T4), }$ and $\mathrm{(T5)}.$ Then, 
the following conditions are equivalent: 
\begin{description}
\item [{(T3)}] There exists $\omega\subseteq\T_{\mathcal{X}}^{\vee}$ which is a relative generator in $\mathcal{X}.$
\item [{(T3')}] There exists $\omega\subseteq\T_{\mathcal{X}}^{\vee}$ which is an $\mathcal{X}$-projective relative generator
in $\mathcal{X}.$ 
\end{description}

Furthermore,  if $\mathcal{C}$ is AB4 (abelian),  $\mathcal{X}=\Addx[\mathcal{X}]$
($\mathcal{X}=\addx[\mathcal{X}]$) and $\mathcal{T}=\mathcal{T}^{\oplus}$
($\mathcal{T}=\mathcal{T}^{\oplus_{<\infty}}$),  then $\mathrm{(T3)}$ and $\mathrm{(T3')}$
are equivalent to the following one: 
\begin{description}
\item [{(T3'')}] there exists $\sigma\subseteq\mathcal{T}_{\mathcal{X}}^{\vee}$
such that $\Addx[\sigma]$ is an $\mathcal{X}$-projective relative
generator in $\mathcal{X}$
\end{description}

($\mathrm{\mathbf{(t3''): }}$ $\, $ there exists $\sigma\subseteq\mathcal{T}_{\mathcal{X}}^{\vee}$
such that $\addx[\sigma]$ is an $\mathcal{X}$-projective relative
generator in $\mathcal{X}$).

\end{prop}
\begin{proof}  The implication (T3) $\Rightarrow$ (T3') follows from Theorem \ref{thm: el par n-X-tilting} (a); and (T3') $\Rightarrow$ (T3)  is trivial. Let $\mathcal{C}$ be AB4, $\mathcal{X}=\Addx[\mathcal{X}]$ and
$\mathcal{T}=\mathcal{T}^{\oplus}$ (the case where $\mathcal{C}$ is
abelian, $\mathcal{X}=\addx[\mathcal{X}]$ and $\mathcal{T=\mathcal{T}^{\oplus_{<\infty}}}$
can be done  by similar arguments).

(T3') $\Rightarrow$ (T3''):  Let $\omega$ be the relative generator
in $\mathcal{X}$ satisfying (T3'). Since $\omega\subseteq\mathcal{X}=\Addx[\mathcal{X}]$, we can take
$\sigma:=\omega.$ 

(T3'') $\Rightarrow$ (T3): It follows from Proposition \ref{prop: equiv a t3} and
Lemma \ref{lem: props T2 con T3' y C2 con C3'-2}. 
\end{proof}

\subsection{\label{sub: Big small}Tilting for classes of compact-like objects}

In this section we will consider a class $\mathcal{X}$ consisting
of compact-like objects in an abelian category $\C.$ We shall see that,  in this case,  a class $\mathcal{T}$
is big $n$-$\mathcal{X}$-tilting if and only if it is small $n$-$\mathcal{X}$-tilting.
Let us begin by defining what kind of compact-like objects we will be
considering. 
\

Let $\C$ be an additive category, 
$\mathcal{T}\subseteq\mathcal{C}$ and $M\in\mathcal{C}.$ We recall that $M$ is \textbf{finitely $\mathcal{T}$-generated} if,  for every  family $\left\{ U_{i}\right\} _{i\in I} \subseteq \mathcal{T}$
such that $\bigoplus_{i\in I}U_{i}$ exists in $\C,$ 
every epimorphism $\varphi: \bigoplus_{i\in I}U_{i}\rightarrow M$ in
$\mathcal{C}$  admits a finite set $F\subseteq I$ such that
the composition 
$\bigoplus_{i\in F}U_{i}\xrightarrow{i_{F, I}}\bigoplus_{i\in I}U_{i}\xrightarrow{\varphi}M$
is an epimorphism,  where $i_{F, I}$ is the natural inclusion. We
 denote by $\operatorname{f.g.}(\mathcal{T})$
the class of all the finitely $\mathcal{T}$-generated objects in $\C.$ It is said that $M$ is \textbf{$\mathcal{T}$-compact} (\textbf{$\mathcal{T}$-compact for monomorphisms})
if,  for every family  $\left\{ U_{i}\right\} _{i\in I}\subseteq\mathcal{T}$
such that $\bigoplus_{i\in I}U_{i}$ exists in $\C,$  every morphism (monomorphism)
$\psi: M\rightarrow\bigoplus_{i\in I}U_{i}$ in $\mathcal{C}$ admits
a finite set $F\subseteq I$ such that $\psi$ factors through the
 inclusion $i_{F, I}: \bigoplus_{i\in F}U_{i}\rightarrow\bigoplus_{i\in I}U_{i}\mbox{.}$
We denote by $\mathcal{K}_{\mathcal{T}}$ ($\mathcal{K}_{\mathcal{T},\mathcal{M}})$
the class of all the $\mathcal{T}$-compact ($\mathcal{T}$-compact for monomorphism) objects in $\mathcal{C}.$ Notice that $\mathcal{K}_{\mathcal{T}}\subseteq\mathcal{K}_{\mathcal{T},\mathcal{M}}.$

\begin{lem}\cite[Chap. II. Lem. 16.1]{mitchell}\label{lem: para compactos}
Let $\mathcal{C}$ be an additive category and $\{A_i\}_{i\in I}\subseteq\C$ be a family of objects such  that  $\bigoplus_{i\in I}A_{i}$ exists in $\C.$  
Then, for a finite subset $F\subseteq I,$  a morphism $\alpha: A\rightarrow\bigoplus_{i\in I}A_{i}$ in $\C$ factors
through $i_{F, I}: \bigoplus_{i\in F}A_{i}\rightarrow\bigoplus_{i\in I}A_{i}$ if, and only if, $\alpha=\sum_{i\in F}u_{i}p_{i}\alpha$, 
where $u_{i}$ and $p_{i}$ are,  respectively,  the i-th injection and the the i-th
projection for the coproduct $\bigoplus_{i\in I}A_{i}$.
\end{lem}

As an easy consequence of Lemma \ref{lem: para compactos}, we get the following corollary.

\begin{cor}\label{lem: compactos son cerrados por cocientes}  Let $\mathcal{C}$
be an additive category and $\mathcal{T}\subseteq\mathcal{C}$. If  
 $\pi: M\rightarrow N$ is an epimorphism in $\mathcal{C}$ with $M\in\mathcal{K}_{\mathcal{T}}$, then $N\in\mathcal{K}_{\mathcal{T}}$.\end{cor}
 
The $\T$-compact objects can be characterized as follows.

\begin{lem}\label{lem: caracterizacion T-compactos}\label{lem: caracterizacion compactos}
For an additive category $\mathcal{C}, $ $\mathcal{T}\subseteq\mathcal{C}$ 
and $M\in\mathcal{C}, $ the following statements are equivalent.
\begin{itemize}
\item[$\mathrm{(a)}$] $M$ is $\mathcal{T}$-compact.
\item[$\mathrm{(b)}$] For every family  $\left\{ U_{i}\right\} _{i\in X} \subseteq \mathcal{T}$
such $\bigoplus_{i\in X}U_{i}$ exists in $\C,$   the map 
$$\upsilon: \bigoplus_{i\in X}\Homx[][M][U_{i}]  \rightarrow\Homx[][M][\bigoplus_{i\in X}U_{i}], \, (\alpha_{i})_{i\in X}\mapsto\sum_{i\in X}u_{i}\alpha_{i}\mbox{, }$$ 
 is an isomorphism,  where $u_{i}: U_{i}\rightarrow\bigoplus_{i\in X}U_{i}$
is the natural inclusion in the coproduct.
\end{itemize}
\end{lem}

\begin{proof} Let $\left\{ U_{i}\right\} _{i\in X}\subseteq\mathcal{T}$
such $\bigoplus_{i\in X}U_{i}$ exists in $\C.$  For every 
$i\in X$, consider the natural projection $p_k:\bigoplus_{i\in X}U_{i}\rightarrow U_k$.
Notice that $\upsilon$ is always a monomorphism.
\

(a) $\Rightarrow$ (b)  Let $\alpha:M\rightarrow\bigoplus_{i\in X}U_{i}$
in $\mathcal{C}$. Since $M\in\mathcal{K}_{\mathcal{T}}$, there is
a finite set $J\subseteq X$ such that $\alpha=\sum_{i\in J}u_{i}p_{i}\alpha,$ see Lemma \ref{lem: para compactos}.
Therefore $\alpha=\upsilon(p_{i}\alpha)_{i\in X}$ and thus $\upsilon$
is surjective.
\

(b) $\Rightarrow$ (a) From (b), we have that 
 every $\alpha\in\Homx[][M][\bigoplus_{i\in X}U_{i}]$ admits an element 
$(\alpha_{i})_{i\in X}\in\bigoplus_{i\in X}\Homx[][M][U_{i}]$ such
that $\alpha=\sum_{i\in X}u_{i}\alpha_{i}\mbox{.}$
Now, since $p_{k}\upsilon(\alpha_{i})_{i\in X}=\alpha_{k}$ $\forall k\in X$,
we get 
$\alpha=\sum_{j\in X}u_{j}\alpha_{j}=\sum_{j\in X}u_{j}(p_{j}\upsilon(\alpha_{i})_{i\in X})=\sum_{j\in X}u_{j}p_{j}\alpha\mbox{.}$
Hence, $M\in\mathcal{K}_{\mathcal{T}}$ by Lemma \ref{lem: para compactos}.
\end{proof}

As a consequence of Lemma \ref{lem: caracterizacion T-compactos}, we get the following result.

\begin{cor}\label{cor: coproducto de3 compactos} Let $\mathcal{C}$ be an additive
category,   $\mathcal{T}\subseteq\mathcal{C}$ and $A=\oplus_{i=1}^nA_i$ in $\C.$ Then $A$ 
is $\mathcal{T}$-compact if,  and only if,  each $A_i$ is 
$\mathcal{T}$-compact.
\end{cor}

\begin{cor}\label{cor: finitamente generado es compacto}  For an additive category $\mathcal{C}, $ 
 $\mathcal{T}\subseteq\mathcal{C}$ and a relative generator $\omega^{\oplus}$ 
 in $\mathcal{C}, $ with $\omega\subseteq\mathcal{K}_{\mathcal{T}}, $ 
 the following statements hold true. 
\begin{itemize}
\item[$\mathrm{(a)}$] $\operatorname{f.g.}(\omega)\subseteq\operatorname{Fac}_{1}(\omega^{\oplus_{<\infty}})\subseteq\mathcal{K}_{\mathcal{T}}$.

\item[$\mathrm{(b)}$] If $\mathcal{C}$ is abelian and $\Extx[1][][\omega][\operatorname{Fac}_{1}(\omega^{\oplus_{<\infty}})]=0$, 
then $\operatorname{Fac}_{1}(\omega^{\oplus_{<\infty}})$ is closed
under extensions in $\mathcal{C}$. 
\end{itemize}
\end{cor}
\begin{proof} The item (a) follows from Corollaries \ref{lem: compactos son cerrados por cocientes}
and \ref{cor: coproducto de3 compactos}. Finally,  the proof of (b) can be done in a similar way as the proof of the  Horseshoe's Lemma. 
\end{proof}

\begin{lem}\label{lem: finitamente generados en Ab5}\cite[Chap. V. Lem. 3.1]{ringsofQuotients}
Let $\mathcal{C}$ be an AB5 category. Then $\operatorname{f.g.}(\mathcal{C})$
is closed under quotients and extensions. 
\end{lem}

We have the following well-known facts. 

\begin{cor}\label{cor: finitamente generado es compacto-1} For  a ring $R$
and $\modd[R]: =\operatorname{f.g.}(\Modx[R]),$  the following
statements hold true. 
\begin{itemize}
\item[$\mathrm{(a)}$] $\modd\subseteq\mathcal{K}_{\Modx[R]}$ and $\modd$ is closed under extensions
and quotients in $\Modx$. In particular,  $\modd$ is right thick in
$\Modx$. 
\item[$\mathrm{(b)}$] $R$ is left noetherian if,  and only if,  $\modd$ is a thick abelian
subcategory of $\Modx$. 
\end{itemize}
\end{cor}

\begin{prop}\label{prop: add Add vs compactos}   Let $\mathcal{C}$ be an abelian
category and  $\mathcal{T}\subseteq\mathcal{C}$.
Then $\Addx[\mathcal{T}]\cap\mathcal{Z}=\addx[\mathcal{T}]\cap\mathcal{Z}$
for every $\mathcal{Z}\subseteq\mathcal{K}_{\mathcal{T}, \mathcal{M}}$.
\end{prop}
\begin{proof}
Let $\mathcal{Z}\subseteq\mathcal{K}_{\mathcal{T}, \mathcal{M}}$. 
Consider $X\in\Addx[\mathcal{T}]\cap\mathcal{Z}$. Then,  there is a
splitting exact sequence 
$\suc[X'][\bigoplus_{i\in I}U_{i}][X][][f]\mbox{ with }\{U_{i}\}_{i\in I}\subseteq\mathcal{T}\mbox{.}$
Let $\mu: X\rightarrow\bigoplus_{i\in I}U_{i}$ be a monomorphism such
that $f\mu=1_{X}$. It follows that there is a finite set $J\subseteq I$
and a morphism $\mu': X\rightarrow\bigoplus_{j\in J}U_{j}$ such that
$\mu=i_{J, I}\circ\mu'$,  where $i_{J, I}: \bigoplus_{j\in J}U_{j}\rightarrow\bigoplus_{i\in I}U_{i}$
is the natural inclusion. Consider the morphism $g: =f\circ i_{J, I}: \bigoplus_{j\in J}U_{j}\rightarrow X$.
Since $g\mu'=f\circ i_{J, I}\circ\mu'=f\mu=1_{X}$,  $g$ is a splitting
epimorphism and thus  $X\in\addx[\mathcal{T}]\cap\mathcal{Z}.$ 
\end{proof}

\begin{lem}\label{lem:  addT precub es AddT precub}   Let $\mathcal{C}$ be an
AB3 category,   $\mathcal{M}\subseteq\mathcal{C}$ and $\alpha: M\rightarrow X$ in $\C.$
If $\alpha$ is an $\addx[\mathcal{M}]$-precover of $X$,  then $\alpha$
is an $\Addx[\mathcal{M}]$-precover of $X$.\end{lem}
\begin{proof}
Let $\alpha: M\rightarrow X$ be an $\addx[\mathcal{M}]$-precover. Since
$\Addx[\mathcal{M}]=\smdx[\mathcal{M}^{\oplus}]$,  it is
easy to see that every morphism $M'\rightarrow X$,  with $M'\in\Addx[\mathcal{M}]$, 
factors through $\mathcal{M}^{\oplus}$. Furthermore,   every morphism $M''\rightarrow X$ with $M''\in\mathcal{M}^{\oplus}$
factors through $M^{\oplus}$. Indeed,  consider a morphism $\alpha': \bigoplus_{i\in I}M_{i}\rightarrow X$, 
with $M_{i}\in\mathcal{M}\: \forall i\in I$,  and the canonical inclusions
$\left\{ v_{i}: M_{i}\rightarrow\bigoplus_{i\in I}M_{i}\right\} _{i\in I}$.
Since $\alpha$ is an $\addx[\mathcal{M}]$-precover,  $\forall i\in I$
there is  $\lambda_{i}: M_{i}\rightarrow M$ such that $\alpha'v_{i}=\alpha\lambda_{i}.$ 
Therefore,  there is $\lambda: \bigoplus_{i\in I}M_{i}\rightarrow M^{(I)}$
such that $\lambda v_{i}=v'_{i}\lambda_{i}$ $\forall\, i\in I, $ where $v'_{i}: M\rightarrow M^{(I)}$
is the natural inclusion in the coproduct. Observe that $\alpha'$ factors through
$\lambda$ and $\alpha'': M^{(I)}\rightarrow X$,  where $\alpha''$
is induced by the coproduct universal property and moreover 
$\alpha''v'_{i}=\alpha, $ for all $i\in I.$ Now,  for each $i\in I, $ we have 
 $\alpha''\lambda v_{i}=\alpha''v'_{i}\lambda_{i}=\alpha\lambda_{i}=\alpha'v_{i}$ and thus 
 $\alpha''\lambda=\alpha'$.

Finally,   we assert that $\alpha''$ factors through $\alpha$.
To show it,  consider the morphism $\sigma: M^{(I)}\rightarrow M$  such that $\sigma v'_{j}=1_{M}, $ for all $j\in I.$
Then,  for each $j\in I, $ we get  $\alpha\sigma v'_{j}=\alpha1_{M}=\alpha=\alpha''v'_{j}$ and so
 $\alpha\sigma=\alpha''$.
 \end{proof}

\begin{thm}\label{thm: n-X-tilting sii n-X-tilting peque=0000F1o} Let $\mathcal{C}$
be an AB4 category and  $\mathcal{T}\subseteq\mathcal{C}$.
Then,  for every $\mathcal{Z}\subseteq\mathcal{K}_{\mathcal{T}, \mathcal{M}}$
 closed under extensions, 
$\Addx[\mathcal{T}]$ is $n$-$\mathcal{Z}$-tilting if and only if
$\addx[\mathcal{T}]$ is $n$-$\mathcal{Z}$-tilting.
\end{thm}
\begin{proof}
Let $\mathcal{Z}\subseteq\mathcal{K}_{\mathcal{T}, \mathcal{M}}$
 be closed under extensions. Since
$\mathcal{C}$ is AB4,  $\pdr[\mathcal{X}][\operatorname{Add}(\mathcal{T})]=\pdr[\mathcal{X}][\mathcal{T}]=\pdr[\mathcal{X}][\operatorname{add}(\mathcal{T})]$.
Now,  by Proposition \ref{prop: add Add vs compactos},   
$\Addx[\mathcal{T}]\cap\mathcal{Z}=\addx[\mathcal{T}]\cap\mathcal{Z}\mbox{.}$
 Hence,  using that $\mathcal{Z}$ is closed under extensions,  
we have $\omega\subseteq(\Addx[\mathcal{T}])_{\mathcal{Z}}^{\vee}$
if and only if $\omega\subseteq(\addx[\mathcal{T}])_{\mathcal{Z}}^{\vee}, $ for any $\omega\subseteq\mathcal{Z}$.
Finally,   (T5) follows from Lemma \ref{lem:  addT precub es AddT precub}
and  Proposition \ref{prop: add Add vs compactos}.
\end{proof}

\begin{cor}\label{cor: coro1 p80}   Let $\mathcal{C}$ be an AB4 category,  $\mathcal{T}\subseteq\mathcal{W}\subseteq\mathcal{C}$, 
and let $\omega^{\oplus}$ be a relative generator in $\mathcal{C}$ such
that $\omega\subseteq\mathcal{K}_{\mathcal{W}}$. Then,  for every
$\mathcal{Z}\subseteq\operatorname{f.g.}(\omega)$ closed under extensions, 
we have that $\Addx[\mathcal{T}]$ is $n$-$\mathcal{Z}$-tilting if
and only if $\addx[\mathcal{T}]$ is $n$-$\mathcal{Z}$-tilting. \end{cor}
\begin{proof}
It follows by Theorem \ref{thm: n-X-tilting sii n-X-tilting peque=0000F1o}
and Corollary \ref{cor: finitamente generado es compacto}.
\end{proof}

\begin{cor}\label{cor: coro2 p80}Let $R$ be a ring and $\mathcal{T}\subseteq\Modx$.
Then,  for every $\mathcal{Z}\subseteq\modd$,  closed under extensions, 
we have that $\Addx[\mathcal{T}]$ is $n$-$\mathcal{Z}$-tilting if
and only if $\addx[\mathcal{T}]$ is $n$-$\mathcal{Z}$-tilting. 
\end{cor}
\begin{proof}
It follows by Corollary \ref{cor: finitamente generado es compacto-1} (a) and
Corollary \ref{cor: coro1 p80}. 
\end{proof}

\subsection{$n$-$\mathcal{X}$-tilting triples in abelian categories}

\begin{defn} We say that $\left(\p;\mathcal{T}\right)$
is a\textbf{ big (small) $n$-$\mathcal{X}$-tilting triple} in an abelian category $\C$ provided  the following statements hold true: 
\begin{description}
\item [(TT1)]  $\p$ is a left cotorsion pair in $\mathcal{X}$ with $\idr[\mathcal{A}][\mathcal{B}\cap\mathcal{X}]=0.$
\item [(TT2)] $\mathcal{B}$ is closed under extensions and direct summands. 
\item [(TT3)] There is a big (small) $n$-$\mathcal{X}$-tilting class $\mathcal{T}$
such that $\mathcal{B}\cap\mathcal{X}=\mathcal{T}^{\bot}\cap\mathcal{X}$
and $\mathcal{T}\cap\mathcal{X}\subseteq\mathcal{A}\cap\mathcal{B}\cap\mathcal{X}$. 
\end{description}
\end{defn}

\begin{lem}\label{lem: lemita inyectivos vs par X-completo a izq} Let $\p$ be
a right $\mathcal{X}$-complete pair in an abelian category $\mathcal{C}$
such that $\mathcal{B}\cap\mathcal{X}=\smdx[\mathcal{B}\cap\mathcal{X}].$ 
If $\alpha\subseteq\mathcal{X}\subseteq{}^{\bot_{1}}\alpha$,  then
$\alpha\subseteq\mathcal{B}\cap\mathcal{X}$.\end{lem}
\begin{proof} It is straightforward.
\end{proof}

\begin{lem} \label{lem: lema previo a 2.80} For   a right $\mathcal{X}$-complete
and $\mathcal{X}$-hereditary pair $\p$ in $C$ such that $\mathcal{B}\cap\mathcal{X}=\smdx[\mathcal{B}\cap\mathcal{X}]$
and $n: =\max\left\{ 1, \pdr[\mathcal{X}][\mathcal{A}]\right\} <\infty, $
 the following statements hold true. 
\begin{itemize}
\item[$\mathrm{(a)}$] Every $X\in\mathcal{X}$ admits an exact sequence 
\[
0\rightarrow X\overset{f_{0}}{\rightarrow}B_{X, 0}\overset{f_{1}}{\rightarrow}B_{X, 1}\rightarrow...\overset{f_{n}}{\rightarrow}B_{X, n}\rightarrow0\mbox{, }
\]
with $B_{X, n}\in\mathcal{A}\cap\mathcal{B}\cap\mathcal{X}$,  $B_{X, i}\in\mathcal{B}\cap\mathcal{X}$ 
and $\Cok[f_{i}]\in\mathcal{A}\cap\mathcal{X}$ $\forall i\in[0, n-1]$.
In particular,  $\coresdimr{\mathcal{B}\cap\mathcal{X}}{\mathcal{X}}{\mathcal{A}\cap\mathcal{X}}\leq n$. 

\item[$\mathrm{(b)}$] Let $\mathcal{A}\cap\mathcal{X}$ and $\mathcal{X}$
be closed under extensions and $^{\bot}(\B \cap \X)\cap \B \cap \X \subseteq \A \cap \B \cap \X$. Then,  every $W\in\mathcal{X}\cap{}^{\bot}\mathcal{X}$
admits an exact sequence 
\[
0\rightarrow W\stackrel{f_{0}}{\rightarrow}B_{W, 0}\stackrel{f_{1}}{\rightarrow}B_{W, 1}\rightarrow...\stackrel{f_{n}}{\rightarrow}B_{W, n}\rightarrow0
\]
with $B_{W, i}\in\mathcal{A}\cap\mathcal{B}\cap\mathcal{X}$ $\forall i\in[0, n]$
and $\Cok[f_{j}]\in\mathcal{A}\cap\mathcal{X}$ $\forall j\in[0, n-1]$.
In particular,  $\coresdimr{\mathcal{A}\cap\mathcal{B}\cap\mathcal{X}}{\mathcal{X}\cap{}^{\bot}\mathcal{X}}{\mathcal{A}\cap\mathcal{X}}\leq n$. 
\end{itemize}
\end{lem}
\begin{proof} (a) Let $X\in\mathcal{X}$. Since $\p$ is right $\mathcal{X}$-complete, 
there is an exact sequence 
$\suc[X][B_{X, 0}][C_{1}][\, ][\, ]\mbox{, }$
 with $B_{X, 0}\in\mathcal{B}\cap\mathcal{X}$ and $C_{1}\in\mathcal{A}\cap\mathcal{X}$.
Repeating the same argument recursively,  we can build an exact sequence
\[
0\rightarrow X\rightarrow B_{X, 0}\rightarrow B_{X, 1}\rightarrow...\rightarrow B_{X, n}\rightarrow C_{n+1}\rightarrow0
\]
with $B_{X, i}\in\mathcal{B}\cap\mathcal{X}$ $\forall i\in[0, n]$
and $C_{n+1}\in\mathcal{A}\cap\mathcal{X}$. Moreover $B_{X, i}\in C_{n+1}^{\bot}$ $\forall i\in[0, n]$ since $\p$ is $\mathcal{X}$-hereditary. Hence 
$\Extx[1][][C_{n+1}][C_{n}]\cong\Extx[n+1][][C_{n+1}][X]=0$ 
since $C_{n+1}\in\mathcal{A}$ and $\pdr[\mathcal{X}][\mathcal{A}]\leq n$.
Therefore
$\suc[C_{n}][B_{X, n}][C_{n+1}][\, ][\, ]$
splits and thus 
$0\rightarrow X\rightarrow B_{X, 0}\rightarrow B_{X, 1}\rightarrow...\rightarrow B_{X, n-1}\rightarrow C_{n}\rightarrow0$
is the desired exact sequence.
\

(b)  Let $W\in\X \cap {}^{\bot}\X$. Consider
the exact sequence 
\[
0\rightarrow W\stackrel{f_{0}}{\rightarrow}B_{W, 0}\stackrel{f_{1}}{\rightarrow}B_{W, 1}\rightarrow...\stackrel{f_{n}}{\rightarrow}B_{W, n}\rightarrow0
\]
obtained in (a). Now,  by using that $\mathcal{A}\cap\mathcal{X}$ is
closed under extensions,  we can conclude that  $B_{W, k}\in\mathcal{A}\cap\mathcal{B}\cap\mathcal{X}$
$\forall k\in[1, n]$. Finally,  consider the exact sequence 
$\suc[W][B_{W, 0}][C_{1}]\mbox{.}$
Since $W\in\mathcal{X}\cap{}^{\bot}\mathcal{X}\subseteq\mathcal{X}\cap{}^{\bot}\left(\mathcal{B}\cap\mathcal{X}\right)$
and $C_{1}\in\mathcal{A}\cap\mathcal{X}\subseteq\mathcal{X}\cap{}^{\bot}\left(\mathcal{B}\cap\mathcal{X}\right)$, 
we have  $B_{W, 0}\in\mathcal{X}\cap{}{}^{\bot}\left(\mathcal{B}\cap\mathcal{X}\right)\cap\mathcal{B}\subseteq \mathcal{X}\cap\mathcal{A}\cap\mathcal{B}\mbox{.}$ 
\end{proof}

The following result is a generalization of \cite[Thm. 3.2]{Hopel-Mendoza}. Note that we are writing,  at the same time,  the big and the small versions.

\begin{thm}\label{thm: dimensiones en un par de cotorsion tilting}\label{thm: dimensiones en un par de cotorsion tilting-1}\label{rem: obs 2.80'}\label{rem: obs 2.80''}
Let $\mathcal{C}$ be an AB4 (abelian) category,   $\mathcal{X}\subseteq\mathcal{C}$
be  closed under extensions  such that $\mathcal{X}=\Addx[\mathcal{X}]$
($\mathcal{X}=\addx[\mathcal{X}]$) and admits an $\mathcal{X}$-injective
relative cogenerator in $\mathcal{X}$ and an $\mathcal{X}$-projective
relative generator in $\mathcal{X}$. Consider a left cotorsion pair $\p$
in $\mathcal{X}$  such that $\idr[\mathcal{A}][\mathcal{B}\cap\mathcal{X}]=0$ and
 $\mathcal{B}=\smdx[\mathcal{B}]$ is closed under extensions,  and let $\kappa : = \A \cap \B \cap \X$. Then, 
the following statements are equivalent: 
\begin{itemize}
\item[$\mathrm{(a)}$]  There exists $\mathcal{T}\subseteq\mathcal{C}$ such that $(\p;\mathcal{T})$
is a big (small) $n$-$\mathcal{X}$-tilting triple.

\item[$\mathrm{(b)}$] $\p$ is right $\mathcal{X}$-complete,  $^{\bot}(\B \cap \X)\cap \B \cap \X \subseteq \A \cap \B \cap \X$,  $\pdr[\mathcal{X}][\mathcal{A}]\leq n$, 
$\kappa=\kappa^{\oplus}$ ($\kappa=\kappa^{\oplus_{<\infty}}$) and
$\kappa$ is precovering in $\mathcal{B}\cap\mathcal{X}.$

\item[$\mathrm{(c)}$] $\kappa$ is a big (small) $n$-$\mathcal{X}$-tilting class such
that $\mathcal{B}\cap\mathcal{X}=\kappa^{\bot}\cap\mathcal{X}$.
\end{itemize}

Furthermore,  if any of the above conditions is satisfied,  then $\mathcal{A}\cap\mathcal{X}\subseteq\kappa^{\vee}$,  
$\kappa=(\mathcal{A}\cap\mathcal{X})^{\bot}\cap\mathcal{A}\cap\mathcal{X}={}^{\bot}(\mathcal{B}\cap\mathcal{X})\cap\mathcal{B}\cap\mathcal{X}$,  
$\coresdimr{\mathcal{B}\cap\mathcal{X}}{\mathcal{X}}{\mathcal{A}\cap\mathcal{X}}\leq\max\left\{ 1, n\right\} $, 
and $\idr[\mathcal{A}\cap\mathcal{X}][\mathcal{X}]=\coresdimr{\mathcal{B}}{\mathcal{X}}{\mathcal{X}}\leq n$.
Moreover,  if $\Addx[\sigma]$ ($\addx[\sigma]$) is  $\X$-projective and a relative generator in $\X$ and $\sigma$ is a (finite) set (and $\addx[X]$ is precovering
in $X^{\bot}\cap\mathcal{X}$ $\forall X\in\kappa$),  then we can
dismiss the hypothesis from $\mathrm{(b)}$ which says that $\kappa$ is precovering in $\mathcal{B}\cap\mathcal{X}$
 and to find $T\in\mathcal{C}$ such that $\Addx[T]=\kappa$ ($\addx[T]=\kappa$).
\end{thm}
\begin{proof}
(a) $\Rightarrow$ (b) By Proposition \ref{lem: partiltingescompleto},  it follows that $\p$ is
$\mathcal{X}$-complete,  $\mathcal{X}$-hereditary,  and $\mathcal{T}\cap\mathcal{X}=\kappa={}{}^{\bot}\left(\mathcal{B}\cap\mathcal{X}\right)\cap\mathcal{B}\cap\mathcal{X}.$
In particular $\kappa=\kappa^{\oplus}$ ($\kappa=\kappa^{\oplus_{<\infty}}$).
Moreover,  since $\mathcal{A}\subseteq{}{}{}^{\bot}\left(\mathcal{B}\cap\mathcal{X}\right)$, 
we can conclude that $\pdr[\mathcal{X}][\mathcal{A}]\leq n$ by \cite[Prop. 4.5 (b)]{parte1},  (T1) and (T4). Finally, 
it follows from (T5) that $\kappa$ is precovering in $\mathcal{B}\cap\mathcal{X}$. 

(b) $\Rightarrow$ (c) Let $\alpha$ be an $\mathcal{X}$-injective relative cogenerator
in $\mathcal{X}$. We claim that $\mathcal{B}\cap\mathcal{X}=\kappa^{\bot}\cap\mathcal{X}$.
Indeed,  since $\kappa\subseteq\mathcal{A}$ and $\idr[\mathcal{A}][\mathcal{B}\cap\mathcal{X}]=0$, 
we have $\mathcal{B}\cap\mathcal{X}\subseteq\kappa^{\bot}\cap\mathcal{X}.$ Let us show 
that $\kappa^{\bot}\cap\mathcal{X}\subseteq \mathcal{B}\cap\mathcal{X}.$ Consider
 $X\in\kappa^{\bot }\cap\mathcal{X}$. Then,  by Lemma \ref{lem: lema previo a 2.80} (a),  
there is an exact sequence 
$0\rightarrow X\overset{f_{0}}{\rightarrow}B_{0}\overset{f_{1}}{\rightarrow}B_{1}\overset{f_{2}}{\rightarrow}...\overset{f_{k}}{\rightarrow}B_{k}\rightarrow0$
such that $B_{k}\in\kappa$,  $B_{i}\in\mathcal{B}\cap\mathcal{X}$
and $\Cok[f_{i}]\in\mathcal{A}\cap\mathcal{X}$ $\forall i\in[0, k]$.
Hence,  using that $X\in\kappa^{\bot}\cap\mathcal{X}$ and $B_{i}\in\mathcal{B}\cap\mathcal{X}\subseteq\kappa^{\bot}\cap\mathcal{X}$, 
we have that $\im[f_{j}]\in\kappa^{\bot}\cap\mathcal{X}$ $\forall j\in[1, k-1]$.\\
Let us prove,  by induction on $k$,  that $X\in\mathcal{B}$. Indeed, 
if $k=0$ then $X\cong B_{0}\in\mathcal{B}$. For the case $k=1$,  we have
an exact sequence 
$\eta_{1}: \;\suc[X][B_{0}][B_{1}][\, ][\, ]\mbox{, }$
with $X\in\kappa^{\bot}\cap\mathcal{X}$ and $B_{1}\in\kappa$. Hence
$\eta_{1}$ splits and thus $X\in\mathcal{B}$.
\

Let $k>1$. Let us show that $B_{k-1}\in\kappa$. Indeed,  using that $B_{k-1}\in\B\cap\X, $ it is enough to show that $B_{k-1}\in\A.$
From the exact sequence 
$\eta_{k}: \;\suc[K_{k-1}][B_{k-1}][B_{k}][\, ][\;]$, 
we get  $B_{k-1}\in\mathcal{A}$ since $K_{k-1}, B_{k}\in\A$ and $\A$ is closed under extensions; proving that $B_{k-1}\in\kappa.$ 
Now,  by using that $B_{k}\in\kappa$ and $K_{k-1}\in\kappa^{\bot}\cap\mathcal{X}$,  we get that 
$\eta_{k}$ splits and thus $K_{k-1}\in\kappa$. Furthermore,  from the exact
sequence 
\[
0\rightarrow X\overset{f_{0}}{\rightarrow}B_{0}\overset{f_{1}}{\rightarrow}B_{1}\overset{f_{2}}{\rightarrow}...\overset{f_{k-2}}{\rightarrow}B_{k-2}\rightarrow K_{k-1}\rightarrow0, 
\]
and using that $K_{k-1}\in\kappa$,  we have by the inductive
hypothesis,  that $X\in\mathcal{B}.$
\
 
Let us show that $\kappa$ is $n$-$\mathcal{X}$-tilting. In order to do that,  we proceed to verify the axioms from (T0) to (T5).
\begin{description}
\item [{(T0)}] It is clear. 
\item [{(T1)}] Since $\kappa\subseteq\mathcal{A}$,  we have $\pdr[\mathcal{X}][\kappa]\leq\pdr[\mathcal{X}][\mathcal{A}]\leq n\mbox{.}$
\item [{(T2)}] Since $\kappa\subseteq\mathcal{B}\cap\mathcal{X}\subseteq\kappa^{\bot}\cap\mathcal{X}$, 
we have $\kappa\cap\mathcal{X}\subseteq\kappa^{\bot}\cap\mathcal{X}$. 
\item [{(T4)}] By Lemma \ref{lem: lemita inyectivos vs par X-completo a izq}, 
$\alpha$ is an $\mathcal{X}$-injective relative cogenerator in $\mathcal{X}$
such that $\alpha\subseteq\mathcal{B}\cap\mathcal{X}$. Hence,  $\alpha\subseteq\mathcal{X}^{\bot}\cap\kappa^{\bot}$
since $\mathcal{B}\cap\mathcal{X}\subseteq\kappa^{\bot}$.
\item [{(T5)}] By hypothesis,  $\kappa$ is precovering in $\mathcal{B}\cap\mathcal{X}$.
Then,  using that $\kappa^{\bot}\cap\mathcal{X}=\mathcal{B}\cap\mathcal{X}$,  we have that 
every $Z\in\kappa^{\bot}\cap\mathcal{X}$ admits a $\kappa$-precover
$T'\rightarrow Z$ with $T'\in\mathcal{X}$. 
\item [{(T3)}] It follows from Lemma \ref{lem: lema previo a 2.80} (b).
\end{description}

(c) $\Rightarrow$ (a) It is clear.

Assume now that one of the above equivalent conditions hold true. Then we have the
following facts. By Lemma \ref{lem: lema previo a 2.80} (a),  $\coresdimr{\mathcal{B}\cap\mathcal{X}}{\mathcal{X}}{\mathcal{A}\cap\mathcal{X}}\leq\max\left\{ 1, n\right\} $ and thus
 $\mathcal{X}\subseteq(\mathcal{B}\cap\mathcal{X})_{\mathcal{A}\cap\mathcal{X}}^{\vee}\subseteq\mathcal{B}_{\mathcal{X}}^{\vee}$. Moreover
 $\idr[\mathcal{A}\cap\mathcal{X}][\mathcal{X}]=\coresdimr{\mathcal{B}}{\mathcal{X}}{\mathcal{X}}$
by \cite[Cor. 4.12 (a)]{parte1}. On the other hand $\kappa=(\mathcal{A}\cap\mathcal{X})^{\bot}\cap\mathcal{A}\cap\mathcal{X}$
by \cite[Theorem 4.24 (a)]{parte1}; and  $\pdr[\mathcal{X}][\mathcal{A}]\leq n < \infty$
by (b). Then,  by applying  \cite[Theorem 4.24 (a1, a2)]{parte1} on $(\A\cap\X, \B)$, 
$\coresdimr{\mathcal{B}}{\mathcal{X}}{\mathcal{X}}=\pdr[\mathcal{A}\cap\mathcal{X}][\mathcal{A}\cap\mathcal{X}]=\pdr[\mathcal{X}][\mathcal{A}\cap\mathcal{X}]\leq\pdr[\mathcal{X}][\mathcal{A}]\leq n$
and $\mathcal{A}\cap\mathcal{X}\subseteq\kappa^{\vee}$.

Finally,  assume that  $\sigma$ is a (finite) set,  ($\addx[X]$ is precovering
in $X^{\bot}\cap\mathcal{X}$ for every $X\in\kappa$),  $\p$ is right
$\mathcal{X}$-complete,  $^{\bot}(\B \cap \X)\cap \B \cap \X \subseteq \A \cap \B \cap \X$,  $\pdr[\mathcal{X}][\mathcal{A}]\leq n$, 
and $\kappa=\kappa^{\oplus}$ ($\kappa=\kappa^{\oplus_{<\infty}}$).
Since $\sigma\subseteq\mathcal{X}\cap{}^{\bot}\mathcal{X}$,  by Lemma \ref{lem: lema previo a 2.80} (b), 
every $W\in\sigma$ admits an exact sequence 
$0\rightarrow W\stackrel{f_{0}}{\rightarrow}B_{W, 0}\stackrel{f_{1}}{\rightarrow}B_{W, 1}\rightarrow...\stackrel{f_{n}}{\rightarrow}B_{W, n}\rightarrow0$
such that $B_{W, i}\in\mathcal{A}\cap\mathcal{B}\cap\mathcal{X}$ $\forall i\in[0, n]$
and $\Cok[f_{j}]\in\mathcal{A}\cap\mathcal{X}$ $\forall j\in[0, n-1]$.
Consider $T_{W}: =\bigoplus_{i=0}^{n}B_{W, i}$ for every $W\in\sigma$.
Since $\kappa=\kappa^{\oplus}$ ($\kappa=\kappa^{\oplus_{<\infty}}$),  we have $T: =\bigoplus_{W\in\sigma}T_{W}\in\kappa$
and $\Addx[T]\subseteq\kappa$ ($\addx[T]\subseteq\kappa$). We claim
that $\Addx[T]=\kappa$ ($\addx[T]=\kappa$). In order to prove it, 
we must show that $T$ is big (small) $n$-$\mathcal{X}$-tilting.
This is done in the same manner as (b) $\Rightarrow$ (c) was proved.
\

Let us show that $\kappa\subseteq\Addx[T]$ ($\kappa\subseteq\addx[T]$).
Consider $X\in\kappa$. Observe that $X\in\mathcal{B}\cap\mathcal{X}\subseteq T^{\bot}\cap\mathcal{X}\mbox{.}$
Now,  $T^{\bot }\cap\mathcal{X}\subseteq\Gennr[\operatorname{Add}(T)][1][\mathcal{X}]$
($T^{\bot }\cap\mathcal{X}\subseteq\Gennr[\operatorname{add}(T)][1][\mathcal{X}]$)
by Lemma \ref{lem: props T2 con T3' y C2 con C3'}. Hence,  by Lemma \ref{lem: props C2 y T2} (a), 
$\Addx[T]=\Addx[T]\cap\mathcal{X}$ ($\addx[T]=\addx[T]\cap\mathcal{X}$)
is a relative generator in $T^{\bot}\cap\mathcal{X}$. Then,  using that 
$X\in T^{\bot}\cap\mathcal{X}$,  we can build an exact sequence 
$0\rightarrow K_{n}\rightarrow T_{n}\overset{f_{n}}{\rightarrow}T_{n-1}\rightarrow...\overset{f_{1}}{\rightarrow}T_{0}\overset{f_{0}}{\rightarrow}X\rightarrow0\mbox{, }$
with $T_{i}\in\Addx[T]$ ($T_{i}\in\addx[T]$) and $K_{i}: =\Kerx[f_{i}]\in T^{\bot}\cap\mathcal{X}$
$\forall i\in[0, n]$. On the other hand $\pdr[\mathcal{X}][X]\leq\pdr[\mathcal{X}][\kappa]\leq n$
and thus 
$\Extx[1][][K_{n-1}][K_{n}]\cong\Extx[n+1][][X][K_{n}]=0.$  Therefore,  $K_{n}\in\mathcal{B}\cap\mathcal{X}\subseteq(\mathcal{A}\cap\mathcal{X})^{\bot}$
since the exact sequence 
$\suc[K_{n}][T_{n}][K_{n-1}][\, ][\: ]$
splits. Then,  using that $T_{i}\in\kappa\subseteq\mathcal{B}\cap\mathcal{X}\subseteq\left(\mathcal{A}\cap\mathcal{X}\right)^{\bot}\, \forall i\in[0, n]$
and $(\mathcal{A}\cap\mathcal{X})^{\bot}$ is closed by mono-cokernels, 
we have $K_{i}\in\left(\mathcal{A}\cap\mathcal{X}\right)^{\bot}$
$\forall i\in[0, n]$. Finally,  since  $K_{0}\in\left(\mathcal{A}\cap\mathcal{X}\right)^{\bot}$
and $X\in\kappa\subseteq\mathcal{A}\cap\mathcal{X}$,  
the exact sequence 
$\suc[K_{0}][T_{0}][X][\, ][\, ]$
 splits and thus $X\in\Addx[T]$($X\in\addx[T]$).
\end{proof}

\begin{rem}
One of the hypotheses (the small case)  in  Theorem  \ref{rem: obs 2.80'} is that the class $\addx[X]$ is precovering
in $X^{\bot}\cap\mathcal{X}$ $\forall X\in\kappa$. We can give two
examples where such condition is satisfied: 
\begin{enumerate}
\item [(i)] $\mathcal{X}: =\mathcal{FP}_{n}$,  $\mathcal{C}=\Modx[R]$,  $\sigma=\{R\}$
with $R$ an $n$-coherent ring,  see Lemma \ref{lem: anillo conmutativo coherente,  entonces precubiertas}.
\item [(ii)] $R$ is an Artin algebra,  $\mathcal{C}: =\modd[R]$,  $\mathcal{X}\subseteq\modd[R]$
and $\sigma=\{R\}$,  see \cite[Prop. 4.2]{auslander1980preprojective}.
\end{enumerate}
\end{rem}

\begin{lem}\label{tilting is cotorsion} Let $\mathcal{C}$ be an abelian
category with enough projectives and injectives and $\T$ be an  $n$-$\C$-tilting class in $\C.$ Then $({}^{\perp}(\T^\perp), \T^\perp)$ is a complete hereditary cotorsion pair in $\C.$
\end{lem}
\begin{proof} Since $\T$ is $n$-$\C$-tilting  in $\C,$ we get from Theorem \ref{thm: el par n-X-tilting} (c) that $({}^{\bot}(\T^\perp),\T^{\bot})$ is a hereditary $\C$-complete pair in $\C$ which is also a cotorsion one  by \cite[Cor. 3.11]{parte1}.
\end{proof}

\begin{cor}\label{corr-tcchp} Let $\C$ be an AB3 (abelian) category with enough injectives and projectives, $\X=\Add(\X)$ ($\X=\add(\X)$) be a class in $\C$ admitting an $\X$-injective relative cogenerator and an $\X$-projective relative generator. Consider the classes:
\

(a) $\mathsf{HCC}_{n,\X}(\C)$ is the class of all the hereditary cotorsion pairs $(\A,\B)$ in $\C$ which are $\X$-complete, $\pd_\X(\A)\leq n$ and $\omega:=\A\cap\B\cap\X$ with $\omega^\oplus=\omega$ (this condition is dismissed in the small case). We say that $(\A,\B), (\A',\B')\in \mathsf{HCC}_{n,\X}(\C)$ are related, and write $(\A,\B)\sim (\A',\B')$ if $\A\cap\B\cap\X=\A'\cap\B'\cap\X.$
\

(b) $\mathsf{Tilt}_{n,\X}(\C)$ is the class of all the big (small) $n$-$\X$-tilting
classes $\T$ in $\C$ such that $\T\subseteq\X.$
\

Then, the map $\varphi:\mathsf{HCC}_{n,\X}(\C)/\!\!\sim\,\to \mathsf{Tilt}_{n,\X}(\C),\;[(\A,\B)]\mapsto \A\cap\B\cap\X,$ is a bijection  whose inverse is $\psi:\mathsf{Tilt}_{n,\X}(\C)\to \mathsf{HCC}_{n,\X}(\C)/\!\!\sim,\; \T\mapsto[({}^\perp(\T^\perp), \T^\perp)].$
\end{cor}
\begin{proof} We consider the big case only (the small one is similar). Recall that $\C$ is AB4 since it is AB3 and has enough injectives and projectives.
Let $(\A,\B)\in \mathsf{HCC}_{n,\X}(\C)$ and $\omega:=\A\cap\B\cap\X.$ Since $(\A,\B)$ is $\X$-complete and hereditary, it follows that $\omega$ is precovering in $\B\cap\X.$ Moreover, by \cite[Thm. 4.24 (b)]{parte1}, we get 
$\omega={}^\perp(\B\cap\X)\cap\B\cap\X.$ Thus, by Theorem \ref{thm: dimensiones en un par de cotorsion tilting}, we have that $\omega\in \mathsf{Tilt}_{n,\X}(\C)$ and $\B\cap\X=\omega^\perp\cap\X.$ Moreover, by Lemma \ref{lem: inf5} (b) and Proposition \ref{prop: equiv a t3}, we have ${}^\perp(\omega^\perp)\cap\omega^\perp\cap\X=\omega\cap\X=\omega$ and thus $(\A,\B)\sim ({}^\perp(\omega^\perp),\omega^\perp).$ 
\

Let $\T\in\mathsf{Tilt}_{n,\X}(\C).$ Then, from \cite[Prop. 4.5 (b)]{parte1}, 
$\pd_\X( {}^\perp(\T^\perp))\leq\pd_\X(\T)\leq n.$ By Lemma \ref{lem: inf5} (b) and Proposition \ref{prop: equiv a t3}, we have ${}^\perp(\T^\perp)\cap\T^\perp\cap\X=\T\cap\X=\T.$  Moreover, by Lemma \ref{tilting is cotorsion} and Theorem \ref{thm: el par n-X-tilting} (c), we conclude that $({}^\perp(\T^\perp), \T^\perp)\in \mathsf{HCC}_{n,\X}(\C).$
\end{proof}

\begin{thm}\label{thm: teo nuevo p.94}\label{thm: main1-2} Let $\mathcal{C}$
be an AB4 (abelian) category,  $\mathcal{X}=\mathcal{X}^{\oplus}\subseteq\mathcal{C}$
($\mathcal{X}\subseteq\mathcal{C}$) be a right thick class admitting an $\mathcal{X}$-projective relative generator in $\mathcal{X}$,  and let $\mathcal{B}\subseteq\mathcal{C}$
be right thick. Then,  the following conditions are equivalent.
\begin{itemize}
\item[$\mathrm{(a)}$] There is a big (small) $n$-$\mathcal{X}$-tilting class $\mathcal{T}\subseteq\mathcal{X}$
such that $\mathcal{B}\cap\mathcal{X}=\mathcal{T}^{\bot}\cap\mathcal{X}$.
\item[$\mathrm{(b)}$] $\mathcal{B}$ satisfies the following conditions: 

\begin{itemize}
\item [$\mathrm{(b0)}$] $\mathcal{B}\cap\mathcal{X}\cap{}^{\bot}(\mathcal{B}\cap\mathcal{X})$
is closed under coproducts (this condition is dismissed in the small case);
\item[$\mathrm{(b1)}$] there is an $\mathcal{X}$-injective relative cogenerator
in $\mathcal{X}$;
\item [$\mathrm{(b2)}$] $\mathcal{B}\cap\mathcal{X}$ is special preenveloping in $\mathcal{X}$;
\item [$\mathrm{(b3)}$] $\pdr[\mathcal{X}][^{\bot}(\mathcal{B}\cap\mathcal{X})]\leq n$;
\item [$\mathrm{(b4)}$] $\mathcal{B}\cap\mathcal{X}\cap{}^{\bot}(\mathcal{B}\cap\mathcal{X})$
is precovering in $\mathcal{B}\cap\mathcal{X}$. 
\end{itemize}
\end{itemize}
Moreover,  if $\mathrm{(a)}$ or $\mathrm{(b)}$ holds true,  we have that $\mathcal{T}=\mathcal{B}\cap\mathcal{X}\cap{}^{\bot}(\mathcal{B}\cap\mathcal{X})$
and $\mathcal{B}$ is $\mathcal{X}$-coresolving. 
\end{thm}
\begin{proof} (a) $\Rightarrow$ (b) Let $\mathcal{T}\subseteq\mathcal{X}$ be a big
(small) $n$-$\mathcal{X}$-tilting class such that $\mathcal{B}\cap\mathcal{X}=\mathcal{T}^{\bot}\cap\mathcal{X}$.
Let us verify the conditions of (b).
\begin{enumerate}
\item [(b0)] Since $\mathcal{T}\subseteq\mathcal{X}$,  by Proposition \ref{prop: (a)} (a, d), 
we have $\mathcal{B}\cap\mathcal{X}\cap{}^{\bot}(\mathcal{B}\cap\mathcal{X})=\mathcal{T}$ and 
 thus (b0) holds true.
\item [(b1)] It follows by (T4).
\item [(b2)] By Theorem \ref{thm: el par n-X-tilting} (c),  the pair $({}^{\bot}(\mathcal{T}^{\bot}), \mathcal{T}^{\bot})$
is right $\mathcal{X}$-complete. Hence,  (b2) follows from Theorem \ref{thm: el par n-X-tilting} (a).
\item [(b3)] It follows from \cite[Prop. 4.5 (b)]{parte1}, 
(T4) and (T1).
\item [(b4)] It follows from (T5) and Proposition \ref{prop: (a)} (a, d). 
\end{enumerate}

(b) $\Rightarrow$ (a) Assume the conditions of (b) hold true,   and let $\alpha$ be an $\X$-injective relative cogenerator in $\X.$ We assert that 
$\alpha\subseteq\mathcal{B}\cap\mathcal{X}.$ Indeed,  using (b2),  we have that the pair $(\X, \B)$ is right $\X$-complete. Hence by Lemma \ref{lem: lemita inyectivos vs par X-completo a izq},  our assertion follows.
Also note that,  since $\B$ and $\X$ are right thick,   $\mathcal{B}$ is closed under extensions and  mono-cokernels  in 
$\mathcal{B}\cap\mathcal{X}$. Therefore,  $\mathcal{B}$ is $\mathcal{X}$-coresolving,  and thus,  by \cite[Lem. 3.4]{parte1} and (b2),   $({}{}^{\bot}(\mathcal{B}\cap\mathcal{X}), \mathcal{B}\cap\mathcal{X})$
is a right $\mathcal{X}$-complete left cotorsion pair in $\mathcal{X}$. Now,  by (b0),  (b3)
and (b4),  we have that $({}{}^{\bot}(\mathcal{B}\cap\mathcal{X})\cap\mathcal{X}, \mathcal{B}\cap\mathcal{X})$
satisfy the conditions of Theorem \ref{thm: dimensiones en un par de cotorsion tilting} (b). Hence, 
by Theorem \ref{thm: dimensiones en un par de cotorsion tilting} (a),  the item
(a) is satisfied.
\end{proof}

\begin{rem}\label{rem:  obs 2.90'-1}\label{rem:  obs1 teo nuevo} Assume the hypotheses
of Theorem \ref{thm: main1-2} with $\mathcal{C}$ AB4,  $\mathcal{X}=\mathcal{X}^{\oplus}$
right thick and $\sigma$ a set, where $\Add(\sigma)$ is an $\X$-projective relative generator in $\X.$ Then,  we can dismiss the condition (b4) from (b).
Moreover,  if Theorem \ref{thm: main1-2}(a) is satisfied,  we can choose $\mathcal{T}\subseteq\mathcal{X}$
such that $\mathcal{T}=\Addx[T]$ with $T\in\mathcal{B}\cap\mathcal{X}\cap{}^{\bot}(\mathcal{B}\cap\mathcal{X})$.
Indeed,  this is a consequence of the last sentence in Theorem \ref{thm: dimensiones en un par de cotorsion tilting} and the fact that $\Addx[T]$ is precovering.
\end{rem}

\begin{rem}\label{rem: obs 2.90''-1}\label{thm: main1-1-1}\label{rem:  obs2 teo nuevo}
Assume the hypotheses of Theorem \ref{thm: main1-2} with $\mathcal{C}$ abelian, 
$\mathcal{X}$ right thick and $\sigma$ a finite set, where $\add(\sigma)$ is an $\X$-projective relative generator in $\X.$ Then,  we can
replace condition (b4) in (b)
for the following condition: 
\begin{description}
\item [{({$*$})}] $\addx[T]$ is precovering in $T^{\bot}\cap\mathcal{X}, $  for each $T\in\mathcal{B}\cap\mathcal{X}\cap {}^{\bot}(\mathcal{B}\cap\mathcal{X}).$
\end{description}
Moreover,  in such case,  we can choose  $\mathcal{T}=\addx[T]$ with $T\in\mathcal{B}\cap\mathcal{X}\cap{}^{\bot}(\mathcal{B}\cap\mathcal{X})$.
Indeed,  this is a consequence of the last sentence in Theorem \ref{thm: dimensiones en un par de cotorsion tilting}.\end{rem}

\begin{cor}\label{cor: biyeccion tiltilng -1}\label{cor:  coro1 teo nuevo} Let
$\mathcal{C}$ be an AB4 category with enough injectives,  $\mathcal{X}=\mathcal{X}^{\oplus}$
be a right thick class admitting an $\mathcal{X}$-injective relative
cogenerator in $\mathcal{X}$,  and let $\sigma$ be a set such that $\Addx[\sigma]$
is an $\mathcal{X}$-projective relative generator in $\mathcal{X}$.
Consider the following classes: 
\begin{description}
\item [{$\mathcal{T}_{n, \mathcal{X}}$}] consisting of all the objects $T\in\mathcal{X}$
that are big $n$-$\mathcal{X}$-tilting; 
\item [{$\mathcal{TP}_{n, \mathcal{X}}$}] consisting of all the right $\mathcal{X}$-complete and
left cotorsion pairs $\p$ such that $\pdr[\mathcal{X}][^{\bot}(\mathcal{B}\cap\mathcal{X})]\leq n$, 
$^{\bot}(\mathcal{B}\cap\mathcal{X})\cap\mathcal{B}\cap\mathcal{X}$
is closed under coproducts and $\mathcal{B}$ is right thick.
\end{description}
Consider the equivalence relation $\sim$ in $\mathcal{T}_{n, \mathcal{X}}$,  where
 $T\sim S$ if $\mbox{ \ensuremath{T^{\bot}\cap\mathcal{X}=S^{\bot}\cap\mathcal{X}}}$;
and the equivalence relation $\approx$ in $\mathcal{TP}_{n, \mathcal{X}}$,  where
 $\p\approx\p[\mathcal{A}'][\mathcal{B}']$ if 
 $\mathcal{B}\cap\mathcal{X}=\mathcal{B}'\cap\mathcal{X}\mbox{.}$
Then,  there is a bijective map
$$\phi: \mathcal{T}_{n, \mathcal{X}}/\!\!\sim\;\;\longrightarrow\;\mathcal{TP}_{n, \mathcal{X}}/\!\!\approx, \quad
[T]\mapsto[({}^{\bot}(T^{\bot}), T^{\bot})].$$
\end{cor}
\begin{proof}
For each $T\in\mathcal{T}_{n, \mathcal{X}}, $ we consider the pair $\mathcal{P}_{T}: =(^{\bot}(T^{\bot}), T^{\bot}).$
Let us show that $\mathcal{P}_{T}\in\mathcal{TP}_{n, \mathcal{X}}$.
To begin with,  it is clear that $T^{\bot}$ is right thick,  and 
by Theorem \ref{thm: el par n-X-tilting} (c),  $\mathcal{P}_{T}$ is $\mathcal{X}$-complete. On the other hand,  
by \cite[Lemma 3.4]{parte1}, 
$\mathcal{P}_{T}$ is left cotorsion since $\mathcal{C}$ has enough
injectives. Moreover,  by Proposition \ref{prop: equiv a t3} and Lemma \ref{lem: props C2 y T2},  
$T^{\bot}\cap{}^{\bot}(T^{\bot}\cap\mathcal{X})\cap\mathcal{X}=\Addx[T]\cap\mathcal{X}=\Addx[T]$
 and thus $T^{\bot}\cap{}^{\bot}(T^{\bot}\cap\mathcal{X})\cap\mathcal{X}$
is closed under coproducts. Finally,  by Proposition \ref{prop: (b)} (b),  $\pdr[\mathcal{X}][^{\bot}(T^{\bot}\cap\mathcal{X})]\leq n$. 

Moreover,  for $T, S\in\mathcal{T}_{n, \mathcal{X}}$,  we note that
$[T]=[S]\: \Leftrightarrow\: [\mathcal{P}_{T}]=[\mathcal{P}_{S}]\mbox{.}$
Therefore,  the map $\phi$
is well-defined and injective. It remains to show that $\phi$ is
surjective.

Let $\p\in\mathcal{TP}_{n, \mathcal{X}}$. In particular,  $\p$ satisfy
conditions (b0),  (b1),  (b2) and (b3) of Theorem \ref{thm: main1-2} (b). Then,  by Remark \ref{rem:  obs 2.90'-1}, 
Theorem \ref{thm: main1-2} (a) is satisfied and thus we can find $T\in\mathcal{T}_{n, \mathcal{X}}$
such that $\mathcal{B}\cap\mathcal{X}=T^{\bot}\cap\mathcal{X}$. Therefore
$\p\approx({}^{\bot}(T^{\bot}), T^{\bot})$ and then $\phi([T])=[(\A, \B)].$
\end{proof}

A similar result as above can be proved for small $n$-$\mathcal{X}$-tilting
objects.
 
\begin{cor}\label{cor:  coro2 teo nuevo} Let $\mathcal{C}$ be an abelian category
with enough injectives,  $\mathcal{X}\subseteq\mathcal{C}$ be a right
thick class admitting an $\mathcal{X}$-injective relative cogenerator
in $\mathcal{X}$,  and let $\sigma$ be a finite set such that $\addx[\sigma]$
is an $\mathcal{X}$-projective relative generator in $\mathcal{X}.$
Consider the following classes: 
\begin{description}
\item [{$s\mathcal{T}_{n, \mathcal{X}}$}] consisting of all the objects $T\in\mathcal{X}$
that are small $n$-$\mathcal{X}$-tilting; 
\item [{$s\mathcal{TP}_{n, \mathcal{X}}$}] consisting of all the right $\mathcal{X}$-complete and
left cotorsion pairs $\p$ such that $\pdr[\mathcal{X}][^{\bot}(\mathcal{B}\cap\mathcal{X})]\leq n$
and $\mathcal{B}$ is right thick.
\end{description}
Consider the equivalence relation $\sim$ in $s\mathcal{T}_{n, \mathcal{X}}$, 
where $T\sim S$ if $\mbox{ \ensuremath{T^{\bot}\cap\mathcal{X}=S^{\bot}\cap\mathcal{X}}}$;
and the equivalence relation $\approx$ in $s\mathcal{TP}_{n, \mathcal{X}}$, 
where $\p\approx\p[\mathcal{A}'][\mathcal{B}']$ if $\mathcal{B}\cap\mathcal{X}=\mathcal{B}'\cap\mathcal{X}\mbox{.}$
Then,  there is an injective map
\[
\phi: s\mathcal{T}_{n, \mathcal{X}}/\!\!\sim\;\;\longrightarrow\; s\mathcal{TP}{}_{n, \mathcal{X}}/\!\!\approx, \quad[T]\mapsto[({}^{\bot}(T^{\bot}), T^{\bot})].
\]
 Furthermore,  $\phi$ is bijective if every $T\in$ $\mathcal{B}\cap\mathcal{X}\cap{}^{\bot}(\mathcal{B}\cap\mathcal{X})$
satisfies that $\addx[T]$ is precovering in $T^{\bot}\cap\mathcal{X}.$
\end{cor}
\begin{proof}
Using Theorem \ref{thm: main1-2} and Remark \ref{rem: obs 2.90''-1},  the proof follows
by similar arguments as in the proof of Corollary \ref{cor: biyeccion tiltilng -1}.
\end{proof}

In the case of a ring $R, $ we get the following result.

\begin{cor}\label{cor: coro3 teo nuevo} Let $R$ be a ring,  $\mathcal{X}=\mathcal{X}^{\oplus}$
be a right thick class in $\Modx[R]$ admitting an $\mathcal{X}$-injective
relative cogenerator in $\mathcal{X}$,  and let $\sigma$ be a set such
that $\Addx[\sigma]$ is an $\mathcal{X}$-projective relative generator
in $\mathcal{X}$. Then,  for every $\mathcal{B}\subseteq\Modx[R]$, 
the following conditions are equivalent. 
\begin{itemize}
\item[$\mathrm{(a)}$] There is a big $n$-$\mathcal{X}$-tilting object $T\in\mathcal{X}$
such that $\mathcal{B}\cap\mathcal{X}=T^{\bot}\cap\mathcal{X}$.
\item[$\mathrm{(b)}$] $\mathcal{B}$ satisfies the following conditions: 

\begin{itemize}
\item[$\mathrm{(b0)}$]   $\mathcal{B}\cap\mathcal{X}\cap{}^{\bot}(\mathcal{B}\cap\mathcal{X})$
is closed under coproducts;
\item[$\mathrm{(b1)}$]  $\mathcal{B}$ is special preenveloping in $\mathcal{X}$;
\item[$\mathrm{(b2)}$]   $\pdr[\mathcal{X}][^{\bot}(\mathcal{B}\cap\mathcal{X})]\leq n$;
\item[$\mathrm{(b3)}$]   $\mathcal{B}$ is right thick in $\Modx[R]$.
\end{itemize}
\end{itemize}
\end{cor}
\begin{proof}
It follows from Corollary \ref{cor:  coro1 teo nuevo}.
\end{proof}

In the case of an Artin algebra $\Lambda, $ we  get the following result.
\begin{cor}
\label{cor: coro4 teo nuevo} Let $\Lambda$ be an Artin algebra, 
$\mathcal{X}$ be a right thick class in $\modd[\Lambda]$ admitting
an $\mathcal{X}$-injective relative cogenerator in $\mathcal{X}$, 
and let $\sigma$ be a set such that $\addx[\sigma]$ is an $\mathcal{X}$-projective
relative generator in $\mathcal{X}$. Then,  for any $\mathcal{B}\subseteq\modd[\Lambda]$, 
the following conditions are equivalent. 
\begin{itemize}
\item[$\mathrm{(a)}$]  There is a small $n$-$\mathcal{X}$-tilting object $T\in\mathcal{X}$
such that $\mathcal{B}\cap\mathcal{X}=T^{\bot}\cap\mathcal{X}$.
\item[$\mathrm{(b)}$]  $\mathcal{B}$ satisfies the following conditions: 

\begin{itemize}
\item[$\mathrm{(b1)}$]   $\mathcal{B}$ is special preenveloping in $\mathcal{X}$;
\item[$\mathrm{(b2)}$]   $\pdr[\mathcal{X}][\operatorname{mod}(\Lambda)\cap{}{}^{\bot}(\mathcal{B}\cap\mathcal{X})]\leq n$;
\item[$\mathrm{(b3)}$]  $\mathcal{B}$ is right thick in $\modd[\Lambda]$.
\end{itemize}
\end{itemize}
\end{cor}
\begin{proof}
It follows from Corollary \ref{cor:  coro2 teo nuevo} and \cite[Prop. 4.2]{auslander1980preprojective}.
\end{proof}

\section{$n$-$\X$-tilting versus other notions of tilting}\label{sec: Examples} 

In this section, we will show that $n$-$\X$-tilting offers a unified framework of different previous notions of tilting which are in the literature.

\subsection{$\infty$-tilting objects and pairs}

 Leonid Positselski and Jan {\v{S}}t'ov{\'\i}{\v{c}}ek defined in \cite{positselski2019tilting} the notion of 
 $\infty$-tilting object and $\infty$-tilting pair. In this section,  we recall these notions and give and interpretation in terms of $n$-$\X$-tilting theory.  We also recall that an AB3{*} category,  having an injective cogenerator,  is  AB3 \cite[Ex. III.2]{mitchell}.
 
\begin{defn}\cite[Sect. 2]{positselski2019tilting}  Let $\mathcal{A}$ be an AB3{*} category which has an injective cogenerator. 
 An object $T\in\mathcal{A}$ is \textbf{$\infty$-tilting} if the following conditions hold true: 
\begin{description}
\item [{($\infty$-T1)}] $\Add(T)\subseteq T^{\bot}.$
\item [{($\infty$-T2)}] $\Inj(\mathcal{A})\subseteq(\Add(T), T^{\bot})_{\infty}^{\wedge}$.
\end{description}
\end{defn}

\begin{defn}\cite[Sect. 3]{positselski2019tilting}  Let $\mathcal{A}$ be an AB3{*} category having an injective cogenerator,  $T\in\A$ and 
$\mathcal{E}\subseteq\mathcal{A}.$ The pair $(T, \mathcal{E})$ is \textbf{$\infty$-tilting} if the following conditions hold true: 
\begin{description}
\item [{($\infty$-PT1)}] The class $\mathcal{E}$ is coresolving.
\item [{($\infty$-PT2)}] $\Add(T)\subseteq\mathcal{E}\subseteq T^{\bot_{1}}.$
\item [{($\infty$-PT3)}] Any $\Add(T)$-precover $\alpha: T'\rightarrow E$ of 
$E\in\mathcal{E}$ is an epimorphism and $\Ker(\alpha)\in\mathcal{E}$.
\end{description}
\end{defn}

\begin{rem}\cite[Sect. 3]{positselski2019tilting}
\label{rem: infty tilting} For an $\infty$-tilting pair $(T, \mathcal{E})$
in an  AB3{*} category $\mathcal{A}$,  which has an injective cogenerator 
$J$,  the following statements hold true: 
\begin{itemize}
\item[$\mathrm{(a)}$] $\Prod(J)=\Inj(\mathcal{A})\subseteq\mathcal{E}\subseteq T^{\bot}.$
\item[$\mathrm{(b)}$] $\Add(T)\subseteq T^{\bot}.$
\item[$\mathrm{(c)}$] $\Add(T)$ is a relative  $\mathcal{E}$-projective generator in  $\mathcal{E}.$  
\item[$\mathrm{(d)}$] $\Prod(J)$ is an $\mathcal{E}$-injective relative cogenerator in  
$\mathcal{E}$.
\end{itemize}
\end{rem}

The connection between $\infty$-tilting objects and pairs is as follows.
 
\begin{lem}\cite[Lem. 3.1]{positselski2019tilting} For a bicomplete abelian category 
$\A, $ which has an injective cogenerator,  and $T\in\A, $ the following statements hold true.
\begin{itemize}
\item[$\mathrm{(a)}$] There exists a class $\mathcal{E}\subseteq\mathcal{A}$ such that $(T, \mathcal{E})$ is an $\infty$-tilting pair if,  and only if,  
$T$ is an $\infty$-tilting object.
\item[$\mathrm{(b)}$] If $T$ is an $\infty$-tilting object,  then $(T, (\Add(T), T^{\bot})_{\infty}^{\wedge})$ is an $\infty$-tilting pair.
\item[$\mathrm{(c)}$] If $(T, \mathcal{E})$ is an $\infty$-tilting pair,  then  $\mathcal{E}\subseteq(\Add(T), T^{\bot})_{\infty}^{\wedge}.$
\end{itemize}
\end{lem}

In what follows,  we show that the $\infty$-tilting pairs are contained in the $n$-$\mathcal{X}$-tilting theory.

\begin{prop}
Let $\mathcal{A}$ be an AB3{*} category having an injective cogenerator,  and let $(T, \mathcal{E})$ be an $\infty$-tilting pair. Then $T$ is a big $0$-$\mathcal{E}$-tilting object such that $\Add(T)=\mathcal{E}\cap{}^{\bot}\mathcal{E}$
and $T^{\bot}\cap\mathcal{E}=\operatorname{Gen}_{1}^{\mathcal{E}}(T)$.
\end{prop}
\begin{proof} Let $J\in\A$ be an injective cogenerator. By 
\cite[Chap. 3,  Cor. 2.9,  p.73]{Popescu},  we get that $\A$ is AB4. By taking into account Remark \ref{rem: infty tilting},  we can show the following: 
\begin{description}
\item [{(T1)}] $\pdr[\operatorname{Add}(T)][\mathcal{E}]=\pdr[T][\mathcal{E}]=0$ since $\mathcal{E}\subseteq T^{\bot}.$
\item [{(T2)}] $\Add(T)\cap\mathcal{E}\subseteq T^{\bot}=\Add(T)^{\bot}$ since $\Add(T)\subseteq\mathcal{E}.$
\item [{(T3)}] We know that $\Add(T)$ is an $\mathcal{E}$-projective relative generator 
in $\mathcal{E}$ and thus $\Add(T)\subseteq(\Add(T))_{\mathcal{E}}^{\vee}$.
\item [{(T4)}] $\mbox{Prod}(J)=\Inj(\mathcal{A})$ is an $\mathcal{E}$-injective
relative cogenerator in $\mathcal{E}$ and $T^{\bot}=\left(\Add(T)\right)^{\bot}$.
\item [{(T5)}] Notice that $\Add(T)$ is precovering since $\mathcal{A}$ is AB3. Using now that $\Add(T)\subseteq\mathcal{E}, $ any 
$X\in T^{\bot}\cap\mathcal{E}$ has an $\Add(T)$-precover
$A\rightarrow X$,  with $A\in\mathcal{E}$. 
\end{description}

Finally,  it is enough to use Corollary \ref{cor: coronuevo pag 55} and Theorem \ref{thm: el par n-X-tilting}
to conclude that $\Add(T)=\mathcal{E}\cap{}^{\bot}\mathcal{E}$ and
$T^{\bot}\cap\mathcal{E}=\operatorname{Gen}_{1}^{\mathcal{E}}(T)\cap\mathcal{E}=\operatorname{Gen}_{1}^{\mathcal{E}}(T)$
since $\mathcal{E}$  is closed under mono-cokernels.
\end{proof}

\subsection{Miyashita tilting modules}

In this section we review the tilting theory developed by Yoichi Miyashita
in \cite{miyashita}. Recall that,  for a ring $R$,  $\projx[R]$ (resp. $\inj(R)$) is
the class of finitely generated projective (resp. injective) left $R$-modules.

\begin{defn}\label{def: miyashita tilting} Let $R$ be a ring. A left $R$-module
$T$ is \textbf{ Miyashita $n$-tilting} if the following conditions hold true.
\begin{description}
\item [{(MT1)}] $T\in\projx[R]_{n}^{\wedge}.$
\item [{(MT2)}] $T\in T^{\bot}.$
\item [{(MT3)}] $R\in\addx[T]_{n}^{\vee}$.
\end{description}
\end{defn}

\begin{prop}\label{prop: tilting miyashita} Let $R$ be a finitely generated $S$-algebra,  where
$S$ is a commutative noetherian ring,  and let $T\in\modd[R]$. Then,  the
following conditions are equivalent.
\begin{itemize}
\item[$\mathrm{(a)}$] $T$ is a Miyashita $n$-tilting module and $\inj(R)$ is a
relative cogenerator in $\modd[R].$
\item[$\mathrm{(b)}$] $T$ is a big $n$-$\modd$-tilting object.
\item[$\mathrm{(c)}$] $T$ is a small $n$-$\modd$-tilting object.
\end{itemize}
Moreover,  if $\inj(R)$ is a
relative cogenerator in $\modd[R]$ and $T$ satisfies $\mathrm{(MT1)}$ and $\mathrm{(MT2)}, $ then $T$ satisfies $\mathrm{(MT3)}$ if and only if 
$\coresdim_{\add(T)}({}_RR)<\infty.$
\end{prop}
\begin{proof}
Observe that $\modd$ is a thick abelian subcategory of $\Mod(R)$ by Corollary \ref{cor: finitamente generado es compacto-1} (b). Moreover, (MT1) and (T1) are equivalent since $\pd_{\modu(R)}(T)=\pd(T).$
\

(a) $\Rightarrow$ (b) Let $T$ be a Miyashita $n$-tilting module and
$\inj(R)$ be a relative cogenerator in $\modd[R]$. Notice that (T4) is trivial and  
(T5) follows from Lemma \ref{lem:  addT precub es AddT precub}. Finally, 
(T3) follows from (MT3) and Remark \ref{rem: T3 simple}. 
\

(b) $\Rightarrow$ (a) Let $T$ be an $n$-${}\modd{}$-tilting object.We only show (MT3). Indeed,  by Theorem \ref{thm: el par n-X-tilting} (c),  there is an exact sequence 
$\suc[R][M_{0}][X_{0}]\mbox{, }$
where $M_{0}\in T^{\bot}\cap{}\modd[R]{}$ and $X_{0}\in{}{}^{\bot}(T^{\bot})\cap{}\modd[R]{}.$
Moreover,  using $R\in{}{}^{\bot}(T^{\bot})$,   by Lemma \ref{lem: inf5} (b) we get
$M_{0}\in{}{}^{\bot}(T^{\bot})\cap T^{\bot}\cap{}\modd[R]{}=\Addx[T]\cap{}\modd[R]{}=\addx[T].$ 
By repeating the above argument,  we can build (inductively) an exact sequence
$\suc[R][M_{0}\rightarrow\cdots\rightarrow M_{k}][X_{k}]\mbox{, }$
with  $X_{i}\in{}{}^{\bot}(T^{\bot})\cap{}\modd[R]$ and $M_{i}\in\addx[T]$ $\forall i\in[1, n]$.
Finally,  $X_{n}\in T^{\bot}\cap{}{}^{\bot}(T^{\bot})\cap{}\modd[R]{}=\addx[T]$
by \cite[Prop. 2.7]{parte1}. 
\

(b) $\Leftrightarrow$ (c) It follows from Corollary \ref{cor: coro2 p80}.\end{proof}

By using Proposition \ref{prop: tilting miyashita} together with the main results
in this paper,  we can infer well-known properties  of Miyashita tilting
modules.  

\subsection{Miyashita tilting for modules of type $FP_{n}$}

In this section we study the left $n$-coherent rings and the left modules of type 
$FP_{n+1}.$ We characterize when some $T\in\Mod(R)$ is a big $n$-$\mathcal{F}\mathcal{P}_{n+1}$-tilting object.
\

 Let $R$ be a ring. Following \cite[Section 1]{bravo2017finiteness},  we recall that  $M\in\Mod(R)$ is called \textbf{finitely $n$-presented} (or of type $FP_n$) if it admits an exact sequence 
$F_{n}\rightarrow F_{n-1}\rightarrow\cdots\rightarrow F_{0}\rightarrow M\rightarrow0$
with $F_{i}\in\projx[R]\, \forall i\in[0, n]$. The class of all the left  $R$-modules of type $FP_n$ is denoted by $\mathcal{FP}_{n}(R).$ Note that $\mathcal{FP}_0(R)=\modu(R).$
 An $M\in\Mod(R)$ is called \textbf{finitely $\infty$-presented} (or of type $FP_\infty$)
if it admits an exact sequence 
$\cdots\rightarrow F_{n}\rightarrow F_{n-1}\rightarrow\cdots\rightarrow F_{0}\rightarrow M\rightarrow0\mbox{, }$
with $F_{i}\in\projx[R]\, \forall i\geq0$. 
The class of all the left $R$-modules of type $FP_\infty$  is denoted by 
$\mathcal{FP}_{\infty}(R).$ Note that
$\mathcal{FP}_{\infty}(R)=\projx[R]_{\infty}^{\wedge}$.

\begin{lem}\cite[Prop. 1.7]{bravo2017finiteness}\label{lem: propiedades de cerradura FPn}
Let $R$ be a ring. Then,  $\mathcal{FP}_{n}(R)$ is right thick and $\mathcal{FP}_{\infty}(R)$
is thick.\end{lem}

\begin{lem}\cite[Lem. 2.11]{bravo2019locally}\label{lem: FPn es cerrado por n-cocientes}
Let $R$ be a ring and $C\in\Mod(R)$ be such that there is an exact sequence
$
F_{n}\rightarrow\cdots\rightarrow F_{1}\rightarrow F_{0}\rightarrow C\rightarrow0\mbox{, }
$
where $F_{i}\in\mathcal{FP}_{n}(R)$ $\forall i\in[0, n].$ Then $C\in\mathcal{FP}_{n}(R).$
\end{lem}

We recall from \cite[Def. 2.2]{bravo2017finiteness} that a ring $R$ is left 
\textbf{$n$-coherent} if $\mathcal{FP}_{n}(R)=\mathcal{FP}_{n+1}(R).$

\begin{lem}\cite[Cor. 2.6]{bravo2017finiteness}\label{lem: n-coherente,  FPn cerrado por nucleos de epis}
Let $R$ be a left $n$-coherent ring. Then $\mathcal{FP}_{n}$ is closed
under epi-kernels.
\end{lem}

\begin{lem}\label{lem: anillo conmutativo coherente,  entonces precubiertas}
For an $n$-coherent commutative ring $R$ and $T\in\mathcal{FP}_{n}(R), $ the following statements hold true.
\begin{itemize}
\item[$\mathrm{(a)}$] $\Hom_R(T, X)\in\mathcal{FP}_{n}(R)$ $\forall\;X\in T^{\bot}\cap\mathcal{FP}_{n}(R).$ 
\item[$\mathrm{(b)}$] Every $X\in T^{\bot}\cap\mathcal{FP}_{n}(R)$ admits an $\addx[T]$-precover. Moreover,  such $\addx[T]$-precover is an $\Addx[T]$-precover.
\end{itemize}
\end{lem}
\begin{proof} (a) There is a family $\left\{ \suc[K_{i+1}][R^{m_{i}}][K_{i}]\right\} _{i=0}^{n}$ of exact sequences in $\Mod(R), $ where $K_{0}=T$ and $m_{i}\in\mathbb{N}\;\forall i\in[0, n], $ since $T\in\mathcal{FP}_{n}(R).$
\

Let $X\in T^{\bot}\cap\mathcal{FP}_{n}(R).$ Then $\Ext_R^1(K_{i}, X)\simeq\Ext^{1+i}_R(T, X)=0\;\forall i\in[0, n], $ and thus,  by applying the functor $\Hom_R(-, X)$ to the above family of exact sequences,  we get the family 
$\left\{ 0\to \Hom_R(K_i, X)\to X^{m_{i}}\to \Hom_R(K_{i+1}, X)\to 0\right\} _{i=0}^{n}$ of exact sequences in $\Mod(R).$ From this family,  we get the exact sequence
\[
0\rightarrow\Hom_R(T, X)\rightarrow X^{m_{0}}\rightarrow X^{m_{1}}\rightarrow\cdots\rightarrow X^{m_{n}}\rightarrow\Hom_R(K_{n+1}, X)\rightarrow0.
\]

Then by Lemmas \ref{lem: propiedades de cerradura FPn} and  \ref{lem: FPn es cerrado por n-cocientes},  it follows that $\Hom_R(K_{n+1}, X)\in\mathcal{FP}_{n}(R).$ Thus,  by using recursively Lemmas \ref{lem: propiedades de cerradura FPn} and \ref{lem: n-coherente,  FPn cerrado por nucleos de epis},  we get that 
 $\Hom_R(T, X)\in\mathcal{FP}_{n}(R)$.
\

(b) Let $X\in T^{\bot}\cap\mathcal{FP}_{n}(R).$ Then,  by (a),  $\Homx[R][T][X]\in\modu(R)$ and let $\{f_{1}, \cdots, f_{k}\}$ be a finite generating set. It is 
straightforward to show that $\alpha: =(f_{1}, \cdots, f_{k}): T^{k}\rightarrow X$
is an $\addx[T]$-precover. The second statement in (b) follows from Lemma \ref{lem:  addT precub es AddT precub}.
\end{proof}

\begin{lem}\label{lem: lemita miyashita} Let $R$ be a ring and $T\in \modd[R]$ be  such that $T\in T^\perp$ and $T\in\proj(R)^\wedge_n.$ Then,  the following statements hold true.
\begin{itemize}
\item[$\mathrm{(a)}$] $\mathcal{Q}: =({}^{\bot}T\cap\mathcal{FP}_{\infty}(R), \addx[T])_{\infty}^{\vee}$
is left thick. 
\item[$\mathrm{(b)}$] $\left(\addx[T]\right)_{\mathcal{FP}_{\infty}(R)}^{\vee}=\addx[T]^{\vee}=\left\{ M\in\mathcal{Q}\;|\: \;\idr[\mathcal{Q}][M]<\infty\right\} $
is left thick.
\end{itemize}
\end{lem}
\begin{proof} By Lemma \ref{lem: propiedades de cerradura FPn},  we know that  $\mathcal{FP}_{\infty}(R)$ is thick in $\Modx[R]$. Note that
$T\in\mathcal{FP}_{\infty}(R)$ and thus $\addx[T]\subseteq\mathcal{FP}_{\infty}(R).$
Moreover $\addx[T]\subseteq T^{\bot}\cap\mathcal{F}\mathcal{P}_{\infty}$ since $T\in T^{\bot}.$ Hence, 
the result follows from \cite[Cor. 4.21]{parte1}.
\end{proof}

\begin{thm}\label{prop: n-FPn+1-tilting} Let $R$ be a left $n$-coherent ring and
$T\in\mathcal{F}\mathcal{P}_{n+1}(R)$. Consider the following statements: 
\begin{itemize}
\item[$\mathrm{(a)}$] $T$ is a Miyashita $n$-tilting $R$-module and there is an $\mathcal{FP}_{n+1}(R)$-injective
relative cogenerator in $\mathcal{FP}_{n+1}(R).$
\item[$\mathrm{(b)}$] $T$ is a big $n$-$\mathcal{F}\mathcal{P}_{n+1}(R)$-tilting object.
\item[$\mathrm{(c)}$] $T$ is a small $n$-$\mathcal{F}\mathcal{P}_{n+1}(R)$-tilting object.
\end{itemize}

Then $\mathrm{(b)}\Rightarrow\mathrm{(a)}$ and $\mathrm{(b)}\Leftrightarrow\mathrm{(c)}$ hold true. 
Furthermore,  $\mathrm{(a)}\Rightarrow\mathrm{(b)}$ holds true if $R$ is commutative.
\end{thm}
\begin{proof}
Note first that,  by \cite[Thm. 2.4 (3)]{bravo2017finiteness},  $R$
is $k$-coherent $\forall k\geq n$. In particular,  by Lemmas \ref{lem: propiedades de cerradura FPn}
and \ref{lem: n-coherente,  FPn cerrado por nucleos de epis},  $\mathcal{FP}_{n+1}(R)$
is thick in $\Modx[R]$.
\

(a) $\Rightarrow$ (b) Let $R$ be commutative. The conditions (T1),  (T2), 
(T4) and (T5) are proved straightforward by using \cite[Lem. 3.1.6]{Approximations}
and Lemma \ref{lem: anillo conmutativo coherente,  entonces precubiertas} (b).
The condition (T3) follows from (MT3) and Lemma \ref{lem: lemita miyashita}.
\

(b) $\Rightarrow$ (a) Let $T$ be an $n$-$\mathcal{F}\mathcal{P}_{n+1}(R)$-tilting
object. By (T4),  we know there is an $\mathcal{FP}_{n+1}(R)$-injective
relative cogenerator in $\mathcal{FP}_{n+1}(R)$. 

Let us prove that $T$ is Miyashita tilting. Since $T\in\mathcal{F}\mathcal{P}_{n+1}(R)$
and $R$ is $\left(n+1\right)$-coherent,   it can be shown that (MT1) is satisfied by using (T1) and the Shifting
Lemma. Moreover,  (MT2) follows
from (T2). Finally,  since $\mathcal{F}\mathcal{P}_{n+1}(R)$ is thick
in $\Mod(R)$,  by Theorem \ref{thm: el par n-X-tilting} (c),  there is an exact sequence 
$\suc[{}{}_{R}R][M_{0}][X_{0}]$
where $M_{0}\in T^{\bot}\cap\mathcal{F}\mathcal{P}_{n+1}(R)$ and $X_{0}\in{}{}^{\bot}(T^{\bot})\cap\mathcal{F}\mathcal{P}_{n+1}(R).$
Then,  using $_{R}R\in{}{}^{\bot}(T^{\bot})$ and Proposition \ref{prop: (a)} (d),  we have 
$M_{0}\in{}{}^{\bot}(T^{\bot})\cap T^{\bot}\cap\mathcal{F}\mathcal{P}_{n+1}(R)=\Addx[T]\cap\mathcal{F}\mathcal{P}_{n+1}(R)=\addx[T].$
By repeating the same arguments (recursively), 
we can build a long exact sequence 
$\suc[R][M_{0}\rightarrow\cdots\rightarrow M_{k}][X_{k}]\, \forall k\geq1$
with $M_{i}\in\addx[T]$ and $X_{i}\in{}{}^{\bot}(T^{\bot})\cap\mathcal{F}\mathcal{P}_{n+1}(R)\;\forall i\in[1, k]$.
Now,  since $\pdr[\mathcal{F}\mathcal{P}_{n+1}(R)][T]\leq n$,  we have $T^{\bot}\cap\mathcal{F}\mathcal{P}_{n+1}(R)$
is closed by $n$-quotients in $\mathcal{F}\mathcal{P}_{n+1}(R)$ by
\cite[Prop. 2.6]{parte1}.
Thus,  we have 
$X_{n}\in T^{\bot}\cap{}{}^{\bot}(T^{\bot})\cap\mathcal{F}\mathcal{P}_{n+1}(R)=\addx[T]$
and therefore $T$ satisfies (MT3). 
\

(b) $\Leftrightarrow$ (c) It follows from  Corollary \ref{cor: coro2 p80} since 
$\mathcal{FP}_{n+1}(R)\subseteq\modd.$\end{proof}

\subsection{\label{sub: Tilting-en-categor=0000EDas} Tilting in exact categories}

It is a known fact that a small exact category can be embedded into
an abelian category. In this section,  we will use this fact to introduce
a tilting theory on small exact categories. Furthermore,  we will see
that the tilting objects obtained by this procedure coincide with
the tilting objects defined by Bin Zhu y Xiao Zhuang in \cite{zhu2019tilting}.
\

Let $\mathcal{A}$ be an additive
category. A\textbf{ kernel-cokernel pair} $(i, p)$ in $\mathcal{A}$
is a sequence of morphisms $A'\stackrel{i}{\rightarrow}A\stackrel{p}{\rightarrow}A''$
in $\A$ such that $i$ is the kernel of $p$ and $p$ is the cokernel of $i$. 
Let $\mathcal{E}$ be a fixed class of kernel-cokernel pairs in $\mathcal{A}$.
A morphism $i$ ($p$,  respectively) is called \textbf{admissible
mono} (\textbf{admissible epi, } respectively) if there is a pair $(i, p)\in\mathcal{E}$. An \textbf{exact category} is a pair $\p[\mathcal{A}][\mathcal{E}]$, 
where $\mathcal{A}$ is an additive category and $\mathcal{E}$ is
a class of kernel-cokernel pairs satisfying certain axioms, see \cite[Def. 2.1]{buhler2010exact} for more details. 
\

Given an exact category $\p[\mathcal{A}][\mathcal{E}]$,  an element $(i, p)\in\mathcal{E}$ is called {\bf short exact sequence} and it is also denoted as $0\to X\xrightarrow{i} Z\xrightarrow{p} Y\to 0.$  Moreover, 
for every $X, Y\in\mathcal{A}$,  we denote by $\mathcal{E}(X, Y)$ the class of all the short exact sequences of the form $0\to Y\rightarrow Z\rightarrow X\to 0.$ Let 
$\p[\mathcal{A}][\mathcal{E}]$
and $\p[\mathcal{A}'][\mathcal{E}']$ be exact categories and $F: \mathcal{A}\rightarrow\mathcal{A}'$
be an additive functor. Following 
\cite[Def. 5.1]{buhler2010exact}, we recall that $F$ is \textbf{exact} if $F(\mathcal{E})\subseteq\mathcal{E}'$.
We say that $F$ \textbf{reflects exactness} in case $(F\alpha, F\beta)\in\mathcal{E}'$
implies $(\alpha, \beta)\in\mathcal{E}$.

 Let $\mathcal{A}$ be an additive
category and $X\xrightarrow{f} Y\xrightarrow{g} X$ in $\mathcal{A}$ such that $fg=1_Y.$ In this case, we say that $f$ is a \textbf{split-epi} and $g$ is a \textbf{split-mono}. Following \cite[Def. 6.1,  Rk. 6.2]{buhler2010exact}, we recall that that 
$\mathcal{A}$
is \textbf{idempotent complete} if any idempotent morphism $p: A\rightarrow A$
in $\mathcal{A}$ admits a kernel. If every split-epi
admits a kernel,  we say that $\mathcal{A}$ is \textbf{weakly idempotent
complete}.
\

Let $(\mathcal{A}, \mathcal{E})$
be an exact category. Following \cite[Def. 11.1]{buhler2010exact}, we recall that  an object $P\in\mathcal{A}$ is \textbf{$\mathcal{E}$-projective} if $\Homx[\mathcal{A}][P][-]: \mathcal{A}\rightarrow\mbox{Ab}$ is
exact. We denote
by $\mbox{Proj}_{\mathcal{E}}(\mathcal{A})$ the class of all the $\mathcal{E}$-projective
objects. The \textbf{$\mathcal{E}$-injective} objects are defined dually,  and the class of all the $\mathcal{E}$-injective objects is denoted by $\mbox{Inj}_{\mathcal{E}}(\mathcal{A}).$ We say that $(\mathcal{A}, \mathcal{E})$ 
 \textbf{has enough $\mathcal{E}$-projectives} \cite[Def. 11.9]{buhler2010exact}
 if every $A\in\mathcal{A}$ admits an admissible epi $P\rightarrow A$
with $P\in\operatorname{Proj}_{\mathcal{E}}(\mathcal{A})$. Dually, 
$\p[\mathcal{A}][\mathcal{E}]$ \textbf{has enough $\mathcal{E}$-injectives}
if every $A\in\mathcal{A}$ admits an admissible mono $A\rightarrow I$
with $I\in\operatorname{Inj}_{\mathcal{E}}(\mathcal{A})$. 

Consider a category $\mathcal{C}$ and $\mathcal{X}\subseteq\mathcal{C}$.
Let $[\mathcal{X}]$ denotes the class of all the objects $Z\in\mathcal{C}$
such that $Z\cong X$ with $X\in\mathcal{X}$. For a functor $F: \A\to \B, $ the class $[F(\A)]$ is also known as the essential image of $F$ in $\B.$

\begin{rem}\label{rem: observacion p152} Let $\mathcal{X}$ and $\mathcal{Y}$ be 
classes objects in a category $\mathcal{C}$. Note that $[\mathcal{X}\cap\mathcal{Y}]\subseteq[\mathcal{X}]\cap[\mathcal{Y}]$.
Furthermore,  $[\mathcal{X}\cap\mathcal{Y}]=\mathcal{X}\cap[\mathcal{Y}]$ if $\mathcal{X}=[\mathcal{X}].$
\end{rem}

\begin{thm}\cite[Thm. A.1]{buhler2010exact}\label{thm: exacta se sumerge en abeliana}
For a small exact category $(\mathcal{A}, \mathcal{E}), $ 
the following statements hold true.
\begin{itemize}
\item[$\mathrm{(a)}$] There is an abelian category $\mathcal{B}$ and a fully 
faithful exact functor $i: \mathcal{A}\rightarrow\mathcal{B}$ that reflects exactness.
Moreover,  $[i(\mathcal{A})]$ is closed under extensions in $\mathcal{B}$.
\item[$\mathrm{(b)}$] The category $\mathcal{B}$ may canonically be chosen to be the category
$\operatorname{Lex}(\mathcal{A}, \mathcal{E})$ (of all the contravariant additive left exact functors $\mathcal{A}\rightarrow\operatorname{Ab}$) and $i$ to be the
Yoneda embedding $i(A): =\Homx[\mathcal{A}][-][A]$.
\item[$\mathrm{(c)}$] Assume that $\mathcal{A}$ is weakly idempotent complete. If $f$
is a morphism in $\mathcal{A}$ such that $i(f)$ is epic (monic) in $\mathcal{B}$, 
then $f$ is an admissible epi (mono).
\end{itemize}
\end{thm}

\begin{defn} Let $(\mathcal{A}, \mathcal{E})$ be an exact category and $\mathcal{C}$
be an abelian category. If there is a fully faithful additive exact functor
$i: \mathcal{A}\rightarrow\mathcal{C}$ that reflects exactness and  $[i(\mathcal{A})]$ is closed under extensions in $\C,$  we
say that $(\mathcal{A}, \mathcal{E})$ \textbf{is embedded} in $\mathcal{C}$, 
and that $i:\A\to \C$ is \textbf{the embedding of $\mathcal{A}$ in }$\mathcal{C}$. 
\end{defn}

Observe that,  having an exact category $(\mathcal{A}, \mathcal{E})$ such 
that $\A$ is a full subcategory of an abelian category $\mathcal{C}$, 
is not the same as having the exact category $(\mathcal{A}, \mathcal{E})$
embedded in $\mathcal{C}, $ via the inclusion $\mathcal{A}\subseteq\mathcal{C}$.
For example,  if $R$ is a ring and $\mathcal{E}$ is the class of all the 
splitting exact sequences in $\Modx[R]$,  then $(\Modx[R], \mathcal{E})$
is an exact category. Notice that $(\Modx[R], \mathcal{E})$ is not embedded
in $\Modx[R]$ unless $R$ be a semisimple ring. We will see a non-trivial
example of this in the section related with Mohamed's contexts in $\modu\,(\Lambda).$
\

In what follows,  we introduce a tilting theory in exact categories.
We can find precedents of this in different contexts.
As examples,  we can cite Maurice Auslander and {\O}yvind Solberg's
relative tilting theory \cite{auslander1993relative2},  or the generalization
of such theory developed by Soud Khalifa Mohamed in \cite{mohamed2009relative}.
In this section,  we will approach to the tilting theory recently developed
by Bin Zhu and  Xiao Zhuang in \cite{zhu2019tilting}.
\

Let $\p[\mathcal{A}][\mathcal{E}]$ be an exact category,  $n\in\mathbb{N}$
and $\mathcal{X}\subseteq\mathcal{A}$. An \textbf{$\mathcal{X}$-resolution}
(of length $n$) in $\p[\mathcal{A}][\mathcal{E}]$ of $A\in\mathcal{A}$
 is a sequence $X_{n}\overset{d_{n}}{\rightarrow}X_{n-1}\rightarrow\cdots\rightarrow X_{1}\overset{d_{1}}{\rightarrow}X_{0}\overset{d_{0}}{\rightarrow}A$ of morphisms  in $\mathcal{A}$ 
such that there is a family $\{ K_{i+1}\overset{g_{i}}{\rightarrow}X_{i}\overset{f_{i}}{\rightarrow}K_{i}\} _{i=0}^{n}$
in $\mathcal{E}$ with $K_{n+1} =0,$ $K_n=X_n, $ $K_{0} =A,$ $g_{n-1} =d_n, $ $f_n =1_{X_n}, $ $f_{0} =d_{0}$
and $d_{i}=g_{i-1}f_{i}$ $\forall i\in[1, n-1]$. We denote by $\mathcal{X}_{n, \mathcal{E}}^{\wedge}$
the class of all the objects $A\in\mathcal{A}$ admitting an $\mathcal{X}$-resolution
in $\p[\mathcal{A}][\mathcal{E}]$ of length $\leq n$. We define the class $\mathcal{X}_{\mathcal{E}}^{\wedge}: =\bigcup_{n=0}^{\infty}\mathcal{X}_{n, \mathcal{E}}^{\wedge}$
and,  for any $A\in\mathcal{X}_{\mathcal{E}}^{\wedge}$,  the $\mathcal{X}$-resolution
dimension of $A$ is $\resdimx{\mathcal{X}, \mathcal{E}}A: =\min\left\{ n\in\mathbb{N}\, |\: A\in\mathcal{X}_{n, \mathcal{E}}^{\wedge}\right\} $.
The notion of  $\mathcal{X}$-coresolution,  the classes $\mathcal{X}_{\mathcal{E}, n}^{\vee}$ and 
$\mathcal{X}_{\mathcal{E}}^{\vee}$,  and the $\X$-coresolution dimension 
$\coresdimr{\mathcal{X}, \mathcal{E}}A{\, }$ of $A$ are defined dually.

Given $A, B\in\mathcal{A}$, $\Extx[1][\mathcal{A}][B][A]$
or $\Extx[1][(\mathcal{A}, \mathcal{E})][B][A]$ denote the class of extensions of $A$ by $B$  whose elements are equivalence classes of $\mathcal{E}(B, A).$ As in the case of abelian categories, we have that 
$\Extx[1][\mathcal{A}][B][A]$ 
is an abelian group with the Baer's sum,  where $0$ is the equivalence
class of exact sequence  
$0\to A\stackrel{\left(\begin{smallmatrix}1\\
0
\end{smallmatrix}\right)}{\rightarrow}A\oplus B\stackrel{\left(\begin{smallmatrix}0 & 1\end{smallmatrix}\right)}{\rightarrow}B\to 0\mbox{.}$
\

Let $\p[\mathcal{A}][\mathcal{E}]$  be with enough
$\mathcal{E}$-projectives and $\mathcal{E}$-injectives. Then,  any $A\in\mathcal{A}$ admits short exact sequences 
$0\to A\rightarrow I\rightarrow A^{1}\to 0$  and 
$0\to A_{1}\rightarrow P\rightarrow A\to 0$
in $\mathcal{E}, $ with $I\in\Injx[\mathcal{A}][\mathcal{E}]$ and $P\in\Projx[\mathcal{A}][\mathcal{E}]$.
In this case,  $A^{1}$ is a \textbf{first cosyzygy} of $A$
and $A_{1}$ is a \textbf{first syzygy} of $A$. Define an
\textbf{$n$-th cosyzygy} (\textbf{$n$-th syzygy},  respectively)
by recursion as the cosyzygy (syzygy) of the $(n-1)$-th cosyzygy
($(n-1)$-th syzygy). The class of all the $n$-th cosyzygies of $A$ is denoted
by $\Sigma^{n}(A)$,  and the class of all the $n$-th syzygies of $A$ is denoted by 
$\Omega_{n}(A)$. In \cite[Lem.  5.1]{liu2019hearts} it is shown that, for $k\geq2,$  $
\Extx[1][\mathcal{A}][X][Y^{k-1}]\cong\Extx[1][\mathcal{A}][X_{k-1}][Y]\quad\forall\,  X_{k-1}\in\Omega^{k-1}(X), \, \forall\,  Y^{k-1}\in\Sigma^{k-1}(Y).$ 
Such group is called the $k$-th Ext of $X$ and $Y, $ and it is denoted by 
  $\Extx[k][\mathcal{A}][X][Y]$ or by $\Extx[k][(\mathcal{A}, \mathcal{E})][X][Y]$. For $\mathcal{Z}\subseteq\mathcal{A}, $ we consider the right orthogonal class 
$\mathcal{Z}^{\bot_{\mathcal{E}}}:  =\left\{ A\in\mathcal{A}\;|\;\Extx[i][\mathcal{A}][Z][A]=0\quad\forall Z\in\mathcal{Z}, \, \forall i>0\right\},$ and  the left orthogonal class ${}^{\bot_{\mathcal{E}}}\mathcal{Z}$ is defined dually. For $A\in\mathcal{A}$ and $\mathcal{X}\subseteq\mathcal{A}, $ we  consider its $\X$-projective dimension 
$
\pdr[\mathcal{E}, \mathcal{X}][A]: =\min\left\{ n\in\mathbb{N}\;|\;\Extx[i][\mathcal{A}][A][-]|_{\mathcal{X}}=0\quad\forall i>n\right\} \mbox{.}
$

Given a class $\mathcal{T}\subseteq\mathcal{A}$,  we define
$
\pdr[\mathcal{E}, \mathcal{X}][\mathcal{T}]: =\sup\left\{ \pdr[\mathcal{E}, \mathcal{X}][T]\;|\;T\in\mathcal{T}\right\} \mbox{.}
$
In case  $\mathcal{X}=\mathcal{A}$,  we set $\pdr[\mathcal{E}][A]: =\pdr[\mathcal{E}, \mathcal{A}][A]$
and $\pdr[\mathcal{E}][\mathcal{T}]: =\pdr[\mathcal{E}, \mathcal{A}][\mathcal{T}]$.
Notice that, by  \cite[Lem. 3]{zhu2019tilting}, we have that, for any  $X\in\mathcal{A},$ $\pdr[\mathcal{E}][X]\leq n$ $\Leftrightarrow$
  $X\in(\Projx[\mathcal{A}][\mathcal{E}]){}_{\mathcal{E},n}^{\wedge}$. 

\begin{lem}\label{lem:  lema notas p120} Let $\p[\mathcal{A}][\mathcal{E}]$
be an exact category with enough $\mathcal{E}$-projectives and $\mathcal{E}$-injectives.
Then,  for $\mathcal{X}, \mathcal{Y}\subseteq\mathcal{A}$,  the following
statements hold true. 
\begin{itemize}
\item[$\mathrm{(a)}$] $\pdr[\mathcal{E}, \mathcal{Y}][\mathcal{X}_{\mathcal{E}}^{\vee}]=\pdr[\mathcal{E}, \mathcal{Y}][\mathcal{X}].$
\item[$\mathrm{(b)}$] $\mathcal{X}_{\mathcal{E}}^{\vee}\subseteq{}^{\bot_{\mathcal{E}}}\left(\mathcal{X}^{\bot_{\mathcal{E}}}\right)$.
\end{itemize}
\end{lem}

\begin{proof} We let it to the reader.
\end{proof}

Let $\p[\mathcal{A}][\mathcal{E}]$ be an exact category,  with enough
$\mathcal{E}$-projectives and $\mathcal{E}$-injectives,  and let $\omega$, $\mathcal{X}$
$\subseteq\mathcal{A}$.
We say that $\omega$ is a \textbf{relative $\mathcal{E}$-cogenerator} in $\mathcal{X}$ if $\omega\subseteq\mathcal{X}$ and every $X\in\mathcal{X}$
admits a short exact sequence $0\to X\rightarrow W\rightarrow X'\to 0$ in $\mathcal{E}$ such that $W\in\omega$ and $X'\in\mathcal{X}$.  Moreover,  $\omega$ is \textbf{$\mathcal{X}$-injective}
if $\idr[\mathcal{E}, \mathcal{X}][\omega]=0$. The
\textbf{relative $\mathcal{E}$-generators} in $\mathcal{X}$ and the \textbf{$\mathcal{X}$-projectives} are defined dually.

\begin{defn}
Let $\p[\mathcal{A}][\mathcal{E}]$ be an exact category with enough
$\mathcal{E}$-projectives and $\mathcal{E}$-injectives. A class
$\mathcal{T}\subseteq\mathcal{A}$ is called \textbf{small $n$-tilting} in 
$\p[\mathcal{A}][\mathcal{E}]$ if the following conditions hold true.
\begin{description}
\item [{(TEC0)}] $\mathcal{T}=\addx[\mathcal{T}].$
\item [{(TEC1)}] $\pdr[\mathcal{E}][\mathcal{T}]\leq n.$
\item [{(TEC2)}] $\mathcal{T}\subseteq\mathcal{T}^{\bot_{\mathcal{E}}}.$
\item [{(TEC3)}] There is a class $\omega\subseteq\mathcal{T}_{\mathcal{E}}^{\vee}$ which is a relative $\mathcal{E}$-generator in $\mathcal{A}.$
\item [{(TEC4)}] $\mathcal{T}$ is precovering in $\mathcal{T}^{\bot_{\mathcal{E}}}.$
\end{description}
An object $T\in\mathcal{A}$ is \textbf{small $n$-tilting
in $(\mathcal{A}, \mathcal{E})$} if $\addx[T]$ is a small $n$-tilting class
in $\p[\mathcal{A}][\mathcal{E}]$. 
\end{defn}

\begin{defn}\cite[Def. 7]{zhu2019tilting} 
Let $\p[\mathcal{A}][\mathcal{E}]$
be an exact category with enough $\mathcal{E}$-projectives and $\mathcal{E}$-injectives.
A class $\mathcal{T}\subseteq\mathcal{A}$ is called \textbf{Zhu-Zhuang $n$-tilting
} if the following conditions hold true.
\begin{description}
\item [{(ZZT0)}] $\mathcal{T}=\addx[\mathcal{T}].$
\item [{(ZZT1)}] $\pdr[\mathcal{E}][\mathcal{T}]\leq n.$
\item [{(ZZT2)}] $\mathcal{T}$ is a relative $\mathcal{E}$-generator
in $\mathcal{T}^{\bot_{\mathcal{E}}}$.
\end{description}
An object $T\in\mathcal{A}$ is called \textbf{Zhu-Zhuang $n$-tilting}
if $\addx[T]$ is a Zhu-Zhuang $n$-tilting class.
\end{defn}

 Let $\p[\mathcal{A}][\mathcal{E}]$
be an exact category with enough $\mathcal{E}$-projectives and 
$\mathcal{E}$-injectives,  $\mathcal{T}\subseteq\mathcal{A}$,  and let $n\geq0$. Following \cite[Sect. 4]{zhu2019tilting},  we denote by $\mbox{Pres}_{(\mathcal{A}, \mathcal{E})}^{n}(\mathcal{T})$ (or $\mbox{Pres}_{\mathcal{E}}^{n}(\mathcal{T})$)
the class of all the objects $X\in\A$ admitting a family $\left\{0\to X_{i+1}\rightarrow T_{i}\rightarrow X_{i}\to 0\right\} _{i=1}^{n}$ of short exact sequences in $\mathcal{E}$, 
where $T_{i}\in\mathcal{T}$ $\forall i\in[1, n]$ and $X_{1}=X$.

\begin{thm}\cite[Theorem 1]{zhu2019tilting}\label{thm: teo exactas} Let $\p[\mathcal{A}][\mathcal{E}]$
be an exact category with enough $\mathcal{E}$-projectives and $\mathcal{E}$-injectives, and   $\mathcal{T}=\addx[\mathcal{T}]\subseteq\mathcal{A}$
such that every object in $\mathcal{T}^{\bot_{\mathcal{E}}}$ admits
a $\mathcal{T}$-precover. Then,  $\mathcal{T}$ is a Zhu-Zhuang $n$-tilting
class if and only if $\operatorname{Pres}_{\mathcal{E}}^{m}(\mathcal{T})=\mathcal{T}^{\bot_{\mathcal{E}}}$, 
where $m: =\max\{1, n\}$.
\end{thm}

\begin{prop}\cite[Rk. 4(1)]{zhu2019tilting}\label{prop: zhu} Let $\p[\mathcal{A}][\mathcal{E}]$
be an exact category with enough $\mathcal{E}$-projectives and $\mathcal{E}$-injectives, 
and let $\mathcal{T}\subseteq\mathcal{A}$ be a Zhu-Zhuang $n$-tilting
class. Then,  the following statements hold true.
\begin{itemize}
\item[$\mathrm{(a)}$] $\mathcal{T}$ is a $\mathcal{T}^{\bot_{\mathcal{E}}}$-projective
relative $\mathcal{E}$-generator in $\mathcal{T}^{\bot_{\mathcal{E}}}.$
\item[$\mathrm{(b)}$] $\Projx[\mathcal{A}][\mathcal{E}]\subseteq\mathcal{T}_{n, \mathcal{E}}^{\vee}.$
\item[$\mathrm{(c)}$] $\left(\mathcal{T}^{\bot_{\mathcal{E}}}\right)_{\mathcal{E}}^{\vee}=\mathcal{A}$.
\end{itemize}
\end{prop}

\begin{lem}\label{lem:  lema1 notas} For a idempotent complete exact category
$\p[\mathcal{A}][\mathcal{E}]$ with enough $\mathcal{E}$-projectives
and $\mathcal{E}$-injectives,  $\mathcal{T}\subseteq\mathcal{A}$
and $\omega\subseteq\mathcal{T}_{\mathcal{E}}^{\vee}$ a relative
$\mathcal{E}$-generator in $\mathcal{A}$,  the following conditions
are satisfied.
\begin{itemize}
\item[$\mathrm{(a)}$] $\mathcal{T}^{\bot_{\mathcal{E}}}\subseteq\Pres[\mathcal{T}][\mathcal{E}][1]$.
\item[$\mathrm{(b)}$] If $\mathcal{T}\subseteq\mathcal{T}^{\bot_{\mathcal{E}}}$ and $\mathcal{T}$
is precovering in $\mathcal{T}^{\bot_{\mathcal{E}}}$,  then $\mathcal{T}$
is a relative $\mathcal{E}$-generator in $\mathcal{T}^{\bot_{\mathcal{E}}}$.
\end{itemize}
\end{lem}

\begin{proof} (a) Let $A\in\mathcal{T}^{\bot_{\mathcal{E}}}$. Since $\omega$ is a
relative $\mathcal{E}$-generator in $\mathcal{A}$,  there is a short
exact sequence $\eta: \: 0\to K\rightarrow W\overset{a}{\rightarrow}A\to 0$ in
$\mathcal{E}$,  with $W\in\omega$. Since $\omega\subseteq\mathcal{T}_{\mathcal{E}}^{\vee}$, 
there is a short exact sequence $0\to W\overset{b}{\rightarrow}B\rightarrow C\to 0$
in $\mathcal{E}$,  with $B\in\mathcal{T}$ and $C\in\mathcal{T}_{\mathcal{E}}^{\vee}\subseteq{}^{\bot_{\mathcal{E}}}(\mathcal{T}^{\bot_{\mathcal{E}}})$
(see Lemma \ref{lem:  lema notas p120}). \\
\fbox{\begin{minipage}[t]{0.3\columnwidth}%
\[
\begin{tikzpicture}[-, >=to, shorten >=1pt, auto, node distance=2cm, main node/.style=, x=1cm, y=1cm]

\node (1) at (0, 0) {$W$};
\node (2) at (1, 0) {$B$};
\node (3) at (2, 0) {$C$};

\node (01) at (0, 1) {$K$};
\node (02) at (1, 1) {$K$};

\node (1') at (0, -1) {$A$};
\node (2') at (1, -1) {$B'$};
\node (3') at (2, -1) {$C$};

\draw[->,  thin]  (1)  to  node  {$b$} (2);
\draw[->,  thin]  (2)  to  node  {$$} (3);

\draw[->,  thin]  (1')  to [above] node  {$t$} (2');
\draw[->,  thin]  (2')  to [below] node  {$$} (3');

\draw[->,  thin]  (1)  to  node  {$a$} (1');
\draw[->,  thin]  (2)  to  node  {$x$} (2');
\draw[-,  double]  (3)  to  node  {$$} (3');

\draw[-,  double]  (01)  to  node  {$$} (02);
\draw[->,  thin]  (01)  to  node  {$$} (1);
\draw[->,  thin]  (02)  to  node  {$$} (2);

\end{tikzpicture}
\]%
\end{minipage}}\hfill{}%
\begin{minipage}[t]{0.6\columnwidth}%
By the push-out diagram of $a$ and $b$,   the
 exact sequence $0\to A\rightarrow B'\rightarrow C\to 0$ in $\mathcal{E}$ splits since
$A\in\mathcal{T}^{\bot_{\mathcal{E}}}$ and $C\in{}^{\bot_{\mathcal{E}}}(\mathcal{T}^{\bot_{\mathcal{E}}})$.
Hence,  there is a morphism $y: B'\rightarrow A$ such that $yt=1_{A}$.
Now,  since $x, y$ are admissible epis,  it follows from \cite[Prop. B.1(ii)]{buhler2010exact}
that $yx: B\rightarrow A$ is an admissible epi with $B\in\mathcal{T}$. %
\end{minipage}
\vspace{0.2cm}

(b) Let $\mathcal{T}\subseteq\mathcal{T}^{\bot_{\mathcal{E}}}$ be precovering
in $\mathcal{T}^{\bot_{\mathcal{E}}}$. Consider $X\in\mathcal{T}^{\bot_{\mathcal{E}}}$
and a $\mathcal{T}$-precover $g: T'\rightarrow X$. By (a), 
there is an exact sequence $0\to R\rightarrow T\stackrel{\alpha}{\rightarrow}X\to 0$
in $\mathcal{E}$,  with $T\in\mathcal{T}$. Now,  since $g$ is a $\mathcal{T}$-precover, 
there is $g': T\rightarrow T'$ such that $\alpha=gg'$. Then,  by \cite[Prop. B.1(iii)]{buhler2010exact}, 
$g$ is an admissible epi. In particular,  there is an exact sequence
$\eta: \: 0\to K\rightarrow T'\stackrel{g}{\rightarrow}X\to 0$ in $\mathcal{E}.$ Finally,  since
$T'\in\mathcal{T}\cap\mathcal{T}^{\bot_{\mathcal{E}}}$ and $g$ is
a $\mathcal{T}$-precover,  it follows by \cite[Fact 1.18,  Prop. 5.2]{liu2019hearts} that $K\in\mathcal{T}^{\bot_{\mathcal{E}}}$
and thus $\mathcal{T}$ is a relative $\mathcal{E}$-generator in
$\mathcal{T}^{\bot_{\mathcal{E}}}$. 
\end{proof}

The following notion is inspired on \cite[Sect. 3]{auslander1993relative2}. 

\begin{defn}\label{def:  auslander-solberg tilting} Let $\p[][\mathcal{E}]$ be
an exact category with enough $\mathcal{E}$-projectives and $\mathcal{E}$-injectives. A class $\mathcal{T}\subseteq\mathcal{A}$ is {\bf Auslander-Solberg $n$-tilting} in 
$\p[][\mathcal{E}]$ if the following conditions hold true.
\begin{description}
\item [{(AST0)}] $\mathcal{T}=\addx[\mathcal{T}].$
\item [{(AST1)}] $\pdr[\mathcal{E}][\mathcal{T}]\leq n.$
\item [{(AST2)}] $\mathcal{T}\subseteq\mathcal{T}^{\bot_{\mathcal{E}}}.$
\item [{(AST3)}] $\operatorname{Proj}_{\mathcal{E}}(\mathcal{A})\subseteq\mathcal{T}_{\mathcal{E}}^{\vee}$. 
\end{description}
An object $T\in\mathcal{A}$ is Auslander-Solberg $n$-tilting
in $\p[\mathcal{A}][\mathcal{E}]$ if the class $\addx[T]$ is Auslander-Solberg
$n$-tilting in $\p[][\mathcal{E}]$.
\end{defn}

\begin{thm}\label{thm:  teo2 notas} Let $\p[][\mathcal{E}]$ be an idempotent
complete exact category with enough $\mathcal{E}$-projectives and
$\mathcal{E}$-injectives,  and let $\mathcal{T}\subseteq\mathcal{A}$.
Then,  the following statements are equivalent.
\begin{itemize}
\item[$\mathrm{(a)}$] $\mathcal{T}$ is Zhu-Zhuang $n$-tilting in $\p[\mathcal{A}][\mathcal{E}]$, 
with $\mathcal{T}$ precovering in $\mathcal{T}^{\bot_{\mathcal{E}}}.$
\item[$\mathrm{(b)}$] $\mathcal{T}$ is small $n$-tilting in $\p[\mathcal{A}][\mathcal{E}].$
\item[$\mathrm{(c)}$] $\mathcal{T}$ is Auslander-Solberg $n$-tilting in $(\mathcal{A}, \mathcal{E})$, 
with $\mathcal{T}$ precovering in $\mathcal{T}^{\bot_{\mathcal{E}}}$. 
\end{itemize}
Furthermore,  if one of the above conditions holds true,  then $\operatorname{Proj}_{\mathcal{E}}(\mathcal{A})\subseteq\mathcal{T}_{n, \mathcal{E}}^{\vee}$. 
\end{thm}
\begin{proof}
It follows from Proposition \ref{prop: zhu} and Lemma \ref{lem:  lema1 notas} (b).\end{proof}

As a consequence of Theorems \ref{thm:  teo2 notas} and \ref{thm: teo exactas}, we get the following result.

\begin{cor}\label{cor: coro p 131} For an
idempotent complete exact category $\p[\mathcal{A}][\mathcal{E}],$ with enough $\mathcal{E}$-projectives
and $\mathcal{E}$-injectives,  and  $\mathcal{T}\subseteq\mathcal{A},$ the 
 following statements are equivalent.
\begin{itemize}
\item[$\mathrm{(a)}$] $\mathcal{T}$ is small $n$-tilting in $\p[\mathcal{A}][\mathcal{E}].$
\item[$\mathrm{(b)}$] $\mathcal{T}$ is precovering in $\mathcal{T}^{\bot_{\mathcal{E}}}$
and $\Pres[\mathcal{T}][\mathcal{E}][m]=\mathcal{T}^{\bot_{\mathcal{E}}}$, 
with $m: =\max\left\{ 1, n\right\} $. 
\end{itemize}
\end{cor}

\begin{lem}\label{rem: pres gen }Let $\p[\mathcal{A}][\mathcal{E}]$ be an idempotent
complete exact category with enough $\mathcal{E}$-projectives and
$\mathcal{E}$-injectives,  embedded in an abelian category $\mathcal{C}$, 
via $i: \mathcal{A}\rightarrow\mathcal{C}$. Then,  for  $\mathcal{T}\subseteq\mathcal{A}$, 
we have $\addx[i(\mathcal{T})]=[i(\addx[\mathcal{T}])]\;\mbox{ and }\;\operatorname{gen}_{n}^{[i(\mathcal{A})]}(i(\mathcal{T}))\cap[i(\mathcal{A})]=[i\left(\operatorname{Pres}_{\mathcal{E}}^{n}(\addx[\mathcal{T}])\right)]\mbox{.}$\end{lem}

\begin{proof} It is straightforward and we left it to the reader
\end{proof}

\begin{lem}\label{lem: rel exact} Let $\p[\mathcal{A}][\mathcal{E}]$ be an exact
category with enough $\mathcal{E}$-projectives and $\mathcal{E}$-injectives
embedded in an abelian category $\mathcal{C}$,  via $i: \mathcal{A}\rightarrow\mathcal{C}$, 
and let $\mathcal{Z}\subseteq\mathcal{A}$. If $i(\Proj_{\mathcal{E}}(\mathcal{A}))\subseteq\Proj(\mathcal{C})$, 
then the following statements hold true.
\begin{itemize}
\item[$\mathrm{(a)}$] $i(\Proj_{\mathcal{E}}(\mathcal{A}))$ is an $i(\mathcal{A})$-projective
relative generator in $i(\mathcal{A}).$
\item[$\mathrm{(b)}$] $i(\Inj_{\mathcal{E}}(\mathcal{A}))$ is an $i(\mathcal{A})$-injective
relative cogenerator in $i(\mathcal{A}).$
\item[$\mathrm{(c)}$] $\Extx[k][\mathcal{\mathcal{C}}][i(A)][i(A')]\cong\Extx[k][\mathcal{\mathcal{\mathcal{A}}}][A][A']\, \forall A, A'\in\mathcal{A}, \, \forall k>0.$
\item[$\mathrm{(d)}$] $i(\mathcal{Z}^{\bot_{\mathcal{E}}})=i(\mathcal{Z})^{\bot}\cap i(\mathcal{A})$
and $i({}{}^{\bot_{\mathcal{E}}}\mathcal{Z})={}^{\bot}i(\mathcal{Z})\cap i(\mathcal{A})$.
\end{itemize}
\end{lem}

\begin{proof} The proof of (a) is straightforward,  and (b) and (d) follow from (c).
\

Let us show (c). Let $A, A'\in\A$ and $k>0.$ For $k=1,$ the result follows since $i$ is an embedding of $\A$ in $\C.$
\

Let $k\geq 2.$ Using (a),  we can construct the exact sequence in $\C$
\[
\suc[i(A{}_{k-1})][i(P_{k-1})\xrightarrow{f_{k-2}}\cdots\xrightarrow{f_{1}}i(P_{1})][i(A)], 
\]
where $i(P_{j})\in i(\operatorname{Proj}_{\mathcal{E}}(\mathcal{A}))\subseteq\Proj(\mathcal{C})\: \forall j\in[1, k-1]$, 
$A_{k-1}\in\mathcal{A}$ y $i(A{}_{j}): =\im[f_{j}]\in i(\mathcal{A})\;\forall j\in[1, k-2].$ Thus by the Shifting Lemma (and the case $k=1$),  we get 
\[
\Ext^k_\C(i(A), i(A'))\simeq\Ext^1_\C(i(A_{k-1}), i(A'))\simeq\Ext^1_\A(A_{k-1}, A')\simeq\Ext^k_\A(A, A').
\]
\end{proof}

\begin{lem}\label{lem: notas p17} Let $F: \mathcal{A}\rightarrow\mathcal{C}$
be a full and faithful functor,  and let $\mathcal{X}, \mathcal{Y}\subseteq\mathcal{A}$ 
 be such that $\left[\mathcal{X}\right]=\mathcal{X}$ and $\left[\mathcal{Y}\right]=\mathcal{Y}$.
If $[F(\mathcal{X})]=[F(\mathcal{Y})]$,  then $\mathcal{X}=\mathcal{Y}$. 
\end{lem}
\begin{proof} It is straightforward. 
\end{proof}

\begin{thm}\label{thm: tilting exacto} Let $\p[\mathcal{A}][\mathcal{E}]$ be
an idempotent complete exact category with enough $\mathcal{E}$-projectives
and $\mathcal{E}$-injectives embedded in an abelian category $\mathcal{C}$, 
via $i: \mathcal{A}\rightarrow\mathcal{C}$,  such that $i(\Proj_{\mathcal{E}}(\mathcal{A}))\subseteq\Proj(\mathcal{C}),$ and let  $\mathcal{T}\subseteq\mathcal{A}$. Then,  the following statements
are equivalent.
\begin{itemize}
\item[$\mathrm{(a)}$] The class $\mathcal{T}$ is small $n$-tilting in $\p[\mathcal{A}][\mathcal{E}]$.
\item[$\mathrm{(b)}$] The class $\left[i(\mathcal{T})\right]$ is small $n$-$[i(\mathcal{A})]$-tilting
in $\mathcal{C}$.
\end{itemize}
\end{thm}
\begin{proof} (a) $\Rightarrow$ (b) Let us show the conditions on $[i(\mathcal{T})]$ to be $n$-$[i(\mathcal{A})]$-tilting
in $\mathcal{C}$. Condition (T0) follows from Lemma \ref{rem: pres gen };
(T1) follows from Lemma \ref{lem: rel exact} (c) and (TEC1); (T2) follows
from (TEC2) and Lemma \ref{lem: rel exact} (d); (T3) follows straightforward
from (TEC3); (T4) follows from Lemma \ref{lem: rel exact} (b, d);
and finally,   (T5) follows from (TEC4) and Lemma \ref{lem: rel exact} (d).
\

(b) $\Rightarrow$ (a) By Theorem \ref{thm:  teo2 notas},  it is enough to show
that $\mathcal{T}$ is Zhu-Zhuang $n$-tilting in $\p[\mathcal{A}][\mathcal{E}]$
and that $\mathcal{T}$ is precovering in $\mathcal{T}^{\bot_{\mathcal{E}}}$. Let us show that $\mathcal{T}$ is precovering in $\mathcal{T}^{\bot_{\mathcal{E}}}$.
By (T5),  every $X\in\left[i(\mathcal{T})\right]\mathcal{^{\bot}}\cap\left[i(\mathcal{A})\right]$
admits an $\left[i(\mathcal{T})\right]$-precover. Then,  since $i$
is full and $\left[i(\mathcal{T})\right]\mathcal{^{\bot}}\cap\left[i(\mathcal{A})\right]=\left[i(\mathcal{T}^{\bot_{\mathcal{E}}})\right]$
by Lemma \ref{lem: rel exact}(d),  any object of $\mathcal{T}^{\bot_{\mathcal{E}}}$
admits a $\mathcal{T}$-precover.

Let $m: =\max\{1, n\}$. By Proposition \ref{prop: primera generalizacion-1}, 
Lemma \ref{rem: pres gen } and Lemma \ref{lem: rel exact}(d),  
\[
\left[i(\operatorname{Pres}_{\mathcal{E}}^{m}(\mathcal{T}))\right]=\operatorname{gen}_{m}^{\left[i(\mathcal{A})\right]}(i(\mathcal{T}))\cap\left[i(\mathcal{A})\right]=i(\mathcal{T}^{\bot})\cap\left[i(\mathcal{A})\right]=[i(\mathcal{T})^{\bot}\cap i(\mathcal{A})]=\left[i(\mathcal{T}^{\bot_{\mathcal{E}}})\right].
\]
Hence,  $\mathcal{T}$ is Zhu-Zhuang $n$-tilting in $(\mathcal{A}, \mathcal{E})$
by Lemma \ref{lem: notas p17} and Theorem \ref{thm: teo exactas}.\end{proof}

\begin{lem}\label{lem: 2.177'}Let $(\mathcal{A}, \mathcal{E})$ be a skeletally
small exact category with enough $\mathcal{E}$-projectives,  and let $\mathcal{P}: =\operatorname{Proj}_{\mathcal{E}}(\mathcal{A})$.
Then,  
$i: \mathcal{A}\rightarrow\Modx[\mathcal{P}^{op}]\mbox{,  }X\mapsto\Homx[\mathcal{A}][-][X]|_{\mathcal{P}}\mbox{, }$
is an additive,  faithful,  full and exact functor that reflects exactness
and such that $i(\mathcal{P})\subseteq\Projx[\operatorname{Mod}(\mathcal{P}^{op})]$.\end{lem}
\begin{proof}
 it follows from \cite[Prop. 2.1]{enomoto2017classifying}.\end{proof}

\begin{cor}
\label{cor: 2.177''} Let $\p[\mathcal{A}][\mathcal{E}]$ be an idempotent
complete,  skeletally small,  exact category with enough $\mathcal{E}$-projectives
and enough $\mathcal{E}$-injectives. Then,  by using the functor $i: \mathcal{A}\rightarrow\Modx[\mathcal{P}^{op}]$
given in Lemma \ref{lem: 2.177'},  the following conditions are equivalent
for a class $\mathcal{T}\subseteq\mathcal{A}.$
\begin{itemize}
\item[$\mathrm{(a)}$] $\mathcal{T}$ is small $n$-tilting in $\p[\mathcal{A}][\mathcal{E}].$
\item[$\mathrm{(b)}$] $\left[i(\mathcal{T})\right]$ is small $n$-$\left[i(\mathcal{A})\right]$-tilting
in $\Modx[\mathcal{P}^{op}]$.
\end{itemize}
\end{cor}

\begin{proof}
It follows from Lemma \ref{lem: 2.177'} and Theorem \ref{thm: tilting exacto}.
\end{proof}

\subsection{S. K. Mohamed's contexts in $\modu\,(\Lambda)$}

In \cite{mohamed2009relative},  Soud Khalifa Mohamed developed a relative
tilting theory inspired in the work of Maurice Auslander and {\O}yvind
Solberg in \cite{auslander1993relative2}. In this section,  we will
review the main aspects of his work and characterize his tilting objects
in terms of $n$-$\mathcal{X}$-tilting theory. We recall that,  in the context
of Artin algebras,  a class $\mathcal{C}\subseteq\modd[\Lambda]$ is
functorially finite if it is a precovering and preenveloping class in $\modd[\Lambda]$.

\begin{prop}\cite{auslander1993relative,  mohamed2009relative}  \label{prop: muhamed}
Let $\Lambda$ be an Artin algebra,  $\mathcal{C}\subseteq\modd[\Lambda]$
be functorially finite and closed under extensions,  and let $\mathcal{X}=\addx[\mathcal{X}]$
be precovering and a relative generator in $\mathcal{C}$. Consider the
class $\mathcal{E}_{\mathcal{X}}$ of all the short exact sequences 
$\suc[A][B][C][\, ][f]$
in $\modd[\Lambda]$ such that $\Homx[\Lambda][X][f]$ is surjective
$\forall X\in\mathcal{\mathcal{X}}$. Then,  the following statements
hold true.
\begin{itemize}
\item[$\mathrm{(a)}$] $\overline{\mathcal{E}_{\mathcal{X}}}: =\left\{ \overline{\eta}\, |\: \eta\in\mathcal{E}_{\mathcal{X}}\right\} $
is an additive subfunctor of $\Extx[1][\Lambda][-][-].$
\item[$\mathrm{(b)}$] For any exact sequence $\suc[A][B][C][\, ][f]$ in $\mathcal{E}_{\mathcal{X}}$
with $B, C\in\mathcal{C}$,  we have that $A\in\mathcal{C}.$
\item[$\mathrm{(c)}$] The pair $(\mathcal{C}, \mathcal{E}_{\mathcal{X}}^{\mathcal{C}})$ is an exact
category,  for $\mathcal{E}_{\mathcal{X}}^{\mathcal{C}}: =\left\{ \eta\in\mathcal{E}_{\mathcal{X}}(X, Y)\, |\: X, Y\in\mathcal{C}\right\}.$
\item[$\mathrm{(d)}$] If $\mathcal{I}_{\mathcal{C}}(\mathcal{E}_{\mathcal{X}}): =\left\{ C\in\mathcal{C}\, |\: \Homx[\Lambda][\eta][C]\mbox{ is exact }\forall\eta\in\mathcal{E}_{\mathcal{X}}\right\} $
is preenveloping in $\mathcal{C}$,  then $(\mathcal{C}, \mathcal{E}_{\mathcal{X}}^{\mathcal{C}})$
has enough $\mathcal{E}_{\mathcal{X}}^{\mathcal{C}}$-injectives.
\item[$\mathrm{(e)}$] The exact category $(\mathcal{C}, \mathcal{E}_{\mathcal{X}}^{\mathcal{C}})$ has enough
$\mathcal{E}_{\mathcal{X}}^{\mathcal{C}}$-projectives and $\mathcal{X}$ 
$=\Proj_{\mathcal{E}_{\mathcal{X}}^{\mathcal{C}}}(\C)$.
\end{itemize}
\end{prop}

\begin{defn} \cite[Sect. 4]{mohamed2009relative}\label{def: tilting mohamed}
Let $\Lambda$ be an Artin algebra. A {\bf Mohamed context} in $\modd[\Lambda]$ is a pair $(\mathcal{C}, \mathcal{X})$ of classes of objects in $\modd[\Lambda]$
satisfying the following conditions. 
\begin{description}
\item [{(MC1)}] $\mathcal{C}$ is functorially finite and closed under extensions.
\item [{(MC2)}] $\mathcal{X}=\addx[\mathcal{X}]$ is precovering and a relative
generator in $\mathcal{C}.$
\item [{(MC3)}] $\mathcal{I}_{\mathcal{C}}(\mathcal{E}_{\mathcal{X}})$
is preenveloping in $\mathcal{C}$. 
\end{description}
\end{defn}

\begin{rem}\label{rem: no sumerge} Let $\Lambda$ be an Artin algebra and let 
$\p[\mathcal{C}][\mathcal{X}]$ be a Mohamed's
context in $\modd[\Lambda]$. 
\begin{enumerate}
\item [(1)] By Proposition \ref{prop: muhamed},  $(\mathcal{C}, \mathcal{E}_{\mathcal{X}}^{\mathcal{C}})$
is a skeletally small exact category with enough $\mathcal{E}_{\mathcal{X}}^{\mathcal{C}}$-projectives
and $\mathcal{E}_{\mathcal{X}}^{\mathcal{C}}$-injectives,  where $\operatorname{Proj}_{\mathcal{E}_{\mathcal{X}}^{\mathcal{C}}}(\mathcal{C})$=$\mathcal{X}$.
\item [(2)]  For $\mathcal{P}: =\operatorname{Proj}_{\mathcal{E}_{\mathcal{X}}^{\mathcal{C}}}(\mathcal{C})=\mathcal{X}$ and 
 $i: \mathcal{C}\rightarrow\Modx[\mathcal{P}^{op}]$, 
$C\mapsto\Homx[\Lambda][-][C]|_{\mathcal{P}},$ the following
statements are equivalent: (a) $\mathcal{C}=\smdx[\mathcal{C}]$ in $\modd[\Lambda];$ (b) $\mathcal{C}$ is idempotent complete; and (c) $[i(\mathcal{C})]=\smdx[[i(\mathcal{C})]]$ in $\Modx[\mathcal{P}^{op}]$. 
\end{enumerate}
\end{rem}

Observe that, in general, for a Mohamed's context $\p[\mathcal{C}][\mathcal{X}]$ in $\modu\,(\Lambda),$ the inclusion functor $(\mathcal{C}, \mathcal{E}_{\mathcal{X}}^{\mathcal{C}})\rightarrow\modd[\Lambda], $ induced by the inclusion $\mathcal{C}\subseteq\modd[\Lambda], $ does not reflect exactness. Indeed,  consider $\Lambda: =\mathbb{R}[x]/\left\langle x^{4}\right\rangle $, 
$M={}_{\Lambda}\Lambda$,  $N: =x^{2}M$ and $K: =N/x^{2}N$. Take $\mathcal{C}: =\modd[\Lambda]$
and  $\mathcal{X}: =\addx[M\oplus K]$. Notice that $(\mathcal{C}, \mathcal{X})$ is a Mohamed context
in $\modd[\Lambda]$. Moreover, the exact sequence 
$\suc[N][M][K][\mbox{ }][\, ]$
does not belong in $\mathcal{E}_{\mathcal{X}}^{\mathcal{C}}$ since
it does not split. Therefore,  $(\mathcal{C}, \mathcal{X})$ is not
embedded in $\modd[\Lambda],$ via the natural inclusion $\mathcal{C}\subseteq\modd[\Lambda].$

\begin{thm}\label{prop: mohamed vs n-X-tilting}\label{thm: mohammed vs n-X-tilting}\label{thm: mohammed vs n-X-tilting-1}
Let $\Lambda$ be an Artin algebra and let $(\mathcal{C}, \mathcal{X})$
be a Mohamed's context in $\modd[\Lambda]$. Consider the exact category
$(\mathcal{C}, \mathcal{E}_{\mathcal{X}}^{\mathcal{C}})$,  $\mathcal{P}: =\operatorname{Proj}_{\mathcal{E}_{\mathcal{X}}^{\mathcal{C}}}(\mathcal{C})=\mathcal{X}$
and the embedding  $i: \mathcal{C}\rightarrow\Modx[\mathcal{P}^{op}]$, 
$C\mapsto\Homx[\Lambda][-][C]|_{\mathcal{P}}$. Then,  for a class $\mathcal{T}\subseteq\mathcal{C}$,  the following statements are equivalent.
\begin{itemize}
\item[$\mathrm{(a)}$] $\mathcal{T}$ is Auslander-Solberg $n$-tilting in $\p[\mathcal{C}][\mathcal{E}_{\mathcal{X}}^{\mathcal{C}}]$
and $\mathcal{T}$ is precovering in $\mathcal{T}^{\bot_{\mathcal{E}_{\mathcal{X}}^{\mathcal{C}}}}$.
\item[$\mathrm{(b)}$] $\left[i(\mathcal{T})\right]$ is small $n$-$\left[i(\mathcal{C})\right]$-tilting
in $\operatorname{Mod}(\mathcal{P}^{op})$.
\item[$\mathrm{(c)}$] $\mathcal{T}$ is Zhu-Zhuang $n$-tilting in $\p[\mathcal{C}][\mathcal{E}_{\mathcal{X}}^{\mathcal{C}}]$
and $\mathcal{T}$ is precovering in $\mathcal{T}^{\bot_{\mathcal{E}_{\mathcal{X}}^{\mathcal{C}}}}$.
\item[$\mathrm{(d)}$] $\mathcal{T}$ is small $n$-tilting in $\p[\mathcal{C}][\mathcal{E}_{\mathcal{X}}^{\mathcal{C}}]$.
\item[$\mathrm{(e)}$] $\mathcal{T}$ is precovering in $\mathcal{T}^{\bot_{\mathcal{E}_{\mathcal{X}}^{\mathcal{C}}}}$
and $\Pres[\mathcal{T}][\mathcal{E}_{\mathcal{X}}^{\mathcal{C}}][m]=\mathcal{T}^{\bot_{\mathcal{E}_{\mathcal{X}}^{\mathcal{C}}}}, $ where $m: =\max\{1, n\}.$
\end{itemize}
\end{thm}

\begin{proof}
It follows from Theorem \ref{thm:  teo2 notas},  Corollary \ref{cor: coro p 131},  \cite[Props. 2.1 and 2.8]{enomoto2017classifying}, 
Remark \ref{rem: no sumerge} (2) and Theorem \ref{thm: tilting exacto}.
\end{proof}

\begin{cor} Let $\Lambda$ be an Artin algebra,  $\p[\mathcal{C}][\mathcal{X}]$
be a Mohamed's context in $\modd[\Lambda]$ and let $T\in\modd[\Lambda].$ 
 Then,  for the  embedding  $i: \mathcal{C}\rightarrow\Modx[\mathcal{X}^{op}]$, 
$C\mapsto\Homx[\Lambda][-][C]|_{\mathcal{X}}, $ the following statements are equivalent.
\begin{itemize}
\item[$\mathrm{(a)}$] $T$ is Auslander-Solberg $n$-tilting in $\p[\mathcal{C}][\mathcal{E}_{\mathcal{X}}^{\mathcal{C}}]$.
\item[$\mathrm{(b)}$] $i(T)$ is small $n$-$[i(\mathcal{C})]$-tilting in $\Modx[\mathcal{X}^{op}]$. 
\item[$\mathrm{(c)}$] $T$ is Zhu-Zhuang $n$-tilting in $\p[\mathcal{C}][\mathcal{E}_{\mathcal{X}}^{\mathcal{C}}]$.
\item[$\mathrm{(d)}$] $T$ is small $n$-tilting in $\p[\mathcal{C}][\mathcal{E}_{\mathcal{X}}^{\mathcal{C}}]$. 
\item[$\mathrm{(e)}$] $\Pres[\operatorname{add}(T)][\mathcal{E}_{\mathcal{X}}^{\mathcal{C}}][m]=T^{\bot_{\mathcal{E}_{\mathcal{X}}^{\mathcal{C}}}}, $ where $m: =\max\left\{ 1, n\right\}.$
\end{itemize}
\end{cor}

\begin{proof}
It follows from Proposition \ref{prop: mohamed vs n-X-tilting} since $\add(T)$ is precovering.
\end{proof}

In what follows,  we will approach briefly to the precursor of S. K. Mohamed's
tilting theory. Namely,  we shall introduce the relative homological
algebra presented by Maurice Auslander and {\O}yvind Solberg in \cite{auslander1993relative,  auslander1993relative2}.
\

Let $\Lambda$ be an Artin algebra and $F$ be an additive subfunctor
of $\Extx[1][\Lambda][-][-]$. A short exact sequence $\eta: \: \suc$ in $\modd[\Lambda]$ is \textbf{$F$-exact}
if $\overline{\eta}\in F(K, N), $ and the class of all the short $F$-exact sequences is denoted 
by $\mathcal{E}_F.$
\

The class $\mathcal{P}(F)$ of the {\bf $F$-projective} modules 
consists of all the $P\in\modu(\Lambda)$ such that for every $F$-exact sequence 
$\eta: \suc, $ the sequence $\Hom_\Lambda(P, \eta)$ is exact. It is said that $F$ has enough projectives if any $M\in\modd[\Lambda]$
admits an $F$-exact sequence $\suc[M'][P][M]$ with $P\in\mathcal{P}(F).$ Define dually $F$-injective modules, the class $\mathcal{I}(F)$ and  the notion of saying that $F$ has enough injectives.
\

Each $\mathcal{X}\subseteq\modd[\Lambda]$ induces two subfunctors $F_{\mathcal{X}}$ and $F^{\mathcal{X}}$ of $\Extx[1][\Lambda][-][-].$ Indeed, for every $A, C\in\modd[\Lambda], $ $F_{\mathcal{X}}(C, A)$ is formed by all the extensions 
$\overline{\eta}\in\Extx[1][\Lambda][C][A]$, 
with $\eta: \: \suc[A][B][C][\, ][\, ]$,  such that $\Homx[\Lambda][-][B]|_{\mathcal{X}}\rightarrow\Homx[\Lambda][-][C]|_{\mathcal{X}}\rightarrow0$
is exact. The functor $F^{\mathcal{X}}$ is defined dually.

\begin{prop}\cite{auslander1993relative,  auslander1993relative2}  
\label{def:  auslader-solberg 2}\label{fact:  auslander-solberg} 
For an Artin algebra $\Lambda$ and a class $\mathcal{X}\subseteq\modd[\Lambda], $ the following statements hold true.
\begin{itemize}
\item[$\mathrm{(a)}$] $F_{\mathcal{X}}$
and $F^{\mathcal{X}}$ are additive subfunctors of $\Extx[1][\Lambda][-][-].$
\item[$\mathrm{(b)}$] Let $F$ be an additive
subfunctor of $\Extx[1][\Lambda][-][-]$. Then the class $\mathcal{E}_F$ is closed under pull-backs,  push-outs and finite coproducts.

\item[$\mathrm{(c)}$] The map
$F\mapsto\mathcal{P}(F)$ is a bijection between the class of all the additive subfunctors
of $\Extx[1][\Lambda][-][-]$ with enough projectives and the class of all the  precovering
classes $\mathcal{X}$  in $\modd[\Lambda]$ such that $\projx[\Lambda]\subseteq\mathcal{X}=\addx[\mathcal{X}]$.

\item[$\mathrm{(d)}$]  Let $F$ be an additive subfunctor of $\Extx[1][\Lambda][-][-]$ with
enough projectives. Then,  $(\modd[\Lambda], \mathcal{E}_{F})$ is an exact category
with $\operatorname{Proj}_{\mathcal{E}_{F}}(\modd[\Lambda])=\mathcal{P}(F)$
and $\operatorname{Inj}_{\mathcal{E}_{F}}(\modd[\Lambda])=\mathcal{I}(F)$. 
\item[$\mathrm{(e)}$] Let $\X\subseteq\modu(\Lambda)$ be such that $\proj(\Lambda)\subseteq\X=\add(\X).$ Then,  $F_\X$ has enough projectives and injectives if,  and only if,  $\X$ is functorially finite.
\end{itemize}
\end{prop}

As a consequence of \cite[Cor. 1.13]{auslander1993relative}, we get the following result that is useful to obtain examples of Mohamed's contexts.

\begin{prop}\label{prop:  prop3 notas} Let $\Lambda$ be an Artin algebra
and let $F$ be an additive subfunctor of $\Extx[1][\Lambda][-][-]$. Then, 
the following statements hold true.
\begin{itemize}
\item[$\mathrm{(a)}$] $F$ has enough projectives and injectives if,  and only if,  $\mathcal{P}(F)$
is functorially finite and $F=F_{\mathcal{P}(F)}$.
\item[$\mathrm{(b)}$] Let $F$ be with enough projectives and injectives. Then,   $(\modd[\Lambda], \mathcal{P}(F))$
is a Mohamed's context in $\modd[\Lambda]$ and $\mathcal{E}_{\mathcal{P}(F)}^{\modd[\Lambda]}=\mathcal{E}_{F}$. 
\end{itemize}
\end{prop}

\begin{cor}\label{thm: auslander-solberg} Let $\Lambda$ be an Artin algebra, 
$F$ an additive subfunctor of $\Extx[1][\Lambda][-][-]$ with enough
projectives and enough injectives,  $\mathcal{T}\subseteq\modd[\Lambda]$ and the  functor
$i: \modd[\Lambda]\rightarrow\Modx[\mathcal{P}^{op}]$,  $M\mapsto\Homx[\Lambda][-][M]|_{\mathcal{P}}$, 
where $\mathcal{P}: =\mathcal{P}(F)$. Then,  the pair $(\modd[\Lambda], \mathcal{E}_{F})$
is an exact category with enough $\mathcal{E}_{F}$-projectives and
$\mathcal{E}_{F}$-injectives. Moreover,  the following statements
are equivalent:
\begin{itemize}
\item[$\mathrm{(a)}$] $\mathcal{T}$ is Auslander-Solberg $n$-tilting in $(\modd[\Lambda], \mathcal{E}_{F})$
and precovering in $\mathcal{T}^{\bot_{\mathcal{E}_{F}}}.$
\item[$\mathrm{(b)}$] $\left[i(\mathcal{T})\right]$ is small $n$-$\left[i(\modd[\Lambda])\right]$-tilting
in $\Modx[\mathcal{P}^{op}].$
\item[$\mathrm{(c)}$] $\mathcal{T}$ is Zhu-Zhuang $n$-tilting in $(\modd[\Lambda], \mathcal{E}_{F})$
and precovering in $\mathcal{T}^{\bot_{\mathcal{E}_{F}}}.$
\item[$\mathrm{(d)}$] $\mathcal{T}$ is small $n$-tilting in $(\modd[\Lambda], \mathcal{E}_{F}).$
\item[$\mathrm{(e)}$] $\Pres[\mathcal{T}][\mathcal{E}_{F}][m]=\mathcal{T}^{\bot_{\mathcal{E}_{F}}}$
and $\mathcal{T}$ is precovering in $\mathcal{T}^{\bot_{\mathcal{E}_{F}}}, $ where $m: =\max\{1, n\}.$ 
\end{itemize}
\end{cor}

\begin{proof}
It follows from Proposition \ref{prop:  prop3 notas} (b) and Theorem \ref{thm: mohammed vs n-X-tilting-1}.
\end{proof}

\begin{cor}\label{cor: resultado de jiaqun wei} Let $\Lambda$ be an Artin algebra,  
$T\in\modd[\Lambda]$,  $n\in\mathbb{N}$,  $m: =\max\{1, n\}$ and let 
$F$ be an additive subfunctor of $\Extx[1][\Lambda][-][-]$ with enough
injectives and projectives. Then,  $T\in\modd[\Lambda]$ is Auslander-Solberg
$n$-tilting in $(\modd[\Lambda], \mathcal{E}_{F})$ if,  and only if, 
$T^{\bot_{\mathcal{E}_{F}}}=\operatorname{Pres}_{\mathcal{E}_{F}}^{m}(\addx[T])$.\end{cor}
\begin{proof}
It follows from Theorem \ref{thm: auslander-solberg} since $\add(T)$ is precovering.
\end{proof}

\begin{rem} Let $\Lambda$ be an Artin algebra.
\

(1) \cite[Thm. 3.10]{anoteonrelative} is a particular case
of Corollary \ref{cor: resultado de jiaqun wei}. Indeed,  if
$F$ is an additive subfunctor of $\Extx[1][\Lambda][-][-]$
with enough projectives and such that $\mathcal{P}(F)$ is of finite type,  then  $F$ has enough projectives and injectives
by \cite[Thm. 1.12]{auslander1993relative},  \cite[Cor. 1.13]{auslander1993relative}
and \cite[Prop. 4.2]{auslander1980preprojective}. 
\

(2)  There are examples where we can apply Corollary \ref{cor: resultado de jiaqun wei} but not \cite[Thm. 3.10]{anoteonrelative}. In order to see that,  we need to give an example of an additive subfunctor $F$ with enough projectives and
injectives and such that $\mathcal{P}(F)$ is not of finite type. Let $\Lambda$ be a quasi-hereditary algebra. Consider the class $\mathcal{F}(\triangle)$ of all the objects
in $\modd[\Lambda]$ filtered by the set of standard modules $\Delta.$ It is well known,  that $\mathcal{F}(\triangle)$ is resolving and functorially finite \cite[Thms. 1 and 3]{ringel1991category}. Therefore,  from Proposition \ref{fact:  auslander-solberg} (e),  we get that $F_{\mathcal{F}(\triangle)}$ has enough projectives
and injectives. Finally,  in \cite[Sect. 3.5]{erdmann2010auslander} we can
find examples of quasi-hereditary algebras where $\mathcal{P}(F_{\mathcal{F}(\triangle)})=\mathcal{F}(\triangle)$
is not of finite type. 
\end{rem}

\subsection{Tilting classes in functor categories}

\renewcommandx\modd[1][usedefault,  addprefix=\global,  1=R]{\operatorname{f.p.}\left(#1\right)}

In the early years of the last decade,  Roberto Mart\'inez Villa
and Mart\'in Ortiz Morales begun a series of research works with
the goal of extending tilting theory to arbitrary functor categories
\cite{martinez2011tilting, martinez2013tilting, martinez2014tilting}.
In the following lines,  we will see how such theory can be related with
 $n$-$\mathcal{X}$-tilting objects. We will start this description by following
the steps of Maurice Auslander in \cite{doi: 10.1080/00927877408548230}.

Let   $\mathcal{C}$  be an skeletally small
additive category and $\Modx[\mathcal{C}^{op}]$ 
denote the category of additive contravariant functors 
$\mathcal{C}\rightarrow\operatorname{Ab}$. 
Notice that $\Modx[\mathcal{C}^{op}]$ is an AB3* and AB5 abelian category having enough projectives and injectives,  and any projective is a direct summand of a coproduct of the form $\bigoplus_{i\in I}\Homx[\mathcal{C}][-][C_{i}]$, 
where $I$ is a set and $C_{i}\in\mathcal{C}\;\forall i\in I$. 
For more details the reader is referred to \cite[Sect. 2,  pp.184-187]{doi: 10.1080/00927877408548230}
 and \cite[Lem. 2]{martinez2014tilting}.

As we study $\Modx[\mathcal{C}^{op}]$ we will be interested in  the following  subcategories. 
The subcategory of {\bf finitely generated} functors $\modu\,(\C^{op})$ whose elements are the functors $F\in\Modx[\mathcal{C}^{op}]$ 
that admit an epimorphism $\bigoplus_{i\in I}\Homx[\mathcal{C}][-][C_{i}]\rightarrow F$ 
where $I$ is a finite set and $C_{i}\in\mathcal{C}\;\forall i\in I$ 
 \cite[Prop. 2.1.(b), p.186]{doi: 10.1080/00927877408548230}; and
 the subcategory of \textbf{finitely presented} functors $\mathrm{f.p.}(\C^{op})$ whose elements are  the functors $F\in\Modx[\mathcal{C}^{op}]$
 that admit an exact sequence $\Homx[\mathcal{C}][-][C']\rightarrow\Homx[\mathcal{C}][-][C]\rightarrow F\rightarrow 0, $
 where $C, C'\in\mathcal{C}.$ 
Furthermore, following \cite{doi: 10.1080/00927877408548230},  we denote by $\projx[\mathcal{C}^{op}]$ the 
category of all the finitely generated projective objects in $\Modx[\mathcal{C}^{op}]$. 

We point out that $\operatorname{Add}(\projx[\mathcal{C}^{op}])=\Proj(\Modx[\mathcal{C}^{op}])$. Therefore, since $\Modx[\mathcal{C}^{op}]$ has enough  projectives,  we have that $\pdr[\operatorname{f.p.}(\C ^{op})][\T]=1$ if and only if $\pd(\T)=1,$ for every class $\T \subseteq \modd[\mathcal{C}^{op}]$. 

We will be interested in particular when all idempotents in  $\C$  split. In that case $\C$ is called an 
\textbf{ annuli variety}. 

\begin{defn}\cite[Def. 8]{martinez2014tilting}\label{def:  funtor tilting}
 A class $\mathcal{T}\subseteq\Modx[\mathcal{C}^{op}]$
is a \textbf{tilting category} if the following conditions hold true.
\begin{description}
\item [{(FT0)}] $\mathcal{T}=\smdx[\mathcal{T}]\subseteq\modd[\mathcal{C}^{op}].$
\item [{(FT1)}] $\pdx[\mathcal{T}]\leq1.$
\item [{(FT2)}] $\mathcal{T}\subseteq\mathcal{T}^{\bot_{1}}.$
\item [{(FT3)}] $\coresdimr{\mathcal{T}}{\operatorname{Hom}_{\mathcal{C}}(-, C)}{\, }\leq1$
$\forall C\in\mathcal{C}$.
\end{description}
An object $T\in\Modx[\mathcal{C}]$ is a \textbf{big (small)
tilting functor} if $\Addx[T]$ ($\addx[T]$) is a tilting category.
\end{defn}

\begin{prop}\label{prop: funtor tilting}
Let  $\mathcal{T}$ be a tilting
category in $\Modx[\mathcal{C}^{op}]$. Then $\Genn[\mathcal{T}][1]=\mathcal{T}^{\bot}$
and $\operatorname{gen}_{1}(\mathcal{T})=\mathcal{T}^{\bot}\cap\modd[\mathcal{C}^{op}]$. 
\end{prop}
\begin{proof} It follows from \cite[Prop. 10]{martinez2014tilting}. 
\end{proof}

\begin{thm}\label{thm: 2.213 clase de funtores tilting} For  $\mathcal{T}\subseteq\modd[\mathcal{C}^{op}], $
 the following statements are equivalent.
\begin{itemize}
\item[$\mathrm{(a)}$] $\mathcal{T}$ is big $1$-$\Modx[\mathcal{C}^{op}]$-tilting.
\item[$\mathrm{(b)}$] $\mathcal{T}=\Addx[\mathcal{T}]$ is a tilting category which 
is precovering in $\mathcal{T}^{\bot_{1}}$.
\end{itemize}
\end{thm}

\begin{proof} 
(a) $\Rightarrow$ (b) Let $\mathcal{T}$ be big $1$-$\Modx[\mathcal{C}^{op}]$-tilting.
In particular $\mathcal{T}=\Addx[\mathcal{T}]$ and by (T5) $\mathcal{T}$
is precovering in $\mathcal{T}^{\bot_{1}}=\mathcal{T}^{\bot}$. Now (FT1) follows from 
(T1),  and (FT2) follows from (T2). 
 Let $\rho: =\{\Homx[\mathcal{C}][-][C]\}_{C\in\mathcal{C}}$
and $\omega: =\Addx[\rho]$. Then,  by Theorem \ref{thm: el par n-X-tilting} (a),  we get $\omega\subseteq{}^{\bot}(\mathcal{T}^{\bot})=\mathcal{T}^{\vee}$.
Moreover,  by Corollary \ref{cor: coronuevo pag 55},   $\coresdimr{\mathcal{T}}{\omega}{\, }\leq\pdr[\, ][\mathcal{T}]\leq1$ and thus 
 (FT3) holds true.
\

(b) $\Rightarrow$ (a) Let $\mathcal{T}=\Addx[\mathcal{T}]$ be a tilting category 
which is precovering in $\mathcal{T}^{\bot}.$ 
Since $\Modx[\mathcal{C}^{op}]$ has enough injectives,  $\mathcal{T}$ satisfies
(T4). Moreover,  $\mathcal{T}$ satisfies (T5) since $\mathcal{T}$ is precovering in 
$\mathcal{T}^{\bot}=\mathcal{T}^{\bot_{1}}.$
Moreover,  using that $\mathcal{T}=\Addx[\mathcal{T}]$ and Proposition \ref{prop: funtor tilting},  we get 
$\T ^{\bot}=\Genn[\mathcal{T}][1]=\operatorname{Gen}_{1}(\mathcal{T}).$ Thus
 $\mathcal{T}$ is $1$-$\Modx[\mathcal{C}^{op}]$-tilting
by Theorem \ref{prop: primera generalizacion}. 
\end{proof}

To state the small version of the above theorem, we will need to recall the following notion. Let $\C$ be an additive category. We recall that 
a morphism $g: C\rightarrow A$ is a \textbf{pseudo-kernel} of a morphism $f: A\rightarrow B$ if the sequence of functors 
$\Homx[\mathcal{C}][-][C]\xrightarrow{(-,g)}\Homx[\mathcal{C}][-][A]\xrightarrow{(-,f)}\Homx[\mathcal{C}][-][B]$
is exact. If any morphism in $\C$ has a pseudo-kernel,  we say that $\C$ has pseudo-kernels. The notion of \textbf{pseudo-cokernel} is introduced dually.

\begin{thm}\label{thm: 2.213'} Let $\mathcal{C}$ be an annuli variety with
pseudo-kernels and such that $\modd[\mathcal{C}^{op}]$ has enough
injectives. Then,  for a class $\mathcal{T}\subseteq\modd[\mathcal{C}^{op}]$, 
the following statements are equivalent.
\begin{itemize}
\item[$\mathrm{(a)}$] $\mathcal{T}$ is small $1$-$\modd[\mathcal{C}^{op}]$-tilting.
\item[$\mathrm{(b)}$] $\mathcal{T}=\addx[\mathcal{T}]$ is a tilting category which is  precovering
in $\mathcal{T}^{\bot_{1}}\cap\modd[\mathcal{C}^{op}]$.
\end{itemize}
\end{thm}
\begin{proof}
Notice that $\mathcal{T}^{\bot_{1}}=\mathcal{T}^{\bot}$ if $\pdr[\, ][\mathcal{T}]\leq1$.
Now,  by \cite[Props. 2.6 and 2.7]{enomoto2017classifying}, 
it follows that $\modd[\mathcal{C}^{op}]$ is a thick class in $\Modx[\mathcal{C}^{op}]$
and $\projx[\mathcal{C}^{op}]$ is a relative generator in $\modd[\mathcal{C}^{op}]$.
Using this fact, the rest of the proof follows as in the proof of Theorem \ref{thm: 2.213 clase de funtores tilting}.
\end{proof}

Recall that a \textbf{dualizing $R$-variety} is an $R$-category
$\mathcal{C}$,  where $R$ is an artinian ring,  such that $\mathcal{C}$
is an annuli variety and the functor $D: \modd[\mathcal{C}]\rightarrow\modd[\mathcal{C}^{op}]$, 
$D(M)(X)=\Homx[R][M(X)][I_{0}(\operatorname{top}(R))]$,  is a duality
 \cite[p.307]{auslander1974stable}. 

\begin{cor}\label{cor: modC es gruesa} Let $\mathcal{C}$ be a skeletally small
additive category with pseudo-kernels. Then,  $\modd[\mathcal{C}^{op}]=(\projx[\mathcal{C}^{op}])_{\infty}^{\wedge}$
and it is a thick class in $\Modx[\mathcal{C}^{op}]$.\end{cor}

\begin{proof}
It follows from \cite[Props. 2.6 and 2.7]{enomoto2017classifying}.
\end{proof}

\begin{lem}\label{rem:  obs a 2.213'} Let $\mathcal{C}$ be a dualizing $R$-variety
with pseudo-kernels and pseudo-cokernels. Then,  $\modd[\mathcal{C}^{op}]=(\projx[\mathcal{C}^{op}])_{\infty}^{\wedge}$
and it is a thick abelian subcategory of $\Modx[\mathcal{C}^{op}]$
with enough injectives. \end{lem}

\begin{proof}
By Corollary \ref{cor: modC es gruesa},  $\modd[\mathcal{C}^{op}]=(\projx[\mathcal{C}^{op}])_{\infty}^{\wedge}$
and it is thick in $\Modx[\mathcal{C}^{op}]$. Similarly,  by Corollary \ref{cor: modC es gruesa}, 
$\modd[\mathcal{C}]=\mathcal{P}(\mathcal{C})_{\infty}^{\wedge}$ and
it has enough projectives. Then,  by the duality $D: \modd[\mathcal{C}]\rightarrow\modd[\mathcal{C}^{op}]$,  we get that 
$\modd[\mathcal{C}^{op}]$ has enough injectives. 
\end{proof}

\begin{defn}\cite[Def. 6]{martinez2013tilting} Let $\mathcal{C}$ be an
annuli variety. A class $\mathcal{T}\subseteq\Modx[\mathcal{C}^{op}]$
is a \textbf{generalized $n$-tilting subcategory} if the following conditions hold true.
\begin{description}
\item [{(FGT0)}] $\mathcal{T}=\smdx[\mathcal{T}].$
\item [{(FGT1)}] $\mathcal{T}\subseteq(\projx[\mathcal{C}^{op}])_{n}^{\wedge}.$
\item [{(FGT2)}] $\mathcal{T}\subseteq\mathcal{T}^{\bot}.$
\item [{(FGT3)}] $\Homx[\mathcal{C}][-][C]\in\mathcal{T}^{\vee}\;\forall C\in\mathcal{C}$.
\end{description}

An object $T\in\Modx[\mathcal{C}^{op}]$ is a \textbf{generalized
small (big) $n$-tilting functor} if $\addx[T]$ ($\Addx[T]$) is a generalized
$n$-tilting subcategory.
\end{defn}

\begin{cor}
Let $\mathcal{C}$ be a skeletally small additive category with pseudo-kernels.
Then,  the functor $\Extx[n][\operatorname{Mod}(\mathcal{C}^{op})][M][-]$ commutes
with coproducts $\forall M\in\modd[\mathcal{C}^{op}]$ and $\forall n\geq1$.\end{cor}

\begin{proof}
It follows from \cite[Cor. 2]{martinez2013tilting} and Corollary \ref{cor: modC es gruesa}.
\end{proof}

\begin{thm}
Let $\mathcal{C}$ be an annuli variety with pseudo-kernels and $\mathcal{T}\subseteq\modd[\mathcal{C}^{op}]$.
Consider the following statements: 
\begin{itemize}
\item[$\mathrm{(a)}$] $\mathcal{T}$ is small $n$-$\modd[\mathcal{C}^{op}]$-tilting.
\item[$\mathrm{(b)}$] $\mathcal{T}=\mathcal{T}^{\oplus_{<\infty}}$ is a generalized $n$-tilting
subcategory and there is a $\modd[\mathcal{C}^{op}]$-injective relative
cogenerator in $\modd[\mathcal{C}^{op}]$.
\end{itemize}
Then,  (a) implies (b). Furthermore,  if $\mathcal{T}$ has pseudo-kernels, 
then (a) and (b) are equivalent. 
\end{thm}
\begin{proof}
(a) $\Rightarrow$ (b) Let $\mathcal{T}$ be small $n$-$\modd[\mathcal{C}^{op}]$-tilting.
By \cite[Prop. C.1.(2)]{bravo2019locally},  it follows that $\rho: =\left\{ \Homx[\mathcal{C}][-][C]\right\} _{C\in\mathcal{C}}$
is a $\modd[\mathcal{C}^{op}]$-projective relative generator in $\modd[\mathcal{C}^{op}]$.
Then,  there is a long exact sequence 
\[
0\rightarrow F_{n+1}\stackrel{f}{\rightarrow}\Homx[\mathcal{C}][-][C_{n}]\rightarrow\cdots\rightarrow\Homx[\mathcal{C}][-][C_{0}]\rightarrow T\rightarrow 0, 
\]
where $F_{n+1}\in\modd[\mathcal{C}^{op}]$. Consider the exact sequence
\[
\eta: \;\suc[F_{n+1}][\mbox{Hom}_{\mathcal{C}}(-, C_{n})][F_{n}][f][\, ]\mbox{.}
\]
Thus,  by the Shifting Lemma and (T1),   we have 
$$\Extx[1][\operatorname{Mod}\left(\mathcal{C}\right)][F_{n}][F_{n+1}]\simeq\Extx[n+1][\operatorname{Mod}\left(\mathcal{C}\right)][T][F_{n+1}]=0.$$
 Hence $\eta$ splits and then $F_{n}\in\projx[\mathcal{C}^{op}].$ 
Therefore $T\in\projx[\mathcal{C}^{op}]_{n}^{\wedge}$ and so (FGT1) holds true. Condition
(FGT2) follows from (T2). By Proposition \ref{prop: (a)-1}(e), 
$\rho\subseteq{}^{\bot}(\mathcal{T}^{\bot})\cap\modd[\mathcal{C}^{op}]={}{}^{\bot}(\mathcal{T}^{\bot})\cap\left(\modd[\mathcal{C}^{op}], \mathcal{T}\right)^{\vee}\subseteq\mathcal{T}{}^{\vee}\mbox{, }$
which proves (FGT3). Finally,  by (T4),  $\modd[\mathcal{C}^{op}]$ has a $\modd[\mathcal{C}^{op}]$-injective
relative cogenerator in $\modd[\mathcal{C}^{op}].$  
\

(b) $\Rightarrow$ (a) By (FGT1),  $\pdr[\operatorname{f.p.}(\mathcal{C}^{op})][\mathcal{T}]\leq\pdx[\mathcal{T}]\leq n$
and thus (T1) is satisfied. Since $\mathcal{T}\subseteq\modd[\mathcal{C}^{op}]$, 
using (FGT2),  we have $\mathcal{T}\cap\modd[\mathcal{C}^{op}]=\mathcal{T}\subseteq\mathcal{T}^{\bot}$
and then (T2) holds true. Note that (FGT3) implies (t3''). Thus
(T4) is satisfied since there is a $\mbox{f.p.}(\mathcal{C}^{op})$-injective
relative cogenerator in $\modd[\mathcal{C}^{op}]$ and $\mathcal{T}\subseteq\modd[\mathcal{C}^{op}]$.
Now,  using that  $\mathcal{T}$ has pseudo-kernels and \cite[Props. 6 and 7]{martinez2013tilting},  we get that $\mathcal{T}$ is precovering
in $\modd[\mathcal{C}^{op}].$ In
particular,  every $Z\in\mathcal{T}^{\bot}\cap\modd[\mathcal{C}^{op}]$
admits a $\mathcal{T}$-precover,  and hence (T5) holds true. Finally, 
by Lemma \ref{lem:  addT precub es AddT precub},  condition (T3) follows since
Definition \ref{def: condiciones T3} (t3'') is satisfied and $\modd[\mathcal{C}^{op}]$
is thick by Corollary \ref{cor: modC es gruesa}.
\end{proof}

\subsection{Silting modules,  quasitilting modules and $n$-$\mathcal{X}$-tilting
objects}

\renewcommandx\modd[1][usedefault,  addprefix=\global,  1=R]{\operatorname{mod}\left(#1\right)}

Silting modules were introduced in \cite{siltingmodulessurvey},  by Lidia Angeleri H\"ugel,  Frederik
Marks and Jorge Vitoria,   as a simultaneous
generalization of tilting modules over an arbitrary ring and support
$\tau$-tilting modules over a finite dimensional algebra. In this
section we will focus on understanding silting theory through $n$-$\mathcal{X}$-tilting
objects. Let us begin by recalling some known results on silting theory. 

\renewcommandx\Homx[3][usedefault,  addprefix=\global,  1=R,  2=M,  3=N]{\operatorname{Hom}{}_{#1}(#2, #3)}
\renewcommandx\Gen[1][usedefault,  addprefix=\global,  1=M]{\operatorname{Gen}_{1}\left(#1\right)}
\renewcommandx\Genn[2][usedefault,  addprefix=\global,  1=M,  2=n]{\operatorname{Gen}_{#2}(#1)}

By \cite[Lemdef 3.1 and Lem. 2.3]{siltingmodulessurvey}, it can be shown the following result.

\begin{lem}\label{lemdef quasitilting} 
For a ring $R$ and $T\in\Mod(R), $  the following statements
are equivalent.
\begin{itemize}
\item[$\mathrm{(a)}$] $\Genn[T][1]=\Genn[T][2]$ is a torsion class in $\Mod(R)$ and 
 $\Homx[R][T][-]$
is exact on $\Genn[T][1]$.
\item[$\mathrm{(b)}$] $\Genn[T][1]=\Genn[T][2]$ and $T\in{}^{\bot_{1}}\Genn[T][1]$.
\item[$\mathrm{(c)}$] $\Genn[T][1]=\overline{\Genn[T][1]}\cap T^{\bot_{1}}$,  where $\overline{\Genn[T][1]}$
is the class of all the submodules of modules in $\Genn[T][1]$.
\end{itemize}
\end{lem}

Let $R$ be a ring and $T\in\Mod(R)$. We recall from \cite[Lemdef 3.1]{siltingmodulessurvey} that $T$ is \textbf{quasitilting}
if $T$ satisfies one of the equivalent conditions in Lemma \ref{lemdef quasitilting}.
Let $E: =\End.$ We recall,  from \cite{faith1972modules}  that $T$ is \textbf{finendo} if $M$ is finitely generated as a left $E$-module.

\begin{defn} For a ring $R, $ we denote by
 $\operatorname{qtilt}(R)$ the class of all the left $R$-modules which are quasitilting and finendo. 
\end{defn}

Note that the relation $\sim$ on $\operatorname{qtilt}(R)$,  where 
$T_{1}\sim T_{2}$ if $\Addx[T_{1}]=\Addx[T_{2}], $
 is an equivalence relation on $\operatorname{qtilt}(R)$.

\begin{thm}\label{thm: quasitilting finendo vs torsion}\cite[Thm. 3.4]{siltingmodulessurvey}
Let $R$ be a ring. Then,  the map $T\mapsto\Gen[T]$ induces a bijection
between the quotient class $\operatorname{qtilt}(R)/\!\!\sim$ and the class of all the  torsion classes $\mathcal{T}\subseteq\Modx[R]$
satisfying the following condition:  
\begin{description}
\item [{(QT)}] every $M\in\Mod(R)$ admits a $\mathcal{T}$-preenvelope $\phi: M\to T_0$ with $\Coker(\phi)$ in ${}^{\bot_{1}}\mathcal{T}.$ 
\end{description}
\end{thm}

Let $R$ be a ring and let $\sigma$ be a morphism
in $\Projx[R].$ Following  \cite[Sect. 3.2]{siltingmodulessurvey},  we consider the class $\mathcal{D}_{\sigma}: =\left\{ X\in\Modx[R]\, |\, \Homx[][\sigma][X]\mbox{ is an epimorphism}\right\}.$ The reader can find the elementary properties of this class in \cite{siltingmodulessurvey}.  

Let $R$ be a ring and $T\in\Mod(R)$. Following \cite[Def. 3.7]{siltingmodulessurvey},  we recall that $T$ is\textbf{ partial silting} if there is a projective presentation $\sigma$ of $T$ such that the following two conditions hold true: 
{\bf (S1)} $\mathcal{D}_{\sigma}$ is a torsion class and {\bf(S2)} $T\in\mathcal{D}_{\sigma}.$ On the other hand,  it is said that $T$ is \textbf{silting} if there is a projective presentation $\sigma$ of $T$ such that
$\Gen[T]=\mathcal{D}_{\sigma}$. In such a case,  we say that $T$ is
silting with respect to $\sigma$. 

\begin{prop}\cite{siltingmodulessurvey} \label{prop: silting vs finendo quasitilting vs tilting} 
For a ring $R, $ the following statements hold true.
\begin{itemize}
\item[$\mathrm{(a)}$]  Every silting left $R$-module
is finendo and quasitilting.
\item[$\mathrm{(b)}$] Let $T\in\Mod(R).$ Then:  $T$ is
$1$-tilting $\Leftrightarrow$ $T$ is faithful and silting $\Leftrightarrow$
 $T$  is faithful,  finendo and quasitilting.
\end{itemize}
\end{prop}
\begin{proof} It follows from \cite[Props. 3.10 and 3.13 (2)]{siltingmodulessurvey}.
\end{proof}

\begin{prop}\cite[Prop. 3.11]{siltingmodulessurvey}\label{prop: silting parcial +S3 es silting-1}\label{prop: silting parcial +S3 es silting}
For a ring $R$ and a projective presentation $\sigma$
of $T\in\Mod(R), $ the following statements are equivalent.
\begin{itemize}
\item[$\mathrm{(a)}$] $T$ is a silting $R$-module with respect to $\sigma$.
\item[$\mathrm{(b)}$] $T$ is a partial silting $R$-module with respect to $\sigma$ and
the following condition holds true: 
\begin{description}
\item [{(S3)}] there is an exact sequence $R\overset{\phi}{\rightarrow}T_{0}\rightarrow T_{1}\rightarrow0$, 
where $T_{0}$, $T_{1}\in\Addx[T]$ and $\phi$ is a $\mathcal{D}_{\sigma}$-preenvelope of $R.$
\end{description}
\end{itemize}
\end{prop}

We are ready to start discussing silting theory through $n$-$\mathcal{X}$-tilting
theory.

\begin{lem}\label{lem: siltingvstiltin} Let $R$ be a ring and let $T\in\Mod(R)$ be such
that $\pdr[\mathcal{Y}][T]\leq1$,  where $\Gen[T]\subseteq\mathcal{Y}\subseteq\Modx[R]$.
Then,  the following statements hold true. 
\begin{itemize}
\item[$\mathrm{(a)}$] $\left(\Gen[T]\right)^{\bot_{1}}\cap\mathcal{Y}\subseteq\left(\Addx[T]\right)^{\bot}\cap\mathcal{Y}$.
\item[$\mathrm{(b)}$] $\left(\Gen[T]\right)^{\bot}\cap\mathcal{Y}=\left(\Gen[T]\right)^{\bot_{1}}\cap\mathcal{Y}$ if $\Gen[T]=\Genn[T][2].$
\end{itemize}
\end{lem}
\begin{proof} The item (a) follows from $\pdr[\mathcal{Y}][T]\leq1.$ To prove (b),  it is enough to show that $\left(\Gen[T]\right)^{\bot_{1}}\cap\mathcal{Y}\subseteq$
$\left(\Gen[T]\right)^{\bot}\cap\mathcal{Y}$. Consider $M\in\left(\Gen[T]\right)^{\bot_{1}}\cap\mathcal{Y}$
and $N\in\Gen[T]$. In particular $\Ext^1_\C(N, M)=0.$ Let $m\geq2.$
Since $\Gen[T]=\Genn[T][2], $
we can build an exact sequence 
$
0\rightarrow K_{m-1}\rightarrow T_{m-1}\rightarrow...\rightarrow T_{1}\rightarrow N\rightarrow0\mbox{, }
$
where $T_{i}\in\Addx[T]$ $\forall i\in[1, m-1]$ and $K_{m-1}\in\Gen[T]$.
 Using that $M\in\left(\Gen[T]\right)^{\bot_{1}}\cap\mathcal{Y}$
$\subseteq\left(\Addx[T]\right)^{\bot}, $ by the Shifting Lemma  we
have 
$
\Extx[m][][N][M]\cong\Extx[1][][K_{m-1}][M]=0
$
 and thus $M\in\Gen[T]^{\bot}\cap\Y$. 
\end{proof}

\begin{defn}
Let $R$ be a ring and let $\left(\mathcal{X}, \mathcal{Y}\right)$ be a pair of classes of objects in $\Modx[R].$ We denote by 
$\mathcal{X}_{(\mathcal{I}, \mathcal{Y})}$ the class of all the 
$R$-modules $M\in\mathcal{X}$ admitting an exact sequence 
$
\suc[M][I_{0}(X)][Y]\mbox{, }
$
where $X\in\mathcal{X}$,  $Y\mathcal{\in\mathcal{Y}}$ and $I_{0}(X)$
is the injective envelope of $X$ in $\Modx[R]$.
\end{defn}

\renewcommandx\Gennr[3][usedefault,  addprefix=\global,  1=\mathcal{T},  2=n,  3=\mathcal{X}]{\operatorname{Gen}_{#2}^{#3}(#1)}

\begin{lem}\label{lem:  para silting implica Gen-tilting }
Let $R$ be a ring and let 
$T\in\Mod(R)$ be quasitilting and such that $\pdr[\mathcal{Y}][T]\leq1$, 
where $\Gen[T]\subseteq\mathcal{Y}\subseteq\Modx[R]$. Then,  for $\omega(T): =\Genn[T][1]_{(\mathcal{I}, T^{\bot_{0}})}, $ the following
statements hold true.
\begin{itemize}
\item[$\mathrm{(a)}$] $T^{\bot}\cap\mathcal{Y}=T^{\bot_{1}}\cap\mathcal{Y}$. 
\item[$\mathrm{(b)}$] $\Gen[T]\subseteq T^{\bot}\cap\mathcal{Y}$.
\item[$\mathrm{(c)}$] $\Addx[T]$ is a relative $\Gen[T]$-projective generator in $\Gen[T]$. 
\item[$\mathrm{(d)}$] $\left(\Gen[T]\right)^{\bot_{1}}\cap\mathcal{Y}=\left(\Gen[T]\right)^{\bot}\cap\mathcal{Y}$.
\item[$\mathrm{(e)}$] $\omega(T)$ is a relative $\Gen[T]$-injective cogenerator in $\Gen[T]$.
\item[$\mathrm{(f)}$] $\omega(T)\subseteq\Gen[T]^{\bot}\subseteq T^{\bot}$.
\item[$\mathrm{(g)}$] $\Gennr[T][1][\operatorname{Gen}_{1}(T)]=\Gen[T]$.
\end{itemize}
\end{lem}
\begin{proof} The item (a) follows from $\pdr[\mathcal{Y}][T]\leq1, $ (b) can be shown from (a) and Lemma \ref{lemdef quasitilting} (b),  and (c) can be obtained from (b) and 
Lemma \ref{lemdef quasitilting} (b). On the other hand,  (d) can be obtained from Lemmas \ref{lemdef quasitilting} (b) and \ref{lem: siltingvstiltin} (b). 
 The first  inclusion in (f) follows from (e) and the second one  
from $T\in\Genn[T][1]$. Moreover,  (g) follows from Lemma \ref{lemdef quasitilting} (b).
\

Let us show (e).  By \cite[Lem. 2.3]{siltingmodulessurvey} and (b),  $\mathfrak{C}: =\left(\Gen[T], T^{\bot_{0}}\right)$
is a torsion pair. Let us prove that $\omega(T)$ is a relative cogenerator
in $\Genn[T][1]$. Let $X\in\Gen[T]$ and $i: X\rightarrow I_{0}(X)$
be its injective envelope. Consider the exact sequences 
$\eta: \;\suc[X][I_{0}(X)][Y][i]$ and $\mu: \;\suc[T'][I_{0}(X)][F][a][b]\mbox{, }$
where $\mu$ is the canonical exact sequence induced by $\mathfrak{C}$
with $T'\in\Gen[T]$ and $F\in T^{\bot_{0}}$. Note that $T'\in\omega(T)$
by definition. Now,  since $X\in\Gen[T]$ and $F\in T^{\bot_{0}}$, 
we have $bi=0$. Then,  $i$ factors through $a$. Consider the exact
sequence $\suc[X][T'][D][f]$,  where $af=i$. Note that $D\in\Gen[T]$
since $T'\in\Gen[T]$. Therefore,  $\omega(T)$ is a relative cogenerator
in $\Genn[T][1]$. 
\

It remains to show that $\omega(T)\subseteq\left(\Gen[T]\right)^{\bot}$.
Let $W\in\omega(T)$. Then,  there is an exact sequence 
$\nu: \quad\suc[W][I][F][\, ][\, ]$
with $I\in\Injx[R]$ and $F\in T^{\bot_{0}}$. Thus,  for each $G\in\Gen[T], $ we have  an epimorphism 
$\Homx[][G][F]\rightarrow\Extx[1][][G][W].
$
 Hence $\Extx[1][][G][W]=0$ since $\mathfrak{C}$
is a torsion pair. Finally,  by (d),  we have 
$W\in\Gen[T]^{\bot_{1}}\cap\Gen[T]\subseteq\Gen[T]^{\bot_{1}}\cap\mathcal{Y}\subseteq\Gen[T]^{\bot}$.
\end{proof}

\begin{thm}\label{thm:  silting es Gen-tilting}
Let $R$ be a ring and let $T\in\Mod(R)$
be quasitilting and such  that $\pdr[\mathcal{Y}][T]\leq1$,  where
$\Gen[T]\subseteq\mathcal{Y}\subseteq\Modx[R]$. Then,  $T$ is big
$1$-$\Gen[T]$-tilting. 
\end{thm}

\begin{proof}
 By Lemma \ref{lemdef quasitilting} (a),  we get that 
 $\Gen[T]$ is closed under extensions and direct summands.
 By Lemma \ref{lem:  para silting implica Gen-tilting } (e, f), 
 (T4) holds true; and 
(T5) is satisfied since $\Addx[T]$ is precovering. Then, 
by Lemma \ref{lem:  para silting implica Gen-tilting } (g) and Theorem \ref{prop: primera generalizacion}, 
it remains to show the equality  $\Gen[T]=T^{\bot}\cap\Gen[T], $ which follows from
 Lemma \ref{lem:  para silting implica Gen-tilting } (b).
 \end{proof}
 
\begin{rem} Let $R$ be a ring.  In \cite[Ex. 3.14]{siltingmodulessurvey},  it was given an
example of an $R$-module $T$,  with $\pdx[T]\leq1$,  such that $T$
is silting but it is not tilting. Therefore,  Theorem \ref{thm:  silting es Gen-tilting}
show the existence of $1$-$\Gen[T]$-tilting $R$-modules of projective
dimension $\leq1$ that are not tilting. 
\end{rem}

In what follows,  we will study $1$-$\Gen[T]$-tilting $R$-modules
with the goal of finding the conditions needed for a $1$-$\Gen[T]$-tilting
$R$-module to be silting. As a result of this pursuit,  we will prove in Theorem \ref{thm: quasitilt sii 1-Gen-tiltin}, 
for an $R$-module $T$ with $\pdx[T]\leq1$,  that $T$ is quasitilting
if and only if $T$ is big $1$-$\Gen[T]$-tilting. 

\begin{prop}\label{prop: n-gen-tilting} Let $R$ be a ring and let $T\in\Mod(R)$ be such
that $\Genn[T][1]$ is closed under extensions and $T$ is big $n$-$\Genn[T][1]$-tilting.
Then,  for $m: =\max\{1, n\}$,  the following statements hold true.
\begin{itemize}
\item[$\mathrm{(a)}$] $\Gennr[T][k][\operatorname{Gen}_{1}(T)]=\Genn[T][k+1]$ $\forall k\geq m.$
\item[$\mathrm{(b)}$] $\Genn[T][m+1]=T^{\bot}\cap\Gen[T].$
\item[$\mathrm{(c)}$] $\Genn[T][k+1]=T^{\bot}\cap\Gen[T]$ $\forall k\geq m.$
\item[$\mathrm{(d)}$] $\Addx[T]\subseteq T^{\bot}\cap\Gen[T]\subseteq\Genn[T][2]$ and $T^{\bot}\cap\Gen[T]$
is closed by $m$-quotients in $\Gen[T].$
\end{itemize}
\end{prop}
\begin{proof}
 Let us show (a). By Theorem \ref{prop: primera generalizacion} (c),    
$\Gennr[T][k][\operatorname{Gen}_{1}(T)]=\Gennr[T][k+1][\operatorname{Gen}_{1}(T)]\subseteq\Genn[T][k+1].$ Consider $X\in\Genn[T][k+1].$ Then,  there is an  exact sequence 
$0\rightarrow K\rightarrow T_{k}\overset{f_{k}}{\rightarrow}...\rightarrow T_{1}\overset{f_{1}}{\rightarrow}X\rightarrow0\mbox{, }$
 with $K\in\Gen[T]$ and $T_{i}\in\Addx[T]$ $\forall i\in[1, k]$.
Thus $X\in\Gennr[T][k][\operatorname{Gen}_{1}(T)]$ since $\Kerx[f_{i}]\in\Gen[T]$
$\forall i\in[1, k].$
Therefore,  (a) holds true. 
\

Finally,  (b),  (c) and (d) follow by (a) and Theorem \ref{prop: primera generalizacion} (b, c, d). 
\end{proof}
\renewcommandx\Extx[4][usedefault,  addprefix=\global,  1=i,  2=R,  3=M,  4=X]{\mathrm{Ext}{}_{#2}^{#1}\left(#3, #4\right)}

\begin{lem}\label{cor:  Gen-tilting es silting} Let $R$ be a ring and let $T\in\Mod(R)$ be big $1$-$\Gen[T]$-tilting and $\pdr[\mathcal{Y}][T]\le1$,  where $\overline{\Genn[T][1]}\subseteq\mathcal{Y}\subseteq \Mod(R).$ Then,  the following statements
hold true.
\begin{itemize}
\item[$\mathrm{(a)}$] $\Gen[T]\subseteq T^{\bot}$.
\item[$\mathrm{(b)}$] $\left(\Gen[T], T^{\bot_{0}}\right)$ is a torsion pair.
\item[$\mathrm{(c)}$] $\Gen[T]=\Gen[T]\cap T^{\bot}=\Genn[T][k+1]$ $\forall k\geq1$.
\item[$\mathrm{(d)}$] $T$ is quasitilting.
\item[$\mathrm{(e)}$] $\left(\Gen[T]\right)^{\bot}\cap\mathcal{Y}=\left(\Gen[T]\right)^{\bot_{1}}\cap\mathcal{Y}\subseteq T^{\bot}\cap\mathcal{Y}$.
\item[$\mathrm{(f)}$] $\Gen[T]\cap{}{}^{\bot_{1}}\left(\Gen[T]\right)=$ ${}{}^{\bot}(T^{\bot}\cap\Gen[T])\cap\Gen[T]=$
$\Addx[T]=$ $={}^{\bot}(T^{\bot})\cap\Gen[T]=$ $(\Addx[T])^{\vee}\cap\Gen[T]$.
\end{itemize}
\end{lem}

\begin{proof} (a)  Let $X\in\Gen[T]$,  $H: =\Homx[R][T][X]$,  and 
$u: T^{\left(H\right)}\rightarrow X$ be the morphism defined by $(m_{f})_{f\in H}\mapsto\sum_{f\in H}f(m_{f})$.
Since $X\in\Gen[T]$,   $u$ is an epimorphism. Thus,  we have the exact sequence
$\eta: \: \suc[K][T^{(H)}][X][][u]\mbox{.}$
Now,  applying $\Homx[][T][-]$ to $\eta$,  we get the long exact sequence 
\[
\Homx[][T][T^{(H)}]\xrightarrow{\Homx[][T][u]}\Homx[][T][X]\rightarrow\Extx[1][R][T][K]\rightarrow\Extx[1][R][T][T^{(H)}]\mbox{.}
\]
By (T2),  $\Extx[1][R][T][T^{(H)}]=0$ and thus $\Extx[1][][T][K]=0$  since 
$\Homx[][T][u]$ is surjective. Therefore,  $K\in T^{\bot_{1}}\cap\mathcal{Y}\subseteq T^{\bot}$
since $\pdr[\mathcal{Y}][T]\leq1$. Then,  using that $T^{(H)},  K\in T^{\bot}$
and $T^{\bot}$ is closed under mono-cokernels,  we get $X\in T^{\bot}$.
\

(b)It follows from (a) and \cite[Lem. 2.3]{siltingmodulessurvey}.
\

(c) By (b),  $\Genn[T][1]$ is closed under extensions. Then,  (c) follows
straightforward by (a) and Proposition \ref{prop: n-gen-tilting} (c).
\

(d) Observe that $\Gen[T]=\Genn[T][2]$ by (c),  and that $\Homx[][T][-]$
is exact on $\Gen[T]$ by (a). Moreover,  $\Gen[T]$ is a torsion class
by (b). Therefore,  $T$ is quasitilting by Lemma \ref{lemdef quasitilting} (a).
\

(e)  It follows by (c) and Lemma \ref{lem: siltingvstiltin}.
\

(f) By (d) and \cite[Lem. 3.3]{siltingmodulessurvey},  we have $\Addx[T]=$ $\Gen[T]\cap{}{}^{\bot_{1}}\left(\Gen[T]\right)$.
Note that $\Addx[T]\subseteq\Genn[T][1]\cap(\Addx[T])^{\vee}$. In order to prove that the inclusion above is an equality,  observe firstly that $\pdr[\operatorname{Gen}_{1}(T)][(\operatorname{Add}(T))^{\vee}]=\pdr[\operatorname{Gen}_{1}(T)][\operatorname{Add}(T)]=0$
by \cite[Lem. 4.3]{parte1} and (a). Let $X\in\Genn[T][1]\cap(\Addx[T])^{\vee}.$ Then,  there is an exact sequence
$\eta: \: \suc[X][T_{0}][X'][\, ][\, ]$
with $T_{0}\in\Addx[T]$ and $X'\in(\Addx[T])^{\vee}.$ Since $\pdr[\operatorname{Gen}_{1}(T)][(\operatorname{Add}(T))^{\vee}]=0, $ we have that $\eta$
splits and then $X\in\Addx[T]$. Therefore $\Addx[T]=$ $\Gen[T]\cap(\Addx[T])^{\vee}$.
\\
Now,  by (b),  we know that $\Genn[T][1]$ is closed under extensions and direct summands. Hence,  applying Theorem \ref{thm: el par n-X-tilting} (a),  we get the equalities
${}^{\bot}(T^{\bot}\cap\Genn[T][1])\cap\Genn[T][1]=(\Addx[T])_{\Genn[T][1]}^{\vee}\cap\Genn[T][1]={ }^{\bot}\left(T^{\bot}\right)\cap\Gen[T]\mbox{.}$
Finally,  observe that $(\Addx[T])_{\Genn[T][1]}^{\vee}=(\Addx[T])^{\vee}$.
\end{proof}

\begin{thm}\label{thm: quasitilt sii 1-Gen-tiltin}Let $R$ be a ring and $T\in\Mod(R)$
with $\pdr[\mathcal{Y}][T]\le1$,  where $\overline{\Genn[T][1]}\subseteq\mathcal{Y}\subseteq \Mod(R)$.
Then,  $T$ is quasitilting if and only if $T$ is big $1$-$\Gen[T]$-tilting. \end{thm}
\begin{proof}
Note that $\Gen[T]\subseteq\overline{\Genn[T][1]}$. Then,  the result
follows from Lemma \ref{cor:  Gen-tilting es silting} (d) and
Theorem \ref{thm:  silting es Gen-tilting}.
\end{proof}

The following lemma is contained in the proof of \cite[Prop. 5.6]{bazzoni2017pure}. 

\begin{lem}\cite{bazzoni2017pure}\label{lem: superfluo} 
Let $R$ be a ring and let $T\in\Mod(R)$ be such that
$\Gen[T]\subseteq T^{\bot_{1}}$. If $P_{1}\overset{\sigma}{\rightarrow}P_{0}$
is a projective presentation of $T$ such that $\Kerx[\sigma]$ is
a superfluous submodule of $P_{1}$,  then $\Genn[T][1]\subseteq\mathcal{D}_{\sigma}$.\end{lem}

\begin{thm}\label{thm: silting vs 1-Gen-tilting} 
Let $R$ be a ring,  $T\in\Mod(R)$
with $\pdr[\mathcal{Y}][T]\leq1$ and $\overline{\Gen[T]}\subseteq\mathcal{Y}\subseteq\Modx[R], $ and let $\sigma: P_{1}\rightarrow P_{0}$
be a projective presentation 
of $T$ with $\Kerx[\sigma]$ a superfluous submodule of $P_{1}.$ Then,  
 the following conditions are equivalent: 
\begin{itemize}
\item[$\mathrm{(a)}$] $T$ is silting with respect to $\sigma.$
\item[$\mathrm{(b)}$] $T$ is big $1$-$\Gen[T]$-tilting and there is a $\mathcal{D}_{\sigma}$-preenvelope
$\phi: R\rightarrow T_{0}$ such that $T_{0}\in\Addx[T]$. 
\end{itemize}
Moreover,  $\mathcal{D}_{\sigma}=\Gen[T]=T^{\bot}\cap\mathcal{Y}=T^{\bot_{1}}\cap\mathcal{Y}$ if one of the above conditions holds true. 
\end{thm}
\begin{proof}
(a) $\Rightarrow$ (b) By Proposition \ref{prop: silting parcial +S3 es silting-1}, 
there is an exact sequence 
$R\overset{\phi}{\rightarrow}T_{0}\rightarrow T_{1}\rightarrow0\mbox{, }$
where $T_{0}$, $T_{1}\in\Addx[T]$ and $\phi$ is a $\mathcal{D}_{\sigma}$-preenvelope.
Finally,  note that $T$ is big 1-$\Gen[T]$-tilting by Proposition \ref{prop: silting vs finendo quasitilting vs tilting} (a)
and Theorem \ref{thm:  silting es Gen-tilting}.
\

(b) $\Rightarrow$ (a) Observe that $\Gen[T]\subseteq\mathcal{D}_{\sigma}$
by Lemma \ref{cor:  Gen-tilting es silting} (a) and Lemma \ref{lem: superfluo}.
Let us show that $\mathcal{D}_{\sigma}\subseteq\Gen[T]$. We know
that there is a $\mathcal{D}_{\sigma}$-preenvelope $\phi: R\rightarrow T_{0}$.
Then,  for $X\in\mathcal{D}_{\sigma}$,  every epimorphism $R^{(\alpha)}\rightarrow X$
factors through the preenvelope $\phi^{(\alpha)}$ via an epimorphism
$T_{0}^{(\alpha)}\rightarrow X$. Therefore,  $\mathcal{D}_{\sigma}\subseteq\Gen[T]$.\end{proof}

Let $R$ be a ring,  $T\in\Mod(R)$ and let $\sigma: P_{1}\rightarrow P_{0}$
be a projective presentation of $T$. The condition of $\Ker({\sigma})$
being a superfluous submodule of $P_{1}$ means that the induced morphism
$P_{1}\rightarrow\im[\sigma]$ is a projective cover. We can find
different contexts where this kind of projective resolution can be
built. For example,  in \cite[Cor. 5.7]{bazzoni2017pure},  the
following conditions on the ring and the module are mentioned:  (i) $R$ is left perfect; (ii) $R$ is semi perfect and $T$ is finitely presented; and (iii) $\pdx[T]\leq1.$

\begin{rem}\label{rem: silting} 
In \cite{breaz2018torsion},  Simion Breaz and Jan {\v{Z}}emli{\v{c}}ka 
studied the torsion classes generated by silting modules. In particular, 
this kind of torsion classes are characterized for perfect and hereditary
rings. Namely,  for a left perfect (or a left hereditary) ring $R$ and
$S\in\Modx[R]$ such that $\Gen[S]$ is a torsion class,  they proved
in \cite[Thms. 2.4 and 2.6]{breaz2018torsion} that $S$
is silting if and only if there is a $\Gen[S]$-preenvelope $\epsilon: R\rightarrow M$
such that $M\in{}{}^{\bot_{1}}\Gen[S]$. Note that,  by \cite[Lem. 3.3]{siltingmodulessurvey}, 
Proposition \ref{prop: silting vs finendo quasitilting vs tilting} and Theorem \ref{thm: quasitilt sii 1-Gen-tiltin},  the preenvelope
$\epsilon: R\rightarrow M$ is the same preenvelope that appears in
Theorem \ref{thm: silting vs 1-Gen-tilting} (b).
Comparing these results,  we observe the following. 
\
 
 (1)  Let $R$ be a left perfect ring and $T\in\Mod(R)$ be big $1$-$\Gen[T]$-tilting.  Since $R$ is left perfect,  for every left $R$-module
we can find a projective presentation $\tau: Q_{1}\rightarrow Q_{0}$, 
with $\Kerx[\tau]$ superfluous in $Q_{0}$. Then,  by Theorem \ref{thm: silting vs 1-Gen-tilting}
and \cite[Thm. 2.4]{breaz2018torsion},  $T$ is silting with respect
to a projective presentation $\rho$ if and only if $T$ is silting
with respect to every projective presentation $\sigma: P_{1}\rightarrow P_{0}$
of $T$,  with $\Kerx[\sigma]$ superfluous in $P_{1}$. 
\

(2) Let $R$ be a left hereditary ring and $T\in\Mod(R)$ be big $1$-$\Gen[T]$-tilting. Since $R$ is left hereditary,  for every left $R$-module
we can find a monomorphic projective presentation $\tau: Q_{1}\rightarrow Q_{0}$, 
and consequently,  with $\Kerx[\tau]$ superfluous in $Q_{0}$. Then, 
by Theorem \ref{thm: silting vs 1-Gen-tilting} and \cite[Thm. 2.6]{breaz2018torsion}, 
$T$ is silting with respect to a projective presentation $\rho$
if and only if $T$ es silting with respect to every projective presentation
$\sigma: P_{1}\rightarrow P_{0}$ of $T$ with $\Kerx[\sigma]$ superfluous
in $P_{1}$. 
\

(3) Let $R$ be a ring.  Proposition \ref{prop: silting parcial +S3 es silting-1}
states that,  for every silting $S\in\Mod(R), $ there is a
$\Gen[S]$-preenvelope $\epsilon: R\rightarrow M$ with $M\in\Addx[S]$.
However,  there are examples where the existence of this preenvelope
does not imply that $S$ is silting (see \cite[Ex. 2.5]{breaz2018torsion}
and \cite[Ex. 5.4]{silting2017}). Therefore,  it is worth noting
that Theorem \ref{thm: silting vs 1-Gen-tilting} give enough conditions in order to have that 
the existence of such preenvelope implies the silting property.
\

(4) In \cite[Cor. 2.9]{breaz2018torsion},  it is proved for
a left perfect (or a left hereditary) ring $R$ that,  for every quasitilting
finendo $Q\in \Mod(R), $ there is a silting $T\in\Mod(R)$ such
that $\Addx[T]=\Addx[Q]$. It is important mentioning that this is not
true for every  ring, see \cite[Ex. 2.10]{breaz2018torsion}
and \cite[Ex. 5.4]{silting2017}. Therefore,  it is  worth noting
that Theorem \ref{thm: silting vs 1-Gen-tilting} give us enough conditions
for a quasitilting finendo $R$-module to be silting. \\
Indeed,  let $T$ be a quasitilting finendo $R$-module such that
$\pdr[\operatorname{Gen}_{1}(T)][T]\leq1$. By Theorems \ref{thm: quasitilt sii 1-Gen-tiltin}
and \ref{thm: quasitilting finendo vs torsion},  $T$ satisfies Theorem \ref{thm: silting vs 1-Gen-tilting}(b).
Therefore,  if $T$ admits a projective presentation $\sigma: P_{1}\rightarrow P_{0}$
with $\Kerx[\sigma]$ superfluous in $P_{1}$,  then $T$ is silting
with respect to $\sigma$ by Theorem \ref{thm: silting vs 1-Gen-tilting}. 
\end{rem}

\section{Tilting and cotorsion pairs in quiver representations}

Let $Q$ be a quiver. That is, a directed graph given by a set of
vertices $Q_{0}$, a set of arrows $Q_{1}$, a \emph{source map} $s:Q_{1}\rightarrow Q_{0}$
and a \emph{target map} $t:Q_{1}\rightarrow Q_{0}$. In this context,
a \textbf{path} $\gamma$ (of length $n$) starting at $s(\gamma):=x$ and ending at $t(\gamma):=y$,
is a sequence of arrows $\gamma:=\alpha_{1}\alpha_{2}\cdots\alpha_{n}$ such
that: $s(\alpha_{n})=x$, $s(\alpha_{k})=t(\alpha_{k+1})$ $\forall k\in[1,,n-1]$,
and $t(\alpha_{1})=y$. Here, we consider the case $n=0$ as the trivial
path ending and starting at $x=y$. It can be defined the 
\emph{free category}, or \emph{category of paths}, generated by $Q$ as
the category whose objects are the vertices in $Q$ and the morphisms are
the paths in $Q.$ The composition of morphisms in the path category is the concatenation of paths in $Q.$
 A quiver $Q$ is \textbf{finite} if $Q_{0}$
and $Q_{1}$ are finite, and $Q$ is \textbf{acyclic}
if there are no paths $\gamma=\alpha_{1}\alpha_{2}\cdots\alpha_{n}$ of length
$n\geq1$ with $s(\gamma)=t(\gamma)$. 

Given an abelian category $\C$, we understand the category $\Rep(Q,\C)$ of representations of the quiver $Q$ in $\C$ as the category of functors from the free category generated
by $Q$ to $\C.$ The
basic tools for working with $\Rep(Q,\C)$ are the following functors.
For $x\in Q_{0}$, we have the evaluation functor $e_{x}:\Rep(Q,\C)\rightarrow\C$
which sends $F$ to its evaluation $F_{x}:=F(x)$, and the stalk functor $s_{x}:\mathcal{C}\rightarrow\Rep(Q,\C)$,
defined by $(s_{x}(C))_{x}=C$ and $(s_{x}(C))_{y}=0$ for all $y\neq x$.
It is well known that, under certain conditions \cite[Prop. 2.18]{AM23}, the functor $e_{x}$
admits a right adjoint $g_{x}:\C\rightarrow\Rep(Q,\C)$ and a left
adjoint $f_{x}:\C\rightarrow\Rep(Q,\C)$. In particular, if $\C$
is AB4 and AB4{*}, or $Q$ is finite and acyclic, then these functors exist and can
be defined as $f_{x}(C)=C^{(Q(x,-))}$ and $g_{x}(C)=C^{Q(-,x)},$ see \cite[Def. 2.16, Prop. 2.17]{AM23}.
Given a class $\X\subseteq\C$, we consider the classes $g_{*}(\X):=\bigcup_{i\in Q_{0}}g_{i}(\X)$
and $s_{*}(\X):=\bigcup_{i\in Q_{0}}s_{i}(\X).$  In case we need to
highlight in which quiver we are working, we will use the notation
$f_{x}^{Q}:=f_{x}$, $g_{x}^{Q}:=g_{x}$ and $e_{x}^{Q}:=e_{x},$ 
see \cite[Section 2.11]{AM23} for more details.

Let $Q$ be a quiver. In case $Q_{0}$ is not finite, it is common to consider the full
subcategory $\Rep^{f}(Q,\C)$ of \textbf{finite-support representations},
i.e. representations $F\in\Rep(Q,\C)$ such that the support $\Supp(F):=\{x\in Q_{0}\,|\:F(x)\neq0\}$ of $F$
is finite. Observe that $\Rep(Q,\C)$ is an abelian category and $\Rep^{f}(Q,\C)\subseteq \Rep(Q,\C)$
is a full abelian subcategory closed under subobjects and
quotients \cite[Rk. 5.4]{AM23}. Following \cite[Def. 2.6]{AM23}, we recall that the quiver $Q$ is \textbf{finite-cone-shape} if for every vertex
$x\in Q_0$ there exists a finite number of paths ending or starting at $x.$ In this case it is known that the
categories $\Rep(Q,\C)$ and $\Rep^{f}(Q,\C)$ are intimately related \cite[Sect. 5]{AM23}.
For example, let $Q$ be finite-cone-shape. It is known that the functors $e_{x}$, $f_{x}$, $g_{x}$ can be restricted to
$\Rep^{f}(Q,\C)$ and such restrictions also form adjoint pairs \cite[Prop. 5.14]{AM23}. Moreover if $\C$ has enough projectives and injectives then so do $\Rep^{f}(Q,\C)$ and 
$\Rep(Q,\C),$ and  $\Inj(\Rep^{f}(Q,\C))=\Inj(\Rep(Q,\C))\cap\Rep^{f}(Q,\C)$ and
 $\Proj(\Rep^{f}(Q,\C))=\Proj(\Rep(Q,\C))\cap\Rep^{f}(Q,\C)$ \cite[Cor. 5.18]{AM23}.

\begin{rem}
Throughout this section, we will be using the results of \cite{AM23}.
So it is worth saying a few words about the hypotheses that appear
in such paper. Namely, in \cite{AM23}, it is introduced
certain cardinal numbers that measure the complexity of a quiver $Q.$ These cardinals are denoted by: $\lccn$(Q), $\rccn(Q)$,
$\ltccn(Q)$, $\rtccn(Q)$, $\ccn(Q)$, $\tccn(Q)$, $\lmcn(Q)$
and $\alpha(Q)$ \cite[Defs. 2.9, 2.10 and 5.8]{AM23}.
By using these cardinals, the conditions $AB3(\kappa),$ $AB4(\kappa),$
$AB3^{*}(\kappa)$ and $AB4^{*}(\kappa)$ appear for an infinite cardinal $\kappa$
greater or equal to these cardinals (such conditions are the usual
Grothendieck conditions restricted to coproducts or products of $<\kappa$
objects). Since we will only be interested in finite-cone-shape quivers,
all these cardinals turn out to be $\leq\aleph_{0}$ and the quiver
turns out to be \emph{rooted} \cite[Sect. 2.8]{AM23}.
In particular, for the scope of this section, the reader does not need these cardinal numbers (but in order to understand the statements we are using from \cite{AM23} they actually do need it) because any abelian category is $AB3(\aleph_{0}),$ $AB4(\aleph_{0}),$ 
$AB3^{*}(\aleph_{0})$ 
and $AB4^{*}(\aleph_{0}).$ It is also worth mentioning that in \cite[Section 5]{AM23}
certain full subcategories of $\Rep(Q,\C)$ denoted by $\Rep^{ft}(Q,\C)$
and $\Rep^{fb}(Q,\C)$ are studied. However, in the finite-cone-shape
case, one has that $\Rep^{fb}(Q,\C)=\Rep^{f}(Q,\C)=\Rep^{ft}(Q,\C),$ see \cite[Lem. 5.11]{AM23}. 
\end{rem}

Without further ado, we present the first central theorem of this section,
which tells us how to build a tilting class in the abelian subcategory $\Rep^{f}(Q,\C)\subseteq\Rep(Q,\C)$ and also in $\Rep(Q,\C)$ from
a tilting one in $\mathcal{C}.$ We recall that an abelian $AB3$ category which has enough injectives is $AB4.$  

\begin{thm}\label{thm:tilting repf} Let $Q$ be a finite-cone-shape quiver, 
$\mathcal{C}$ an abelian category with enough injectives,  $\mathcal{T}=\add(\mathcal{T})$
a precovering and $n$-$\C$-tilting class in $\C,$ and let $\X:=\Rep^{f}(Q,\C)\subseteq\Rep(Q,\C).$ Then  $\mathbb{T}:=\add\left(g_{*}(\mathcal{T})\right)\subseteq \X$ and the following statements hold true.
\begin{enumerate}
\item $\mathbb{T}$ is precovering in $\X.$
\item $\mathbb{T}$ is an $(n+1)$-$\X$-tilting class in $\X.$
\item $\mathbb{T}$ is an $(n+1)$-$\X$-tilting class in $\Rep(Q,\C).$
\item Let $\C$ be $AB3,$  $\T=\Add(T)$ for some $T\in\C$ and $S:=\bigoplus_{x\in Q_0}g_x(T).$ Then $\mathbb{T}=\Add(S)\cap\X$ and $S$ is a big $(n+1)$-$\X$-tilting object in $\Rep(Q,\C).$
\end{enumerate}
\end{thm}
\begin{proof}
Before proceeding with the proof, we point out the following facts: (i) the functors 
$f_{i},g_{i}:\C\to\Rep(Q,\C)$ are exact for
all $i\in Q_{0}$ \cite[Rk. 2.19 (c)]{AM23}; (ii) $\pd(\T)\leq n$, $\T\subseteq\T^{\bot}$ and there
is a class $\omega\subseteq\mathcal{T}^{\vee}$ which is a 
generator in $\mathcal{C};$ and (iii) $\Rep^{ft}(Q,\C)=\Rep^{f}(Q,\C)=\Rep^{fb}(Q,\C)$ and
$g_{z}(C),f_{z}(C)\in\Rep^{f}(Q,\C)$ for all $z\in Q_{0}$ and all
$C\in\C$  \cite[Prop. 5.6 (a), Prop. 5.7 (a), Lem. 5.11 ]{AM23}. In particular, we get that $\mathbb{T}\subseteq \Rep^{f}(Q,\C).$
\

(a) It follows from \cite[Prop. 5.14 (b)]{AM23}. 
\

(b) We proceed by proving all the conditions of Definition
\ref{def: X-tilting}. In what follows, we will be working in the abelian category
$\X=\Rep^{f}(Q,\C)$. In particular, for each $F\in\X,$ we have 
$
\pd(F):=\min\{n\in\mathbb{N}\,|\:\Ext_{\X}^{k}(C,-)=0\:\forall k>n\};
$
 and for all $\mathcal{Z}\subseteq\X,$ we set
$
\mathcal{Z}^{\bot}:=\{F\in\X\,|\:\Ext_{\X}^{k}(-,F)|_{\mathcal{Z}}=0\;\forall k\geq1\}.
$
\

{\bf (T0):}  By definition $\mathbb{T}$ is closed under direct summands.
\

{\bf (T1):} $\pd(\mathbb{T})=\pd(g_{*}(\mathcal{T}))\leq \pd(\T)+1,$ see \cite[Prop. 5.14 (d)]{AM23}. 
\

{\bf (T2):} Let $T,T'\in\mathcal{T}$ and $x,y\in Q_{0}$. Then, for
$k\geq1$, we have that 
\[
\Ext_{\Rep^{f}(Q,\C)}^{k}(g_{x}(T),g_{y}(T'))\cong\Ext_{\mathcal{C}}^{k}(e_{y} g_{x}(T),T')=\Ext_{\mathcal{C}}^{k}(T^{Q(y,x)},T')=0\text{,}
\]
where the above isomorphism is given by \cite[Prop. 5.7 (d)]{AM23}
and the last equality follows from the fact that $Q(y,x)$ is finite.
Therefore $g_{*}(\mathcal{T})\subseteq (g_{*}(\mathcal{T}))^{\bot}$ and thus 
$\mathbb{T}\subseteq\mathbb{T}^{\bot}.$
\

{\bf (T3):} We know that there is a class $\omega\subseteq\mathcal{T}^{\vee}$ which is a generator in $\mathcal{C}.$ By \cite[Prop. 3.30 (a)]{AM23}, we get that
$\Omega:=\coprod_{<\aleph_{0}}f_{*}(\omega)$ is a generator in $\Rep^{f}(Q,\C).$
\

 Let us show that $\Omega\subseteq \mathcal{T}^{\vee}.$
Consider $W\in\omega$ and $i\in Q_{0}$. Using that $W\in\mathcal{T}^{\vee},$ we have that $g_{i}e_{j}f_{k}(W)\in g_{i}\left(\mathcal{T}\right)^{\vee}$
for all $i,j,k\in Q_{0}$ since $g_{i}$ is exact and $\mathcal{T}=\add(\mathcal{T})$.
Now, by \cite[Cor. 5.12 (b2)]{AM23}, for all $k\in Q_{0}$
there is a short exact sequence 
\[\eta:\;
\suc[f_{k}(W)][\prod_{i\in Q_{0}}g_{i}e_{i}(f_{k}(W))][\prod_{\rho\in Q_{1}}g_{s(\rho)}e_{t(\rho)}(f_{k}(W))]
\]
in $\Rep^f(Q,\C).$ Using that $Q$ is a finite-cone-shape quiver, it can be shown that only a finite number of factors in $\prod_{i\in Q_{0}}g_{i}e_{i}(f_{k}(W))$ and $\prod_{\rho\in Q_{1}}g_{s(\rho)}e_{t(\rho)}(f_{k}(W))$
are not zero. Indeed, $e_{i}(f_{k}(W))\neq0$ $\Leftrightarrow$ $Q(k,i)\neq \emptyset,$ 
and hence $g_{i}e_{i}(f_{k}(W))\neq 0$ only for $i\in t(Q(k,-));$ however the set
 $t(Q(k,-))$ is finite since $Q(k,-)$ is finite. Similarly, $g_{s(\rho)}e_{t(\rho)}(f_{k}(W))\neq 0$ $\Leftrightarrow$
 $t(\rho)\in t(Q(k,-)),$ and thus, there is only a finite
number of arrows $\rho$ such that $g_{s(\rho)}e_{t(\rho)}(f_{k}(W))\neq0$
since the set $\bigcup_{i\in t(Q(k,-))}Q(i,-)$ is finite. Then, 
it follows that the middle and right terms in the exact sequence $\eta$ belong to 
$\mathbb{T}{}^{\vee}$. Therefore, by \cite[Cor. 4.21]{parte1},
we have that $f_{k}(W)\in\mathbb{T}^{\vee}$ and hence $\Omega=\coprod_{<\aleph_{0}}f_{*}(\omega)\subseteq\mathbb{T}{}^{\vee}$.
\

{\bf (T4):} Since $\C$ has enough injectives, we get from \cite[Cor. 5.18(e)]{AM23} that $\Rep^f(Q,\C)$ has enough injectives; and thus (T4) follows.
\

{\bf (T5):} This follows from (a).
\

Let us prove (c). We show that $\mathbb{T}$ is $(n+1)$-$\X$-tilting in $\Rep(Q,\C),$ for $\X:=\Rep^f(Q,\C).$ We point out that $\X\subseteq \Rep(Q,\C)$ is a full abelian subcategory which is closed under subobjects and quotients \cite[Rk. 5.4]{AM23}.
We proceed by showing that all the conditions of Definition
\ref{def: X-tilting} hold true for $\mathbb{T}$ and $\X.$ In what follows we will be working in
the abelian category $\Rep(Q,\C)$. In particular, the $\X$-projective dimension of $F\in\Rep(Q,\C)$ is given by 
$\pd_\X(F)=\min\{n\in\mathbb{N}\,|\:\Ext_{\Rep(Q,\C)}^{k}(F,-)|_{\X}=0\:\forall k>n\}.$
 and 
$
\mathcal{Z}^{\bot}:=\{F\in\Rep(Q,\C)\,|\:\Ext_{\Rep(Q,\C)}^{k}(-,F)|_{\X}=0\;\forall k\geq1\},$
 for $\mathcal{Z}\subseteq\Rep(Q,\C)$. 
\
 
  {\bf (T1):} Let $T\in\mathcal{T}$ and $z\in Q_{0}$. By replacing $\pd$
(the projective dimension in the abelian category $\Rep^{ft}(Q,\C)$)
with $\pd_\X$ (the relative projective dimension in $\Rep(Q,\C)$) in the proof of \cite[Prop. 5.9 (a)]{AM23} and using \cite[Prop. 2.18 (a)]{AM23}
we can conclude that
\begin{alignat*}{1}
\pd_\X(g_{z}(T)) & \leq\sup_{i\in Q_{0}}\{\pd_{\C}(T^{Q(i,z)})\}+1=\pd_{\C}(T)+1\text{.}
\end{alignat*}
Therefore $\pd_\X(\mathbb{T})=\pd_\X(g_*(\T))\leq \pd(\T)+1.$
\

 {\bf (T2):} It follows as in the proof of (T2) in (b) by using \cite[Prop. 2.18(b)]{AM23}.
 \

 {\bf (T3):} It follows from the condition (T3) in (a) since  $\mathbb{T}\subseteq\X$ and $\X\subseteq \Rep(Q,\C)$ is a full abelian subcategory which is closed under subobjects and quotients.
 \

 {\bf (T4):} Since $\C$ has enough injectives, we get from \cite[Cor. 5.18 (e)]{AM23} that $\X$ has an $\X$-injective relative cogenerator. Therefore (T4) holds true since  $\mathbb{T}\subseteq\X.$ 
\

 {\bf (T5):} This follows from (c) and the inclusion $\mathbb{T}\subseteq\X.$
 \
 
 Let us prove (d). We show that the class $\Add(S)$ is $(n+1)$-$\X$-tilting  in $\Rep(Q,\C),$ for $\X=\Rep^f(Q,\C).$ We proceed by showing that all the conditions of Definition
\ref{def: X-tilting} hold true for $\Add(S)$ and $\X$ in $\Rep(Q,\C).$ Notice that $(\Add(F))^\perp =F^\perp,$ for any $F\in\Rep(Q,\C)$ since $\C$ is $AB4$  \cite[Lem. 3.18]{AM23}. Let us show, firstly, that $\mathbb{T}=\Add(S)\cap\X.$ Indeed, for $x,y\in Q_0$ and a set $I,$ by using that $Q(-,x)$ is finite,  we have that $(g_x(T^{(I)}))_y=(T^{(I)})^{Q(y,x)}=(T^{Q(y,x)})^{(I)})=((g_x(T))^{(I)})_y$ and thus $g_x(T^{(I)})\simeq (g_x(T))^{(I)}.$ Therefore $\mathbb{T}=\add(g_*(\T))\cap\X=\Add(g_*(\T))\cap\X=\Add(S)\cap\X.$
\

{\bf (T1):} It follows as in the proof of (T1) in (c).
 \ 
 
 {\bf (T2):} It follows as in the proof of (T2) in (c) since $\mathbb{T}=\Add(S)\cap\X$ and $\Add(S)^\perp=S^\perp.$
 \
 
 {\bf (T3):} Notice that $\mathbb{T}^\vee\subseteq \Add(S)^{\vee}_\X$ since $\X\subseteq\Rep(Q,\C)$ is a full abelian subcategory closed under quotients. Thus (T3) follows from (T3) in (b).
  \
 
 {\bf (T4):} Since $\Add(S)^\perp=S^\perp=\bigcap_{x\in Q_0}(g_x(T))^\perp,$ we have that the $\X$-injective cogenerator from  (T4) (c) belongs to  $\Add(S)^\perp.$
 \
 
 {\bf (T5):} It follows from (a).
 
\end{proof}
\begin{example}
Let $Q$ be a finite-cone-shape quiver and $\C$ be an abelian category
with enough projectives and injectives. Observe that in this case $\Proj(\C)$ is
an $0$-$\C$-tilting class in $\C$. Therefore, by Theorem \ref{thm:tilting repf},
we have that $\add\left(g_{*}(\Proj(\C))\right)$ is an $1$-$\Rep^{f}(Q,\C)$-tilting
class in $\Rep^{f}(Q,\C)$. Notice that this was proved in \cite[Prop. 3.9]{bauer2020cotorsion}
for the case when $Q$ is finite and acyclic. 
\end{example}
\begin{example}
Let $k$ be a field, $Q$ be a finite acyclic quiver and $\C:=\modu(k)$ 
the category of finite dimensional $k$-vector spaces. In
this case $\Rep^{f}(Q,\C)=\Rep(Q,\C)$ and hence, by the example
above, we have that $\mathcal{T}:=\add\left(g_{*}(\C)\right)$
is an $1$-$\Rep(Q,\C)$-tilting class in $\Rep(Q,\C)$. Moreover, we have that $g_{x}(k)$ is 
the injective representation at the vertex $x\in Q_0$ and  $\mathcal{T}=\add(T),$ where 
 $T:=\bigoplus_{x\in Q_{0}}g_{x}(k)$ is a small $1$-$\Rep(Q,\C)$-tilting
object. Furthermore, since $\Rep(Q,\C)\cong\modu(kQ)$, it follows
from Proposition 4.7 that the module corresponding to $T$ is a Miyashita
$1$-tilting $kQ$-module. 
\end{example}
\begin{example}
Let $k$ be a field, $Q$ and $S$ be finite acyclic quivers and
$\C:=\modu(k)$ be the category of finite dimensional $k$-vector spaces. By \cite[Lem. 1.3]{leszczynski1994representation}, we know that  $\Rep(Q,\Rep(S,\C))\simeq\Mod(kQ\otimes_{k}kS)$ and thus, by proceeding as in the example above, we can conclude that the module
corresponding to $\bigoplus_{y\in Q_{0}}\bigoplus_{x\in S_{0}}g_{y}^{Q}g_{x}^{S}(k)$
is a Miyashita $2$-tilting $\left(kQ\otimes_{k}kS\right)$-module.
Notice  that, for a finite set ${Q^1,\cdots ,Q^n}$ of finite acyclic quivers, 
we can repeat these arguments recursively to get a Miyashita
$n$-tilting $\left(kQ^{1}\otimes_{k}\cdots\otimes_{k}kQ^{n}\right)$-module. 
\end{example}

Our next goal is to give a description of the cotorsion pair induced
by a tilting class which is built in the previous theorem. For this, we recall
the following definitions from \cite{holm2019cotorsion}. Let $Q$
be a quiver, $\mathcal{C}$ be an abelian category and $\A\subseteq\mathcal{C}$.
The following notation is convenient: $\Rep(Q,\A):=\{F\in\Rep(Q,\C)\,|\:F_{x}\in\A\text{ for all }x\in Q_{0}\}$.
For $x\in Q_{0}$, define $Q_{1}^{x\rightarrow*}$ as the set of arrows
starting at $x$. Now, for $F\in\Rep(Q,\C)$ and $x\in Q_{0}$, define
$\psi_{x}^{F}:F_{x}\rightarrow\prod_{\alpha\in Q_{1}^{x\rightarrow*}}F_{t(\alpha)}$
as the morphism induced by the universal property of products through
the family of morphisms $\{F(\alpha)\,|\:\alpha\in Q^{x\rightarrow*}\}$.
Lastly, define $\Psi(\A)$ as the class of the representations
$F\in\Rep(Q,\C)$ such that $\psi_{x}^{F}$ is an epimorphism and
$\Ker(\psi_{x}^{F})\in\A$ for all $x\in Q_{0},$ see also in \cite{AM23} for more details. 

For the full abelian subcategory $\X:=\Rep^{f}(Q,\C)\subseteq\Rep(Q,\C)$ and a class 
$\Z\subseteq\X,$ we consider the following orthogonal classes $\Z^{\perp_\X}:=\{F\in\X\;:\;\Ext^{k}_\X(-,F)|_\Z=0\;\forall\,k\geq 1\}$ and  $\Z^{\perp}:=\{F\in\Rep(Q,\C)\;:\;\Ext^{k}_{\Rep(Q,\C)}(-,F)|_\Z=0\;\forall\,k\geq 1\}.$ In general, we only have that 
$\Z^{\perp}\cap\X\subseteq \Z^{\perp_\X}.$ Dually, we have the classes ${}^{\perp_\X}
\Z$ and ${}^{\perp}\Z.$

The second central theorem of this section tells us how to construct cotorsion pairs in the category of representations from tilting classes in $\C.$

\begin{thm}\label{Rep-tilt-pair}
Let $Q$ be a finite-cone-shape quiver, $\mathcal{C}$ be an abelian
category with enough projectives and injectives, $\mathcal{T}=\add(\mathcal{T})$ be a precovering and $n$-$\C$-tilting class in $\C,$ and let $\X:=\Rep^{f}(Q,\C)\subseteq\Rep(Q,\C).$ Then $\mathbb{T}:=\add\left(g_{*}(\mathcal{T})\right)\subseteq \X$ and the following statements hold true.
\begin{enumerate}
\item $({}^{\bot_\X}(\mathbb{T}^{\bot_\X}),\mathbb{T}^{\bot_\X})$ is a hereditary complete cotorsion pair in the abelian category $\X.$
\item $({}^{\bot}(\mathbb{T}^{\bot}),\mathbb{T}^{\bot})$ is a hereditary complete cotorsion pair in $\Rep(Q,\C).$ Moreover $\mathbb{T}^{\bot}=\Psi(\mathcal{T}^{\bot})$ and $^{\bot}(\mathbb{T}^{\bot})=\Rep(Q,{}^{\bot}(\mathcal{T}^{\bot}))$. 
\item Let $\C$ be $AB3,$  $\T=\Add(T)$ for some $T\in\C$ and $S:=\bigoplus_{x\in Q_0}g_x(T).$ Then $({}^{\bot}(S^{\bot}),S^{\bot})$ is a complete hereditary cotorsion pair in $\Rep(Q,C),$ $S^{\bot}=\Psi(T^{\bot})$ and $^{\bot}(S^{\bot})=\Rep(Q,{}^{\bot}(T^{\bot})).$
\end{enumerate}
\end{thm}
\begin{proof} Observe that the abelian categories $\X$ and $\Rep(Q,\C)$ have enough projectives and injectives since $\C$ is so  \cite[Cor. 5.18 (e,f)]{AM23}. 
\

(a) It follows from  Lemma \ref{tilting is cotorsion} since the abelian category $\X$ has enough projectives and injectives and, by Theorem \ref{thm:tilting repf} (b), we know that $\mathbb{T}$ is $(n+1)$-$\X$-tilting in $\X.$ 
\

(b) By Lemma \ref{tilting is cotorsion}, we know that $\p:=(^{\bot}(\mathcal{T}^{\bot}),\mathcal{T}^{\bot})$ is a complete hereditary cotorsion pair in $\C.$ Then, by \cite[Cor. 5.18 (a,g)]{AM23} it follows that
$(\Rep(Q,\A),\Psi(\B))$ is a complete hereditary cotorsion pair in $\Rep(Q,\C)$ and   $g_{*}(\mathcal{A})^{\bot}=\Psi(\mathcal{B}).$
Moreover,
by \cite[Lem. 4.3]{parte1} and Theorem 3.12 (a), it follows that 
$g_{*}(\mathcal{A})^{\bot}=g_{*}(\mathcal{T}^{\vee})^{\bot}=(g_{*}(\mathcal{T})^{\vee}){}^{\bot}=g_{*}(\mathcal{T})^{\bot}\text{.}$
Therefore, $\mathbb{T}^{\bot}=g_{*}(\T)^{\bot}=g_{*}(\mathcal{A})^{\bot}=\Psi(\mathcal{B})$.
Finally, since $(\Rep(Q,\A),\Psi(\B))$ is a hereditary cotorsion
pair, we have that $^{\bot}(\mathbb{T}^{\bot})={}^{\bot}\Psi(\mathcal{B})=\Rep(Q,\A).$
\

(c) Since $S:=\bigoplus_{x\in Q_0}g_x(T)$ and $\T=\Add(T),$ we get $\Add(S)=\Add(g_*(T))=\Add(g_*(\T))=\Add(\mathbb{T}).$ Therefore $S^\perp=\Add(S)^\perp=\Add(\mathbb{T})^\perp=\mathbb{T}^\perp$ and hence (c) follows from (b).
\end{proof}

\begin{example}
Let $Q$ be a finite-cone-shape quiver, $R$ be a ring, $\C=\Mod(R)$
and $T\in\Mod(R)$ be a big $n$-$\C$-tilting module. Then,  by Theorems \ref{thm:tilting repf} and \ref{Rep-tilt-pair}, we have
that $\mathbb{T}:=\add(g_*(\Add(T)))$
is a small $(n+1)$-$\Rep^{f}(Q,\C)$-tilting class in $\Rep(Q,\C)$ and $S:=\bigoplus_{x\in Q_{0}}g_{x}(T)$ is big $(n+1)$-$\Rep^{f}(Q,\C)$-tilting object in $\Rep(Q,\C).$  
Moreover $({}^{\bot}(S^{\bot}),S^{\bot})$ is a complete hereditary cotorsion pair in $\Rep(Q,C),$ $S^{\bot}=\Psi(T^{\bot})$ and $^{\bot}(S^{\bot})=\Rep(Q,{}^{\bot}(T^{\bot})).$ 
\end{example}

\section*{Acknowledgements}

Part of the research presented in this paper was conducted while the first
named author was on a post-doctoral fellowship at Centro de Ciencias
Matem\'aticas, UNAM Campus Morelia, funded by the DGAPA-UNAM. The
first named author would like to thank all the academic and administrative
staff of this institution for their warm hospitality, and in particular
Dr. Raymundo Bautista (CCM, UNAM) for all his support. 


\bibliographystyle{plain}




\end{document}